\documentclass[
reqno,
10pt,
oneside
]{article}

\usepackage[ngerman, english]{babel}
\usepackage[utf8]{inputenc}
\usepackage{enumerate}
\usepackage[normalem]{ulem}
\usepackage[babel,french=guillemets,german=swiss]{csquotes}
\usepackage[center,font=small]{caption}
\usepackage{cite}
\usepackage{bbm}

\usepackage{tocbasic}
\DeclareTOCStyleEntry[
beforeskip=.2em plus 1pt,
]{tocline}{section}

\usepackage[%
left=1in,
right=1in,
bottom=1in,
top=1in,
footskip=0.25in,
]{geometry}

\usepackage{color}
\definecolor{LinkColor}{rgb}{0,0,1}
\definecolor{LinkColor2}{rgb}{0,0.5,0}
\definecolor{lbcolor}{rgb}{0.85,0.85,0.85}
\definecolor{FrameColor}{rgb}{0.85,0.85,0.85}
\definecolor{rosso}{rgb}{0.8,0,0}
\definecolor{darkgreen}{rgb}{0,0.5,0}

\usepackage{enumitem}

\usepackage{graphicx}
\usepackage{placeins}
\usepackage{overpic}

\usepackage{amsmath}
\usepackage{amssymb}
\usepackage{dsfont}
\usepackage{empheq}
\usepackage{amsthm}
\numberwithin{equation}{section}

\usepackage[%
pdftitle={Titel},%
pdfauthor={Autor},%
pdfcreator={LaTeX, LaTeX with hyperref and KOMA-Script},
pdfsubject={Betreff}, 
pdfkeywords={Keywords}
]{hyperref} 

\hypersetup{%
	colorlinks	=true,
	linkcolor	=LinkColor,%
	anchorcolor	=LinkColor,%
	citecolor	=LinkColor2,%
	filecolor	=LinkColor,%
	menucolor	=LinkColor,%
	urlcolor	=LinkColor,%
}

\newtheorem{theorem}{Theorem}[section]
\newtheorem{lemma}[theorem]{Lemma}

\newtheorem{definition}[theorem]{Definition}

\theoremstyle{definition}
\newtheorem{remark}[theorem]{Remark}

\makeatletter
\renewenvironment{proof}[1][\proofname]{%
	\par\pushQED{\qed}\normalfont%
	\topsep6\p@\@plus6\p@\relax
	\trivlist\item[\hskip\labelsep\bfseries#1\@addpunct{.}]%
	\ignorespaces
}{%
	\popQED\endtrivlist\@endpefalse
}
\makeatother

\makeatletter
\renewcommand\paragraph{\@startsection{paragraph}{4}{\z@}%
	{1ex \@plus1ex \@minus.2ex}%
	{-1em}%
	{\normalfont\normalsize\bfseries}}
\renewcommand\subparagraph{\@startsection{paragraph}{4}{\z@}%
	{1ex \@plus1ex \@minus.2ex}%
	{-1em}%
	{\normalfont\normalsize\itshape}}
\makeatother

\newcommand{\norm}[1]{\ensuremath\left\| #1 \right\|}
\newcommand{\bignorm}[1]{\ensuremath\big\| #1 \big\|}
\newcommand{\abs}[1]{\ensuremath\left|#1 \right|}

\newcommand{\inn}[2]{\ensuremath\left( #1 \hspace{1pt}{,}\hspace{1pt} #2 \right)}
\newcommand{\biginn}[2]{\ensuremath\big( #1 \hspace{1pt}{,}\hspace{1pt} #2 \big)}
\newcommand{\ang}[2]{\ensuremath\left< #1 \hspace{1pt}{,}\hspace{1pt} #2 \right>}

\def\Vp{{(\HH^1)'}}

\newcommand{\R}{\mathbb R}

\newcommand{\n}{\mathbf{n}}
\renewcommand{\v}{{\boldsymbol v}}
\newcommand{\w}{\mathbf {w}}
\newcommand{\z}{\mathbf {z}}

\newcommand{\p}{\mathbf{p}}
\newcommand{\s}{\mathbf{s}}
\newcommand{\m}{\mathbf{m}}

\newcommand{\h}{\mathbbm{h}}
\newcommand{\pp}{\mathbbm{p}}
\newcommand{\f}{\mathbbm{f}}
\newcommand{\g}{\mathbbm{g}}

\newcommand{\ttheta}{{\boldsymbol\vartheta}}
\newcommand{\ttau}{{\boldsymbol\tau}}

\newcommand{\D}{\mathbf{D}}
\newcommand{\I}{\mathbf{I}}
\newcommand{\N}{\mathbf{N}}
\newcommand{\G}{\mathbf{G}}
\renewcommand{\S}{\mathbf{S}}
\newcommand{\T}{\mathbf{T}}

\newcommand{\HH}{\boldsymbol{H}}

\newcommand{\LL}{\boldsymbol{L}}

\newcommand{\WW}{\boldsymbol{W}}
\newcommand{\0}{\boldsymbol{0}}

\newcommand{\CWK}{\mathcal{W}_k}
\newcommand{\CZK}{\mathcal{Z}_k}

\newcommand{\bphi}{\boldsymbol{\varphi}}
\newcommand{\bmu}{\boldsymbol{\mu}}
\newcommand{\bsig}{\boldsymbol{\sigma}}
\newcommand{\bu}{\boldsymbol{u}}

\newcommand{\be}{\boldsymbol{\eta}}
\newcommand{\bz}{\boldsymbol{\zeta}}
\newcommand{\bt}{\boldsymbol{\theta}}
\newcommand{\bx}{\boldsymbol{\xi}}

\newcommand{\Ns}{\mathbf{N}_{\bsig}}
\newcommand{\Np}{\mathbf{N}_{\bphi}}
\newcommand{\Nss}{\mathbf{N}_{\bsig\bsig}}
\newcommand{\Nsp}{\mathbf{N}_{\bsig\bphi}}
\newcommand{\Npp}{\mathbf{N}_{\bphi\bphi}}

\newcommand{\Gs}{\mathbf{G}_{\bsig}}
\newcommand{\Gp}{\mathbf{G}_{\bphi}}
\newcommand{\Gss}{\mathbf{G}_{\bsig\bsig}}
\newcommand{\Gsp}{\mathbf{G}_{\bsig\bphi}}
\newcommand{\Gpp}{\mathbf{G}_{\bphi\bphi}}

\newcommand{\Ss}{\mathbf{S}_{\bsig}}
\newcommand{\Sp}{\mathbf{S}_{\bphi}}
\newcommand{\SG}{\mathbf{S}_{\Gamma}}
\newcommand{\Sv}{S_{\v}}

\newcommand{\Th}{\boldsymbol{\theta}}
\newcommand{\Lam}{\boldsymbol{\Lambda}}

\newcommand{\CC}{\mathbb{C}}
\newcommand{\DD}{\mathbb{D}}

\newcommand{\intO}{\int_\Omega}
\newcommand{\intQ}{\int_{Q}}
\newcommand{\intG}{\int_\Gamma}
\newcommand{\intS}{\int_\Sigma}

\newcommand{\eps}{\varepsilon}

\newcommand{\dx}{\;\mathrm dx}

\newcommand{\dt}{\;\mathrm dt}
\newcommand{\ds}{\;\mathrm ds}
\newcommand{\dxt}{\;\mathrm d(x,t)}

\newcommand{\dS}{\;\mathrm dS}

\newcommand{\ddt}{\frac{\mathrm d}{\mathrm dt}}

\newcommand{\del}{\partial}
\newcommand{\delt}{\partial_{t}}

\newcommand{\deln}{\partial_\n}

\newcommand{\Grad}{\nabla}
\newcommand{\Lap}{\Delta}
\newcommand{\Div}{\textnormal{div}}

\newcommand{\emb}{\hookrightarrow}

\newcommand{\suchthat}{\;\ifnum\currentgrouptype=16 \middle\fi|\;}

\def\Lip{Lip\-schitz}
\def\Hol{H\"older}

\def\rhs{right-hand side}

\def\<#1>{\mathopen\langle #1\mathclose\rangle}
\def\genspazio #1#2#3#4#5{#1^{#2}(#5,#4;#3)}
\def\spazio #1#2#3{\genspazio {#1}{#2}{#3}T0}

\def\spaziok #1#2#3{\genspazio {#1}{#2}{#3}{T_k}0}

\def\L {\spazio L}
\def\H {\spazio H}
\def\W {\spazio W}

\def\Lk {\spaziok L}
\def\Hk {\spaziok H}
\def\Wk {\spaziok W}

\begin{document}

%
%

\title{\sc
	Existence of weak solutions to multiphase Cahn--Hilliard--Darcy and Cahn--Hilliard--Brinkman models for stratified tumor growth with chemotaxis and general source terms}

\author{Patrik Knopf \footnotemark[1] 
		\and Andrea Signori \footnotemark[2]}

\date{ }

\renewcommand{\thefootnote}{\fnsymbol{footnote}}

\footnotetext[1]{
    Department of Mathematics, 
    University of Regensburg, 
    93053 Regensburg, 
    Germany \newline
	\tt(\href{mailto:patrik.knopf@mathematik.uni-regensburg.de}{patrik.knopf@mathematik.uni-regensburg.de}).
}
	
\footnotetext[2]{
    Dipartimento di Matematica ``F. Casorati'',
    Universit\`a di Pavia,
    via Ferrata 5, 27100 Pavia,
    Italy\newline
	\tt(\href{mailto:andrea.signori01@unipv.it}{andrea.signori01@unipv.it}).
	}

\maketitle

\begin{center}
	\small
	{
		\textit{This is a preprint version of the paper. Please cite as:} \\  
		P.~Knopf, A.~Signori, \textit{Communications in Partial Differential Equations} \textbf{47} (2022), no. 2, 233–278.  \\ 
		\url{https://doi.org/10.1080/03605302.2021.1966803}
	}
\end{center}

\smallskip

%
%

\begin{abstract}
We investigate a multiphase Cahn--Hilliard model for tumor growth with general source terms. The multiphase approach allows us to consider multiple cell types and multiple chemical species (oxygen and/or nutrients) that are consumed by the tumor. Compared to classical two-phase tumor growth models, the multiphase model can be used to describe a stratified tumor exhibiting several layers of tissue (e.g., proliferating, quiescent and necrotic tissue) more precisely.
Our model consists of a convective Cahn--Hilliard type equation to describe the tumor evolution, a velocity equation for the associated volume-averaged velocity field, and a convective reaction-diffusion type equation to describe the density of the chemical species. The velocity equation is either represented by Darcy's law or by the Brinkman equation.
We first construct a global weak solution of the multiphase Cahn--Hilliard--Brinkman model.
After that, we show that such weak solutions of this system converge to a weak solution of the multiphase Cahn--Hilliard--Darcy system as the viscosities tend to zero in some suitable sense. This means that the existence of a global weak solution to the Cahn--Hilliard--Darcy system is also established.


\end{abstract}



\setlength\parindent{0ex}
\setlength\parskip{1ex}
\allowdisplaybreaks

%
%

\section{Introduction}

The growth of cancer cells is affected by many biological and chemical mechanisms. 
Although there already exists a large amount of experimental data resulting from clinical experiments, the possibilities of predicting tumor growth are still in great need of improvement.
In particular, it is crucial to gain a better understanding of the underlying biological mechanisms such as proliferation, chemotaxis and necrosis. 

In the recent past, several mathematical models for tumor growth have been developed and analyzed from many different viewpoints. Especially diffuse interface models have gained a lot of interest (see, e.g., \cite{FrigeriGrasselliRocca,HilhorstKampmannNguyenZee,OdenTinsleyHawkins,GarckeLamNuernbergSitka}) and, at least for some of them, it could already be shown that they compare very well with clinical data (cf.~\cite{AgostiEtAl,ACGH,BearerEtAl,FrieboesEtAl}). Therefore, such models might provide further insights into tumor growth dynamics, especially to understand its key mechanisms and to develop patient-specific treatment strategies.

Many of these diffuse interface models for tumor growth consist of a \emph{Cahn--Hilliard equation} with additional source terms to describe the tumor, coupled to a \emph{reaction-diffusion type equation} to describe chemical substances which are consumed by the tumor (usually oxygen and/or nutrients).
Most of these models are two-component phase field models, meaning that only two types of cells, namely tumor cells and healthy cells, are considered. We refer to
\cite{GarckeLam2,ColliGilardiHilhorst, FrigeriGrasselliRocca, frig-lam-roc, FLS, garcke-lam-signori, SS, col-gil-roc-spr, col-gil-roc-spr2} for the analysis of such models, and to \cite{GLS_OPT, signori1, signori2, signori3, signori4, signori5, RSS, CSS1, CSS2, Cavaterra, ColliGilardiRoccaSprekels, GarckeLamRocca, KahleLam, ST} for the investigation of associated optimal control problems.

It is further known that biological materials usually exhibit viscoelastic properties. For that reason, it was suggested in several works in the literature to include an additional velocity equation in tumor growth models to describe such effects. In some papers, the \emph{Stokes equation} was employed to describe the tumor as a viscous fluid (see, e.g., \cite{ByrneKing,Chaplain,FranksKing,Friedman,Friedman2}). 
In other works, \emph{Darcy's law}, which is usually used to describe a viscous flow permeating a porous medium, was chosen instead (cf.~\cite{ByrneChaplain,FranksKing2,Greenspan}). In general, in the context of tumor growth models, both descriptions are a reasonable choice as the Reynolds number associated with the biological tissues is very small. The decision between Stokes and Darcy depends on the concrete situation that is to be described.
However, from the viewpoint of mathematical analysis, Darcy's law is often more difficult to handle because no derivatives of the velocity field, which could be used to obtain additional regularity, are involved in the equation. 
In recent times, \emph{Brinkman's equation} has also become a popular option (cf.~\cite{Donatelli,Perthame,Srinivasan,Zheng}) as it interpolates between the Stokes type and the Darcy type description. 

The Cahn--Hilliard equation coupled to Darcy's law is sometimes also referred to as the Cahn--Hilliard--Hele--Shaw system (especially in the context of two-phase flows). We refer, for example, to \cite{Dede,Giorgini,Feng} for its mathematical investigation. 
The Cahn--Hilliard--Brinkman system was investigated, for instance, in \cite{bosiaconti,contigiorg}.
A two-component Cahn--Hilliard--Brinkman model for tumor growth (including a reaction-diffusion type equation to describe the nutrient density) was proposed and analyzed in \cite{EbenbeckGarcke}. A simplified variant of this model was studied in \cite{EbenbeckGarcke2,EbenbeckGarcke3,EbenbeckLam,eben-knopf1,eben-knopf2}.

Although such two-cell-species Cahn--Hilliard type models are very viable when describing the growth of a young tumor whose evolution is mainly governed by proliferation, they are somewhat limited when processes such as \emph{necrosis} (cf.~\cite{GarckeLamNuernbergSitka}) or \emph{hypoxia} (that is an undersupply of oxygen, cf.~\cite{Byrne}) of tumor cells have already taken place. Indeed, as illustrated in Figure~\ref{fig:tumor}, larger and more mature tumors tend to become \emph{stratified} (cf.~\cite{Roose,Wallace,Sheratt}), meaning that the tumor tissue consists of several layers where each of them exhibits different properties. Indeed, spectroscopic imaging and mapping techniques (see, e.g., \cite{Zhang-IB}) suggest that in many situations, a tumor consists of three layers: a quickly proliferating outer rim, an intermediate quiescent layer whose cells suffer from hypoxia, and a necrotic core whose cells have already died off. For a more detailed discussion, we refer the reader to Section~\ref{SEC:CONC}.

For these reasons, several multiphase models, which allow to describe multiple types of cell species and nutrients, have already been introduced in the literature. We refer the reader to \cite{AstaninPreziosi, araujo, Frieb, sciume, wise2008three, GarckeLamNuernbergSitka,Matioc,fritz} and the references therein. 

\paragraph{A multiphase Cahn--Hilliard model for tumor growth.}
In this paper, we combine the ideas of  \cite{GarckeLamNuernbergSitka} and \cite{EbenbeckGarcke}, and we consider the following \emph{multiphase Cahn--Hilliard model} for tumor growth:
\begin{subequations}
	\label{MCHB}
\begin{alignat}{2}
	\label{MCHB:1}
	\Div(\v) 
	&= \Sv(\bphi,\bsig)
		&&\quad \text{in $Q$},\\ 
	\label{MCHB:2}
	 \Div\big(\T(\bphi,\v,p)\big) + \nu\v 
	&= (\Grad\bphi)^\top \bmu  + (\Grad\bsig)^\top \Ns(\bphi,\bsig) 
		&&\quad \text{in $Q$},\\
	\label{MCHB:3}
	\delt\bphi + \Div(\bphi\otimes\v) 
	&= \Div\big( \CC(\bphi,\bsig) \Grad\bmu \big) 
	    + \Sp(\bphi,\bsig,\bmu)
		&&\quad \text{in $Q$},\\
	\label{MCHB:4}
	\bmu &= -\gamma\eps \Lap\bphi + \gamma\eps^{-1} \Psi_{\bphi}(\bphi) 
	    + \Np(\bphi,\bsig)
		&&\quad \text{in $Q$},\\
	\label{MCHB:5}
	\delt\bsig + \Div(\bsig\otimes\v) 
	&= \Div\big( \DD(\bphi,\bsig) \Grad\Ns(\bphi,\bsig) \big) 
	    - \Ss (\bphi,\bsig,\bmu)
		&&\quad \text{in $Q$},\\[1.5ex]
	\label{MCHB:6}
	\deln\bphi 
	&= \0
		&&\quad \text{on $\Sigma$},\\
	\label{MCHB:7}
\deln\bmu 
		&= \0
		&&\quad \text{on $\Sigma$},\\
	\label{MCHB:8}
	\DD(\bphi,\bsig) \Grad\N_{\bsig}(\bphi,\bsig) \n
	&= \SG(\bphi,\bsig) 
		&&\quad \text{on $\Sigma$},\\
	\label{MCHB:9}
	\T(\bphi,\v,p)\n &= 0
		&&\quad \text{on $\Sigma$},\\[1.5ex]
	\label{MCHB:10}
	\bphi\vert_{t=0}
	&= \bphi_0
	&&\quad \text{in $\Omega$},\\
	\label{MCHB:11}
	\bsig\vert_{t=0} 
	&= \bsig_0
	&&\quad \text{in $\Omega$}.
\end{alignat}
\end{subequations}

\smallskip
Here, $\Omega\subset\R^d$, with $d\in\{2,3\}$, denotes a bounded, smooth domain with boundary $\Gamma$, and $T>0$ stands for an arbitrary final time. 
The outward unit normal vector of $\Gamma$ is denoted by $\bf n$, $\deln$ denotes the corresponding outward normal derivative, whereas $\otimes$ denotes the standard tensor product between two vectors. 
We further use the notation $Q:=\Omega\times (0,T)$ and $\Sigma:=\Gamma\times (0,T)$ .

In this system of partial differential equations, the following quantities are involved:%
\begin{itemize}[leftmargin=*]
\item The tumor is represented by the vector-valued \emph{phase field function} $\bphi = (\varphi_1,...,\varphi_L)^\top$ (with $L\in\mathbb N$). 
For any $i\in\{1,...,L\}$, the component $\varphi_i$ denotes the volume fraction of the $i$-th tumor cell type. 
The healthy cells are represented by $\varphi_0$, which is defined as
\begin{align*}
    \varphi_0 := 1 - \sum_{i=1}^L \varphi_i
    \quad\text{in $Q$}.
\end{align*} 
This ensures that all volume fractions add up to one, that is
\begin{align}
\label{SUM}
    \sum_{i=0}^{L} \varphi_i = 1
    \quad\text{in $Q$}.
\end{align}
The vector of \emph{chemical potentials} associated with the phase field $\bphi$ is denoted by $\bmu = (\mu_1,...,\mu_L)^\top$. Moreover, $\CC(\cdot,\cdot)$ is the mobility tensor. 
The pair $(\bphi,\bmu)$ is mainly governed by the Cahn--Hilliard type subsystem \eqref{MCHB:3}--\eqref{MCHB:4}.
Here, $\Psi_{\bphi}$ denotes the gradient of a given \emph{multi-well potential} $\Psi$ that is a coercive function which is bounded from below and attains its global minimum at $\mathbf{0}$ and at the unit vectors $\mathbf{e}_i$, $i=1,...,L$.
By this choice, it is energetically favourable (cf.~\eqref{DEF:EN}) for the components $\varphi_i$, $i=0,...,L$ to attain values close to one (i.e., only the $i$-th cell type is present) or close to zero (i.e., the $i$-th cell type is not present) in most parts of the domain $\Omega$. These regions where only one cell type is present are separated by a diffuse interface whose thickness is related to the constant $\eps>0$. Therefore, $\eps$ is usually chosen to be very small. Moreover, the constant $\gamma>0$ is related to the surface tension at the interface.
\item The nutrients are represented by the vector-valued function $\bsig = (\sigma_1,...,\sigma_M)^\top$ (with $M\in\mathbb N$). For any $j\in\{1,...,M\}$, the component $\sigma_j \ge 0$ denotes the density distribution of the $j$-th chemical species. These chemical species are usually oxygen and carbohydrates which are consumed by the tumor cells. The functions $\N_{\bphi}$ and $\N_{\bsig}$ denote the partial derivatives of the \emph{chemical free energy density} $N$ with respect to the $\bphi$ and the $\bsig$ variable, respectively. Moreover, $\DD(\cdot,\cdot)$ denotes the mobility tensor corresponding to $\bsig$.
\item The function $\v = (v_1,...,v_d)$ represents the \emph{volume-averaged velocity field} of the mixture, and $p$ denotes the associated \emph{pressure}. 
The quantity $\nu$ in \eqref{MCHB:2} stands for the permeability and is assumed to be a positive constant.
The symbol $\T(\bphi,\v,p)$ denotes the \emph{viscous stress tensor} which is defined as
\begin{align}
	\label{stress}
	\T(\bphi,\v,p) := 2 \eta(\bphi) \D\v + \lambda(\bphi) \Div(\v) \I - p\I\,,
\end{align}
where
\begin{align}
    \label{symmgrad}
	\D\v := \frac 12 \Big( \Grad\v + (\Grad\v)^\top \Big)
\end{align}
stands for the \emph{symmetrized velocity gradient}. Here, $\eta$ and $\lambda$ are nonnegative functions representing the \emph{shear viscosity} and the \emph{bulk viscosity}, respectively. 

If $\eta$ and $\lambda$ are identically zero, \eqref{MCHB:2} is known as \emph{Darcy's law}, and we refer to the system \eqref{MCHB} as the \emph{multiphase Cahn--Hilliard--Darcy system} (MCHD). 
If $\eta$ and $\lambda$ are positive, \eqref{MCHB:2} is called the \emph{Brinkman equation}, which can be regarded as an interpolation between Darcy's law ($\eta=\lambda=0$ and $\nu>0$) and the Stokes equation ($\eta,\lambda>0$ and $\nu=0$). In this scenario, the system \eqref{MCHB} is referred to as the \emph{multiphase Cahn--Hilliard--Brinkman system} (MCHB).

\item The homogeneous Neumann boundary conditions \eqref{MCHB:6} and \eqref{MCHB:7} are standard choices for Cahn--Hilliard type equations. 
The condition \eqref{MCHB:6} entails that the mass flux over the boundary is zero. 
If the diffuse interface associated with the phase-field $\bphi$ intersects the boundary $\Gamma$, the condition \eqref{MCHB:7} enforces a perfect ninety degree contact angle. However, as we are mainly interested in situations where the tumor is confined in the domain $\Omega$ (i.e., the interface does not intersect the boundary at all), the condition \eqref{MCHB:7} is primarily motivated from the viewpoint of mathematical analysis.
\item  The condition \eqref{MCHB:8} describes the nutrient flux over the boundary which is governed by the source term $\SG$. In particular, if $\SG$ is identically zero, no nutrients can enter or leave the domain over the boundary.

In the Brinkman case ($\eta,\lambda>0$), the condition \eqref{MCHB:9} can be understood as a \emph{``no friction'' boundary condition} on the velocity field.  In contrast to more traditional boundary conditions, \eqref{MCHB:9} allows us to handle general solution dependent source terms $\Sv(\bphi,\bsig)$ in \eqref{MCHB:1}. For instance, the no-slip boundary condition $\v\vert_\Sigma = 0$ or the no-penetration boundary condition $\v\vert_\Sigma\cdot \n = 0$ would enforce the unpleasant compatibility condition $ \int_\Omega \Sv(\bphi,\bsig) \dx = 0$, which is avoided by using the no-friction boundary condition \eqref{MCHB:9}. A further advantage of the no-friction condition is that no boundary contributions of the velocity field appear in the weak formulation of the system \eqref{MCHB}. This is very favorable for the mathematical analysis and also for finite element approximations in the context of numerical methods. 

In the Darcy case ($\eta\equiv 0$ and $\lambda\equiv 0$), the boundary condition \eqref{MCHB:9} degenerates to a homogeneous Dirichlet boundary condition on the pressure, i.e., $p\vert_\Sigma = 0$.

\item The functions $\Sv$, $\Sp$, and $\Ss$ are generic source terms that can be specified depending on the application. In Section~\ref{SECT:CONC:S}, we present a concrete example for a suitable choice of these source terms in a four-cell-species tumor model.
\end{itemize}

Furthermore, it is worth mentioning that the model \eqref{MCHB} is associated with the following free energy (cf. \cite{GarckeLamNuernbergSitka}):
\begin{align}
    \label{DEF:EN}
    E(\bphi,\bsig) 
    = \intO  \gamma\eps^{-1} \Psi(\bphi) 
        + \frac {\gamma \eps}2 \sum_{i=1}^{L}|\nabla \varphi_i|^2 \dx 
      + \intO N(\bphi,\bsig) \dx.
\end{align}
Here, the first integral is referred to as the \emph{Ginzburg--Landau energy}. 
The second contribution is the \emph{chemical free energy}. It is associated with the nutrient density $N$ (cf.~\eqref{DEF:N1}) which is usually assumed to be of the form
\begin{align*}
	N(\bphi,\bsig) = \frac {\chi_{\bsig}}2|\bsig|^2 - G(\bphi,\bsig)
\end{align*}
for a suitable function $G$. A reasonable choice for the function $G$ in a four-cell-species tumor model is presented in Section~\ref{SECT:CONC:N}. 

In the absence of source terms (i.e., $\Sv\equiv 0$, $\Sp\equiv \mathbf{0}$, $\Ss\equiv \mathbf{0}$, and $\SG\equiv\mathbf{0}$), we obtain the following energy law:
\begin{align}
\label{EN:LAW}
	\begin{aligned} 
	&\ddt E(\bphi,\bsig)
	+ \intO 2 \eta(\bphi)|\D \v|^2+\nu|\v|^2 \dx
	\\ 
	&\quad
	+ \intO \CC(\bphi,\bsig) \Grad \bmu : \Grad \bmu 
	    +\DD(\bphi,\bsig) \Grad \N_{\bsig}(\bphi,\bsig) : \Grad \N_{\bsig}(\bphi,\bsig) \dx 
    = 0.
    \end{aligned}
\end{align}
If the tensors $\CC$ and $\DD$ are chosen  appropriately (i.e., at least positive semidefinite), then both integrals on the left-hand side are nonnegative. This implies that the energy is decreasing along solutions of the system \eqref{MCHB} over the course of time. Therefore, \eqref{EN:LAW} describes the {dissipation} of the free energy, and in this context, the integrals on the left-hand side of \eqref{EN:LAW} can be understood as the \emph{dissipation rate}.
This means that, at least in the absence of source terms, the model \eqref{MCHB} is \emph{thermodynamically consistent}.

The multiphase Cahn--Hilliard--Darcy model (MCHD) is heavily based on the model derived in \cite{GarckeLamNuernbergSitka}. 
The only difference is that we are using a different right-hand side in the velocity equation \eqref{MCHB:2}, which is of the same type as the one proposed for the two-cell-species scenario in \cite{EbenbeckGarcke}. We point out that this choice plays a crucial role in the derivation of the energy dissipation law \eqref{EN:LAW} and thus, it also provides some advantages for the mathematical analysis.

The Darcy type description is particularly suitable if the viscoelastic flow associated with the biological tissues is assumed to behave like a viscous fluid permeating a porous medium. Although there are some situations where this assumption is justified, this is not always the case. 
Therefore, the multiphase Cahn--Hilliard--Brinkman model (MCHB) might sometimes provide a better description. 
At least formally, the model (MCHB) converges to the model (MCHD) as the viscosities $\eta$ and $\lambda$ tend to zero. We will show that this asymptotic limit can be rigorously verified on the level of weak solutions. 

In \cite{GarckeLamNuernbergSitka}, also several numerical simulations for the system (MCHD) (with slightly different boundary conditions) were presented. In the case $L=M=1$, where only two cell species (namely tumor cells and healthy cells) and one nutrient species are considered, the existence of weak solutions to the Cahn--Hilliard--Darcy model with a special choice of source terms and slightly different boundary conditions compared to \eqref{MCHB} was established in \cite{GarckeLam1}.
For further mathematical investigations related to the two-cell-species Cahn--Hilliard--Darcy system we refer the reader to \cite{GarckeLam1,GarckeLam3,GarckeLam4,GarckeLamSitkaStyles} and the references therein.

Existence results for solutions to multiphase Cahn--Hilliard--Darcy systems for tumor growth which are related to \eqref{MCHB} can be found in \cite{DFRSS, FLRS}.
Although the models studied in \cite{DFRSS} and \cite{FLRS} allow for more general potentials $\Psi$ (including singular potentials like the logarithmic Flory--Huggins potential) which definitely makes the construction of solutions more challenging, they can at least to some extend be understood as a simplified variant of the model \eqref{MCHB}.
For instance, the systems investigated in \cite{DFRSS, FLRS} are limited to three cell species ($L=2$) and one nutrient species ($M=1$), the mobility tensors $\CC(\cdot,\cdot)$ and $\DD(\cdot,\cdot)$ are constant diagonal matrices, the nutrient equation is quasi-stationary, and chemotaxis mechanisms are neglected.

In the special case $L=1$ and $M=1$, the Cahn--Hilliard--Brinkman model (MCHB) was introduced in \cite{EbenbeckGarcke}, where also the existence of weak solutions was established. A numerical investigation can be found in \cite{EbenbeckGarcke3}. In \cite{EbenbeckGarcke2}, a simplified version of this model was investigated, where the time derivative and the convection term in the nutrient equation are neglected. This means that the simplified nutrient equation is a quasi-static elliptic equation. For this model, the authors proved strong well-posedness and showed that the solutions converge to the corresponding Cahn--Hilliard--Darcy model as the viscosities $\eta$ and $\lambda$ tend to zero. For the analysis of weak and stationary solutions of this system with singular potentials, we refer to \cite{EbenbeckLam}. In \cite{eben-knopf1,eben-knopf2}, optimal control problems for this simplified model were investigated. We further want to mention \cite{Dharmatti}, where the optimal control of a nonlocal Cahn--Hilliard--Brinkman model (without nutrient equation) was studied.

\paragraph{Structure of this paper.} The paper is structured as follows. In Section~\ref{SEC:CONC}, based on the general multiphase Cahn--Hilliard model \eqref{MCHB}, we present a concrete example for a four-cell-species tumor model ($L=3$) with one species of nutrient ($M=1$). In particular, we describe how the source terms and the chemical free energy density can be chosen (in accordance with the mathematical analysis) to describe biologically relevant mechanisms. In Section~\ref{SEC:MA}, we first fix some notation, recall auxiliary results and introduce assumptions that are necessary for the mathematical analysis. After that, we present the main results of our paper. The existence of a weak solution to 
(MCHB) is established in Theorem~\ref{THM:EXISTENCE:WEAK}. In Theorem~\ref{THM:EXISTENCE:WEAK:DARCY}, we show that the weak solutions of the system (MCHB) constructed in Theorem~\ref{THM:EXISTENCE:WEAK} converge to a weak solution to the 
system (MCHD) as the viscosities $\eta$ and $\lambda$ tend to zero in a suitable sense. 
This is indeed a novel result since even in the two-cell-species model presented in \cite{EbenbeckGarcke}, this asymptotic has not been investigated. In particular, this proves the existence of weak solutions to the model (MCHD). 
We point out that this ``Darcy limit'' was also established rigorously in \cite{EbenbeckGarcke2} for strong solutions to a related two-cell-species model with a simplified quasi-stationary nutrient equation.
 
The proof of Theorem~\ref{THM:EXISTENCE:WEAK} is given in Section~\ref{SECT:EX:MCHB}, whereas the proof of Theorem~\ref{THM:EXISTENCE:WEAK:DARCY} is presented in Section~\ref{SECT:EX:MCHD}.

\section{A concrete tumor model with four cell species}
\label{SEC:CONC}
In order to describe a stratified tumor by the system \eqref{MCHB}, we now suggest explicit choices for the source terms $\Sv$, $\Sp$, $\Ss$, and $\SG$ as well as the chemical free energy density $N$. As often suggested in the literature (see, e.g., \cite{Roose,Wallace,Sheratt} and the references therein), we assume that the tumor exibits three layers: 
\begin{itemize}
    \item a \emph{proliferating rim} whose cells consume nutrients and oxygen to proliferate rapidly,
    \item an intermediate \emph{quiescent region} whose cells do not proliferate any more as they suffer from hypoxia (lack of oxygen) and/or an undersupply of nutrients,
    \item and a \emph{necrotic core} whose cells have already died due to the lack of oxygen and nutrients.
\end{itemize}
An illustration of such a stratified tumor can be found in Figure~\ref{fig:tumor}. In the mathematical model, we thus choose $L=3$ to describe the three tumor layers as well as the healthy cells. 
The proliferating tumor cells are associated with $\varphi_1$, the component $\varphi_2$ stands for the quiescent tissue, whereas $\varphi_3$ corresponds to the necrotic region. The volume fraction of the healthy cells is thus given as
\begin{align*}
    \varphi_0 = 1 - \sum_{i=1}^3 \varphi_i\;.
\end{align*}
For simplicity, we restrict ourselves to consider oxygen and nutrients as one single chemical species, meaning that $M=1$. Therefore, the nutrient density is a scalar function; to emphasize this, we will thus write $\sigma$ instead of $\bsig$. 

\begin{figure}[ht!]
\bigskip\centering
\includegraphics[width=0.7\textwidth, tics=10]{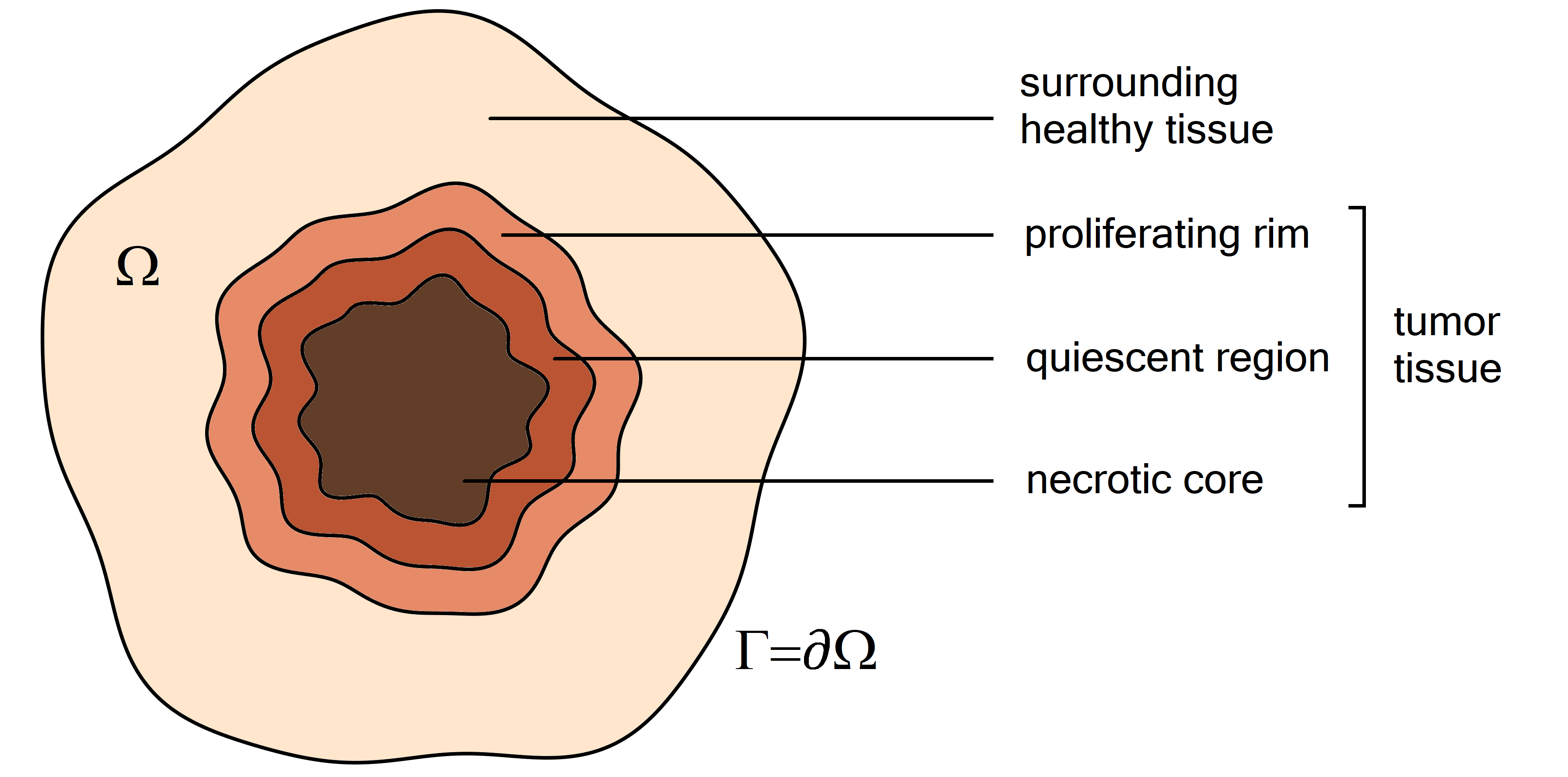}
\caption{Schematic representation of the layers of a stratified tumor ($L=3$).}
\label{fig:tumor}
\end{figure}


\subsection{The source terms}\label{SECT:CONC:S}

We first present some explicit choices for the source terms. We assume that $\Sv$, $\Sp$, and $\Ss$ depend only on $\bphi$ and $\sigma$ but not on $\bmu$. Thus, with some abuse of notation, we write
\begin{align*}
    \Sp(\bphi,\sigma) = \Sp(\bphi,\sigma,\bmu),\quad 
    \Ss(\bphi,\sigma) = \Ss(\bphi,\sigma,\bmu).
\end{align*}
For the source term of the nutrient equation \eqref{MCHB:5}, we make the ansatz
\begin{align}
    \label{DEF:Ss}
    \Ss(\bphi,\sigma) 
    &= \mathcal C \varphi_1 \sigma 
        - \mathcal B (\sigma_\Omega - \sigma).
\end{align}
Here, the term $- \mathcal C \varphi_1 \sigma$ describes the consumption of nutrients by the proliferating cells at a constant rate $\mathcal C > 0$.
Moreover, $\mathcal B$ denotes a positive amplifying constant, and the function $\sigma_\Omega$ stands for a given nutrient concentration provided by preexisting blood vessels permeating the tissue. Hence, the term
$\mathcal B (\sigma_\Omega - \sigma)$ describes supply ($\sigma<\sigma_\Omega$) or deprivation ($\sigma>\sigma_\Omega$) of nutrients by the vasculature.
In a scenario of pure avascular growth, this term can be neglected.

For the source term in the phase field equation, the following choices are reasonable:%
\begin{subequations}
    \label{DEF:SP}
\begin{align}
    \label{DEF:SP:1}
    \Sp(\bphi,\sigma) 
    &= \big(
            \varphi_1(\mathcal P\sigma - \mathcal Q),\,
            \mathcal Q\varphi_1 - \mathcal A \varphi_2,\,
            \mathcal A\varphi_2 - \mathcal D \varphi_3
        \big)^\top,
    \\
    \label{DEF:SP:2}
    \Sp(\bphi,\sigma) 
    &= \big(
            \tfrac{1}{\eps}\,\pp(\varphi_1)
                (\mathcal P\sigma - \mathcal Q),\,
            \tfrac{1}{\eps}\, \pp(\varphi_2)
                (\mathcal Q - \mathcal A),\,
            \tfrac{1}{\eps}\, \pp(\varphi_3)
                (\mathcal A - \mathcal D)
        \big)^\top.
\end{align}
\end{subequations}
We point out that similar choices were discussed in \cite{GarckeLamNuernbergSitka} for a three-cell-species tumor model ($L=2$) neglecting the quiescent region.
In \eqref{DEF:SP}, the positive constants $\mathcal P$, $\mathcal Q$, $\mathcal A$ and $\mathcal D$ denote the \emph{proliferation rate}, the \emph{quiescence rate}, the \emph{apoptosis rate}, and the \emph{degradation rate}, respectively. Moreover, $\pp$ stands for the polynomial $\pp(s)=s^2(1-s^2)^2$, $s\in\R$.

The option \eqref{DEF:SP:1} models the increase of proliferating tumor cells at the rate $\mathcal P$. The proliferating cells become quiescent at the rate $\mathcal Q$ which in turn means that the quiescent cells increase at the rate $\mathcal Q$. Similarly, due to apoptosis, the quiescent cells decrease and the necrotic cells increase at the rate $\mathcal A$. Eventually, the necrotic cells degrade at the rate $\mathcal D$.

The additional idea in \eqref{DEF:SP:2} is that the expressions $\pp(\varphi_i)$, $i=1,2,3$, are positive at the diffuse interface (i.e., in $(0,1)$) but zero at the values corresponding to the regions where only one cell type is present (i.e., in $\{0,1\}$). This means that the evolution of the interface is directly influenced by the source terms. The scaling factor $\frac 1 \eps$ is chosen as in \cite{GarckeLamNuernbergSitka,HilhorstKampmannNguyenZee} in order to retain the possibility of passing to the (formal) sharp interface limit $\eps\to 0$.

Furthermore, as shown in \cite{GarckeLamNuernbergSitka}, the property
\begin{align*}
    \sum_{i=0}^L \varphi_i = \varphi_0 + {\mathbf{1}}\cdot \bphi  = 1,
\end{align*}
where ${\mathbf{1}} = (1,...,1)^\top\in\R^L$, entails that the source term $\Sv$ needs to be chosen as
\begin{align}
    \label{DEF:SV*}
    \Sv(\bphi,\sigma) = \mathbf{1}\cdot \Sp(\bphi,\sigma) + S_{\varphi_0}(\bphi,\sigma),
\end{align}
where $S_{\varphi_0}(\bphi,\sigma)$ is the source term associated with the healthy tissue described by $\varphi_0$. 

For instance, if $\Sp$ is chosen as suggested in \eqref{DEF:SP:1}, a reasonable choice is 
\begin{align}
    \label{DEF:SP0}
    S_{\varphi_0}(\bphi,\sigma) &= - \kappa \, \mathcal P \sigma \, \varphi_1 
    \quad\text{for some $\kappa\in[0,1]$}.
\end{align}
In the case $\kappa=1$, the mass gain of tumor cells equals the mass loss of healthy cells. This would be the case if all newly emerged tumor cells originate from corrupted healthy cells. If $\kappa=0$, the formation of tumor cells does not mean any loss of healthy cells, whereas the choice $\kappa\in(0,1)$ interpolates between these rather extreme scenarios.

If $\Sp$ is given by \eqref{DEF:SP:2}, we recommend to choose $S_{\varphi_0}(\bphi,\sigma) = 0$ as proposed in \cite{GarckeLamNuernbergSitka}.

Although the options \eqref{DEF:SP}, \eqref{DEF:SV*} and \eqref{DEF:SP0} make sense from the modeling perspective, they do not fulfill the assumptions \ref{ass:sources} and \ref{ass:Sv} we have to make in Section~\ref{SECT:ASS} for the mathematical analysis.
Namely, we require that
\begin{align*}
    \abs{\Sv(\bphi,\sigma)} \le A,\qquad
    \abs{\Sp(\bphi,\sigma)} \le B (|\bphi| + |\sigma| + 1)
\end{align*}
for constants $A,B>0$ depending neither on $\bphi$ nor on $\sigma$. 

To overcome this issue, we replace the term $\mathcal P\sigma$ in \eqref{DEF:SP} by a bounded expression. We assume that there exists a critical nutrient concentration $c_p>1$ such that the proliferation does not increase any more, even if a larger amount of nutrient ($\sigma>c_p$) is available. Therefore, we introduce a nondecreasing function $P\in C^1_b(\R)$ which satisfies
\begin{align}
    \label{COND:P}
    \begin{cases}
        P(s) = \mathcal P s 
        &\text{for all  $s \in [0,c_p-1] $},\\
        P(s) = \mathcal Pc_p 
        &\text{for all $s \in [c_p,\infty)$}.
    \end{cases}
\end{align}

Moreover, for any fixed $r>0$, we introduce a truncation function $\h_r\in C^1_b(\R)$ which satisfies 
\begin{align*}
    \h_r(s) = s \quad
    \text{for all $s \in [-r,1+r]$.}
\end{align*}
If the multi-well potential $\Psi$ is reasonably chosen and $r>0$ is not too small, the values of the components $\varphi_i$ will not exceed the interval $[-r,1+r]$. Choosing $r=1$ should usually be more than enough to ensure this condition. In this case, replacing $\varphi_i$ by $\h_r(\varphi_i)$ does not have any effect on the solution of the system \eqref{MCHB}. We thus choose
\begin{subequations}
\label{DEF:SPM}
\begin{align}
    \begin{aligned}
        \label{DEF:SPM:1}
        \Sp(\bphi,\sigma) 
        &= \big(
                \h_r(\varphi_{1})\, P(\sigma) - \mathcal Q \varphi_1,\,
                \mathcal Q\varphi_1 - \mathcal A \varphi_2,\,
                \mathcal A\varphi_2 - \mathcal D \h_r(\varphi_{3})
            \big)^\top,\\
        S_{\varphi_0}(\bphi,\sigma) 
        &= -\kappa P(\sigma) \h_r(\varphi_{1})\, \qquad\text{for some $\kappa\in [0,1]$},
    \end{aligned}
\end{align}
or
\begin{align}
    \label{DEF:SPM:2}
    \begin{aligned}
    \Sp(\bphi,\sigma) 
    &= \big(
            \tfrac{1}{\eps}\, \pp_r(\varphi_1)
                \big(P(\sigma) - \mathcal Q\big),\,
            \tfrac{1}{\eps}\, \pp_r(\varphi_2)
                (\mathcal Q - \mathcal A),\,
            \tfrac{1}{\eps}\, \pp_r(\varphi_3)
                (\mathcal A - \mathcal D)
        \big)^\top,\\
    S_{\varphi_0}(\bphi,\sigma) 
    & = 0,
    \end{aligned}
\end{align}
\end{subequations}
where $\pp_r:=\pp\circ \h_r$ is a bounded function.
It is easily seen that both \eqref{DEF:SPM:1} and \eqref{DEF:SPM:2} satisfy the assumption \ref{ass:sources}. Moreover, if the source term $\Sv$ is chosen as proposed in \eqref{DEF:SV*}, it fulfills the assumption \ref{ass:Sv} for any $\kappa\in[0,1]$. 

For the source term $\SG$ appearing in the boundary condition \eqref{MCHB:8} for the nutrient equation \eqref{MCHB:5}, we assume that it depends only on the nutrient density $\sigma$. With some abuse of notation, we thus write
\begin{align*}
    S_\Gamma (\sigma) = \SG (\bphi, \sigma).
\end{align*}
As suggested in \cite{EbenbeckGarcke,EbenbeckGarcke2}, we propose the choice
\begin{align}
    \label{concrete:SG}
     S_\Gamma (\sigma) = K (\sigma_\Gamma - \sigma),
\end{align}
where $\sigma_\Gamma$ is a given function describing a preexisting nutrient supply over the boundary, and $K$ is a nonnegative permeability constant.
Notice that in the case $K>0$, \eqref{MCHB:8} is a Robin type boundary condition, whereas if $K=0$, it reduces to a no-flux condition. Moreover, the formal asymptotic limit $K \to \infty$ would produce the Dirichlet condition $\sigma = \sigma_\Gamma$ on $\Sigma$. 
This limit has been investigated in \cite{EbenbeckGarcke2} for a simplified two-cell-species version of the system \eqref{MCHB} with a quasi-stationary nutrient equation (with $M=1$).

\subsection{The chemical free energy density}\label{SECT:CONC:N}

For the chemical free energy, we use a similar decomposition of $N$ as proposed in \cite[Sect.~1]{GarckeLamNuernbergSitka}. Namely, we choose 
\begin{align}
    \label{DEF:N1}
	N(\bphi,\sigma) = \frac {\chi_{\bsig}}2|\sigma|^2 - G(\bphi,\sigma),
\end{align}
where the function $G$ is defined as
\begin{align}
    \label{DEF:N2}
     G(\bphi,\sigma) 
    := \chi_{\bphi} \sigma \varphi_1 
        + \f(\sigma) \varphi_2
        + \g(\sigma) \varphi_3.
\end{align}
Here, the term $\chi_{\bphi} \sigma \varphi_1$ describes the \emph{chemotaxis} mechanism which drives the proliferating tumor cells to grow towards regions of high nutrient concentration. 
We further assume that there exist critical nutrient concentrations $0<c_n<c_q<\infty$ and functions $\f, \g\in C^2(\R)$ that satisfy the following conditions:
\begin{align}
    \label{DEF:N3}
    & \begin{cases}
        \text{$\f> 0$ on $(c_n,c_q)$,} \\
        \text{$\f\le 0$ on $[c_q,\infty)$,}\\
    \end{cases}\\[1ex]
    \label{DEF:N4}
    &\begin{cases}
        \text{$\g>0$ on $(-\infty,{c_n})$,} \\
        \text{$\g\le 0$ on $[{c_n},\infty)$.}
    \end{cases}
\end{align}
The reasons behind these choices are the following:
\begin{itemize}
    \item If the nutrient concentration lies between the critical values $c_n$ and $c_q$, we expect the cells to become quiescent due to a lack of nutrient.
    This means that the amount of quiescent cells (that are associated with $\varphi_2$) will increase.
    We describe this behavior by assuming that $\f$ is positive on $(c_n,c_q)$.
    As a consequence the whole term $ \f(\sigma) \varphi_2$ is positive, provided that $\varphi_2$ is positive, and with regard to the energy $E$ presented in \eqref{DEF:EN}, it is thus energetically favorable if $\varphi_2$ further increases. 
    
    \item If the nutrient concentration is below the critical value $c_n$, we expect the cells to necrotize, meaning that the amount of necrotic cells (described by $\varphi_3$) will increase.
    To model this behavior, we assume that $\g>0$ on $(-\infty,c_n)$.
    This entails that the term $ \g(\sigma) \varphi_3$ becomes positive if $\varphi_3$ is positive. It is thus energetically favorable if $\varphi_3$ increases. 
    
\item  We point out that in \eqref{DEF:N3}, the sign of $\f$ on the interval $(-\infty,c_n)$ is not prescribed as this strongly depends on the modeling details. For instance, if $\f\le 0$ on $(-\infty,c_n)$, the cells will not become quiescent as long as the nutrient concentration is below $c_n$. They will rather necrotize due to the term $\g(\sigma)\varphi_3$. On the other hand, if $\f > 0$ on $(-\infty,c_n)$, there is a competition between quiescence and necrosis effects. 
\end{itemize} 
 In the mathematical analysis, we are only able to handle the case where $\f$ and $\g$ are affine functions (see Assumption~\ref{ass:nutrient}). Hence, in accordance with \eqref{DEF:N3} and \eqref{DEF:N4}, the only reasonable choices are
\begin{align*}
    \f(s) := \alpha (c_q - s),
    \quad s\in\R,
    \qquad\text{and}\qquad
    \g(s) := \beta (c_n - s),
    \quad s\in\R,
\end{align*}
with given positive constants $\alpha$ and $\beta$ acting as weights.

\subsection{The tensors $\CC$ and $\DD$} \label{SECT:TENS}
In general, the mobility tensors $\CC$ and $\DD$ can be fourth-order tensors (see \ref{ass:tensors} in Section~\ref{SECT:ASS}) or second-order tensors (i.e., matrices; see Remark~\ref{REM:ASS}(a)). A simple but very common choice is to choose $\CC \in\R^{L\times L}$ and $\DD \in\R^{M\times M}$ as diagonal matrices with uniformly positive functions as diagonal entries. However, it is worth mentioning that under the assumptions in Section~\ref{SECT:ASS}, even more complicated choices would be possible, as long as the tensors are uniformly positive definite.

In the present scenario with $L=3$ and $M=1$, we consider a matrix-valued function $\CC:\R^3\times\R\to\R^{3\times 3}$ and a scalar function $\DD:\R^3\times\R\to\R$. 
We assume that the function $\DD$ is uniformly positive, and that
\begin{align}
    [\CC(\bphi,\sigma)]_{ij} = m_i(\bphi,\sigma) \delta_{ij},\quad  i,j = 1,...,L,
\end{align}
where $m_i:\R^3\times\R\to\R$, $i=1,...,L$ are given, uniformly positive functions and $\delta_{ij}$ stands for the Kronecker symbol.

%
%

\section{Mathematical analysis} \label{SEC:MA}

\subsection{Notation} 
We first fix some notation that will be used throughout the paper.

The natural numbers excluding zero are denoted by $\mathbb N$, whereas the natural numbers including zero are denoted as $\mathbb N_0$.
For any Banach space $X$ we denote its associated norm by $\norm{\,\cdot\,}_X$, and its topological dual space by $X'$. The duality pairing of $X'$ and $X$ is denoted by $\<\cdot,\cdot>_X$. If $X$ is a Hilbert space, we denote its inner product by $(\cdot, \cdot)_X$. Note that for Banach spaces $X$ and $Y$, the intersection
$X\cap Y$ is also a Banach space with respect to the norm $\|\cdot\|_{X \cap Y}:=
\|\cdot\|_X+\|\cdot\|_Y$. 

For any $k\in\mathbb N_0$, $C^k(U)$ stands for the space of $k$-times continuously differentiable functions on any set $U$ for which this definition makes sense. The subspace $C^k_b(U)$ consists of all bounded functions in $C^k(U)$ whose partial derivatives up to the order $k$ are also bounded. Note that $C^k_b(U)$ is a Banach space with respect to its standard norm which is denoted by $\norm{\,\cdot\,}_{C^k_b}$.
Moreover, $C^k_c(U)$ denotes the space of $C^k(U)$-functions that have compact support in $U$. In the case $k=0$, we just write $C(U)=C^0(U)$, $C_b(U)=C_b^0(U)$, and $C_c(U)=C_c^0(U)$.

For any $1 \leq p \leq \infty$ and $k \geq 0$, the standard Lebesgue and Sobolev spaces defined on $\Omega$ are denoted as $L^p(\Omega)$ and $W^{k,p}(\Omega)$, and the corresponding norms are denoted as $\norm{\,\cdot\,}_{L^p(\Omega)}=\norm{\,\cdot\,}_{L^p}$ and $\norm{\,\cdot\,}_{W^{k,p}(\Omega)}=\norm{\,\cdot\,}_{W^{k,p}}$, respectively. 
In the case $p = 2$, these spaces are Hilbert spaces and we use the standard notation $H^k(\Omega) = W^{k,2}(\Omega)$. A similar notation is used for Lebesgue and Sobolev spaces on $\Gamma$, where the norms are denoted by $\norm{\,\cdot\,}_{L^p(\Gamma)}=\norm{\,\cdot\,}_{L^p_\Gamma}$ and $\norm{\,\cdot\,}_{W^{k,p}(\Gamma)}=\norm{\,\cdot\,}_{W^{k,p}_\Gamma}$.

For brevity, we sometimes write $\LL^p$, $\WW^{k,p}$, $\HH^{k}$ and $\LL^p_\Gamma$, $\WW_\Gamma^{k,p}$, and $\HH_\Gamma^{k}$ to denote the corresponding spaces of vector or matrix-valued functions
defined on $\Omega$ and $\Gamma$, respectively.

We further define the Hilbert space
\begin{align*}
    H^2_\n(\Omega;\R^n) := \{ \bz \in H^2(\Omega;\R^n):
        \deln \bz = 0 \;\text{a.e. on $\Gamma$} \}
\end{align*}
for any $n\in\mathbb{N}$.
Moreover, we set
\begin{align*}
    \LL^2_{\Div}(\Omega):= \{ {\bf f} \in L^2(\Omega;\R^d): \Div ({\bf f}) \in L^2(\Omega)\}.
\end{align*}
Here, the relation $\Div ({\bf f}) \in L^2(\Omega)$ means that the divergence exists in the weak sense and belongs to $L^2(\Omega)$. Notice that $\LL^2_{\Div}(\Omega)$ is a Hilbert space when equipped with the inner product
\begin{align*}
    \inn{\bf f}{\bf g}_{\LL^2_{\Div}}^2
    := \inn{\bf f}{\bf g}_{\LL^2}^2
    +\biginn{\Div({\bf f} )}{\Div({\bf g} )}_{L^2}^2
    \quad\text{for all ${\bf f},{\bf g}\in \LL^2_{\Div}(\Omega)$,}
\end{align*}
and its induced norm.
Moreover, we recall that for any ${\bf f}\in \LL^2_{\Div}(\Omega)$, the expression ${\bf f} \cdot {\bf n}$ is well defined on $\Gamma$ by the following integration by parts formula:
\begin{align}
    \label{int:parts}
    \< {\bf f} \cdot {\bf n}, \phi >_{H^{1/2}(\Gamma)}
    = 
    \intO {\bf f} \cdot \nabla \phi \dx
    + \intO \phi \, \Div ({\bf f}) \dx
    \quad \text{for all $\phi \in H^1(\Omega)$;}
\end{align}
see, e.g., \cite[Sect.~III.2]{Galdi}. Moreover, there exists a positive constant $C_{\Div}$, which depends only on $\Omega$, such that
\begin{align*}
    \norm{{\bf f} \cdot {\bf n}}_{(H^{1/2}_{\Gamma})'}
    \leq C_{\Div}\norm{{\bf f}}_{\LL^2_{\Div}}.
\end{align*}

For vectors $\mathbf{a}=(a_1,...,a_k)^\top\in\R^k$ and $\mathbf{b}=(b_1,...,b_l)^\top\in\R^l$, we denote the standard tensor product by $\mathbf{a}\otimes\mathbf{b}$ which produces an element of $\R^{k\times l}$ and is defined componentwise as $(\mathbf{a}\otimes\mathbf{b})_{ij} = a_i b_j$ for all $i\in\{1,...,k\}$, $j\in\{1,...,l\}$. 

For given matrices ${\bf A,B}\in \R^{n\times m }$, we define the scalar product
\begin{align*}
	{\bf A : B} \,:= \sum_{i=1}^n \sum_{j=1}^m [{\bf A}]_{ij}\, [{\bf B}]_{ij} \;.
\end{align*}
Furthermore, for any fourth order tensor $\CC$ in $\R^{n\times m \times n\times m}$, and any matrix ${\bf A}\in \R^{n\times m }$, we set 
\begin{align*}
	[\CC {\bf A}]_{ij} 
	: = \sum_{k=1}^n \sum_{l=1}^m \,[\CC]_{ijkl} \, [{\bf A}]_{kl}
	\quad\text{for all $i\in\{1,...,n\}$ and $j\in\{1,...,m\}$,}
\end{align*} 
and we use the notation
\begin{align*}
    \abs{\mathbf{A}} 
        = \Bigg(\sum_{i=1}^n \sum_{j=1}^m
        \abs{[\mathbf A]_{ij}}^2 \Bigg)^{\frac 12}, 
    \qquad
    \abs{\mathbb{C}} 
        = \Bigg(\sum_{i,k=1}^n \sum_{j,l=1}^m 
        \abs{[\mathbf C]_{ijkl}}^2 \Bigg)^{\frac 12}.
\end{align*}

\subsection{Auxiliary results} 
Before diving into the obtained results, let us first introduce some useful auxiliary results.

The following interpolation result for Sobolev spaces on bounded domains can be found in \cite[Sec~4.3.1, Thm.~1]{triebel}.
\begin{lemma}[Interpolation between Sobolev spaces]
    \label{LEM:INT}
    Suppose that $\Omega\subset\R^d$ with $d \in \mathbb N$ is a bounded smooth domain. Furthermore, let $\theta\in (0,1)$ be arbitrary, and let $r$, $s_0$, and $s_1$ be any real numbers satisfying
    $$r = (1-\theta) s_0 + \theta s_1.$$ 
    Then $H^r(\Omega)$ is the real interpolation of $H^{s_0}(\Omega)$ and $H^{s_1}(\Omega)$ with interpolation parameter $\theta$, that is,
    \begin{align*}
        H^r(\Omega) = \big( H^{s_0}(\Omega) , H^{s_1}(\Omega) \big)_{\theta,2}\;.
    \end{align*}
    In particular, there exists a constant $C>0$ such that for all $f\in H^{s_0}(\Omega) \cap H^{s_1}(\Omega)$,
    \begin{align*}
        \norm{f}_{H^r(\Omega)} 
        \le C \norm{f}_{H^{s_0}(\Omega)}^{1-\theta}
            \norm{f}_{H^{s_1}(\Omega)}^{\theta}.\\
    \end{align*}
\end{lemma}

\medskip

Next, we recall a well-known result related to the solvability of the divergence equation. The lemma and the corresponding proof can be found, e.g., in \cite[Sec. III.3]{Galdi}.
\begin{lemma}
	\label{LEM:diveq}
	Let $\Omega \subset \R^d$, $d \geq 2$, be a bounded domain with \Lip\ boundary $\Gamma$, and let $q \in (1,\infty)$ be arbitrary.
	Then, for every $f \in L^q(\Omega)$ and ${\bf a} \in \WW^{1- \frac 1{q},1}(\Gamma)$ with
	\begin{align}
		\intO f \dx = \intG {\bf a} \cdot {\bf n} \dS,
		\label{cond}
	\end{align}
	there exist a strong solution $\bu \in \WW^{1,q}(\Omega)$ to the problem
	\begin{align*}
		\begin{cases}
		\Div (\bu) = f \quad  &\hbox{in $\Omega,$}
		\\ 
		\bu = {\bf a} \quad & \hbox{on $\Gamma,$ }
		\end{cases}
\end{align*}		
and a positive constant $C_{\Omega,q}$ depending only on $\Omega$ and $q$, such that
\begin{align*}
	\norm{\bu}_{\WW^{1,q}(\Omega)} \leq C_{\Omega,q}\Big(\norm{f}_{L^q(\Omega)} + \norm{\bf a}_{ \WW^{1-  1/q,1}(\Gamma)} \Big).
\end{align*}
\end{lemma}

\bigskip

We further recall the following inequalities:
\begin{itemize}
    \item \textbf{Korn's inequality}: 
    Let $\Omega \subset \R^d$ with $d \geq 2$ be a bounded domain with Lipschitz-boundary.
    There exists a positive constant $C_\text{K}$ depending only on $\Omega$ such that for all $\v \in \HH^1(\Omega)$,
    \begin{align}
        \label{Korn}
        \norm{\v}_{\HH^1} \leq C_\text{K} \Big (\norm{\v}^2_{\LL^2} + \norm{\D\v}^2_{\LL^2} \Big)^{1/2}.
    \end{align}
    \item \textbf{Gagliardo--Nirenberg inequality}:
    Let $\Omega \subset \R^d$ with $d \geq 2$ be a bounded domain with Lipschitz-boundary. 
    We assume that $p,q,r \in [1,\infty]$, $m,j\in\mathbb N_0$ with $0\le j < m$, and $\theta \in [\frac jm, 1]$ satisfy the relation
    \begin{align*}
        j- \frac dp = \left(m -\frac dr\right)\theta + \left(- \frac dq \right)(1-\theta).
    \end{align*}
    Then, there exists a positive constant $C_\text{GN}$ depending only on $\Omega$, $d$, $m$, $j$, $p$, $q$, $r$, and $\theta$, such that for all ${\bf f} \in \WW^{m,r}(\Omega) \cap \LL^q(\Omega)$,
    \begin{align}
        \label{GN}
        \norm{{D^j  \bf f}}_{\LL^p} \leq C_\text{GN} \norm{{\bf f}}_{\WW^{m,r}}^{\theta}\norm{{\bf f}}_{\LL^q}^{1-\theta}.
    \end{align}
    \item  \textbf{Agmon's inequality} (see, e.g., \cite[Lem.~4.10]{ConstantinFoias}):
    Let $\Omega \subset \R^d$ with $d \in\{2,3\}$ be an open, bounded domain of class $C^2$.
    There exists a positive constant $C_\text{AG}$ such that for all $\mathbf{f}\in \HH^2(\Omega)$,
    \begin{align}
        \label{AG}
        \norm{\mathbf{f}}_{\LL^\infty} 
        \le C_\text{AG} \norm{\mathbf{f}}_{\HH^1}^{\frac 12} \norm{\mathbf{f}}_{\HH^2}^{\frac 12}.
    \end{align}
    \end{itemize}

\subsection{Assumptions} \label{SECT:ASS}

The following assumptions are supposed to hold throughout the paper.

\newcounter{assump}
\begin{enumerate}[label=$\bf {A \arabic*}$, ref = $\bf {A \arabic*}$]

\item \label{ass:dom} \stepcounter{assump}
The set $\Omega\subset\R^d$ with $d \in \{2,3\}$ is a smooth, bounded domain with boundary $\Gamma:=\del\Omega$. Moreover, the parameters $\nu$, $\eps$, and $\gamma$ are given positive constants.

\item \label{ass:tensors} \stepcounter{assump}
The phase mobility tensor $\CC$ and the nutrient mobility tensor $\DD$ are  bounded, continuously differentiable functions 
\begin{align}
    \label{DEF:CD}
    \CC: \R^L \times \R^M \to \R^{L\times d \times L\times d }\;, 
    \quad
    \DD: \R^L \times \R^M \to \R^{M\times d \times M\times d }.
\end{align}
 Moreover, $\CC$ and $\DD$ are symmetric in the following sense:
\begin{alignat*}{2}
    [\CC]_{ijkl} &= [\CC]_{klij} 
    \quad\text{for all $i,k\in\{1,...,L\}$ and $j,l\in\{1,...,d\}$,}
    \\
    [\DD]_{ijkl} &= [\DD]_{klij} 
    \quad\text{for all $i,k\in\{1,...,M\}$ and $j,l\in\{1,...,d\}$.}
\end{alignat*}
There further exist positive constants $C_0$ and $D_0$ such that
for all $\p \in\R^L$, $\s\in\R^M $, $\mathbf{A} \in \R^{L\times d}$, and $\mathbf{B} \in \R^{M\times d}$,
\begin{equation*}
  C_0 |\mathbf{A}|^2 
  \leq \CC({\p, \s}) \mathbf{A} : \mathbf{A}, 
  \qquad 
  D_0 |\mathbf{B}|^2 
  \leq \DD({\p, \s}) \mathbf{B} : \mathbf{B}.
  \end{equation*}
This means that the tensors $\CC(\p,\s)$ and $\DD(\p,\s)$ are uniformly positive definite for all $\p \in\R^L$ and $\s\in\R^M $. 

\item \label{ass:visc} \stepcounter{assump}
The viscosities are functions $\eta,\lambda\in C^1_b(\R^L; \R)$, and there exist constants $\eta_0,\eta_1,\lambda_*$ such that, for every $\p \in \R^L$, it holds
\begin{align}\label{EST:VISC}
   0< \eta_0 \leq \eta (\p) \leq \eta_1,
    \qquad 
    0 \leq \lambda (\p) \leq \lambda_*.
\end{align}

\item \label{ass:nutrient} \stepcounter{assump}
The chemical free energy density $N$ is defined as
\begin{align}
    \label{DEF:N}
    N:\R^L\times\R^M \to \R, \quad
    N(\p,\s) 
    = \frac {\chi_{\bsig}}2|\s|^2 - G(\p,\s).
\end{align}
 Here, $\chi_{\bsig}$ is a positive constant, and the function $G$ is given as
\begin{align}
    G(\p,\s) = 
        \s^\top \mathbb{B} \p 
        + \mathbf{a}\cdot\p 
        + \mathbf{b}\cdot\s
        + c
    \quad\text{for all $\p\in\R^L$, $\s\in\R^M$},
\end{align}
with prescribed coefficients $\mathbb{B}\in\R^{M\times L}$,  $\mathbf{a} \in \R^L$, $\mathbf{b} \in \R^M$ and $c \in \R$.
We will use the notation
\begin{subequations}
\begin{alignat*}{2}
    \Gp(\p,\s) &:= \del_\p G(\p,\s) 
    = \mathbb{B}^\top \s + \mathbf{a},  
    &\quad\Np(\p,\s) &:= \del_\p N(\p,\s)
    =-\mathbb{B}^\top \s- \mathbf{a},\\
    \Gs(\p,\s) &:= \del_\s G(\p,\s) 
    = \mathbb{B}\p + \mathbf{b},
    &\quad\Ns(\p,\s) &:= \del_\s N(\p,\s)
    = \chi_{\bsig} \s - \mathbb{B} \p - \mathbf{b},\\
    \Gsp(\p,\s) &:= \del_\s\del_\p G(\p,\s) = \mathbb{B},
    &\quad\Nsp(\p,\s) &:= \del_\s\del_\p N(\p,\s) = -\mathbb{B}, \\
    \Gss(\p,\s) &:= \del_\s^2 G(\p,\s) = \mathbf{0}, 
    &\quad\Nss(\p,\s) &:= \del_\s^2 N(\p,\s) = \chi_{\bsig}\,\mathbb{I}, \\
    \Gpp(\p,\s) &:= \del_\p^2 G(\p,\s) = \mathbf{0},
    &\quad\Npp(\p,\s) &:= \del_\p^2 N(\p,\s) = \mathbf{0},
\end{alignat*}
\end{subequations}
where $\mathbb I$ stands for the identity matrix in $\R^{M\times M}$.
In particular, we have $\Np=-\Gp$, $\Nsp=-\Gsp$, and $\Npp=-\Gpp$.

Consequently, there exist positive constants 
$A_G$, $B_G$, $C_G$ and $D_G$
such that for all $\p\in\R^L$ and $\s\in\R^M$:
\begin{gather}
    \label{EST:G}
    |G(\p, \s)| \leq C_G ( |\p||\s|+|\p|+|\s|+1),
    \\
    \label{EST:DG}
    |\Gp(\p, \s) | \leq A_G (|\s|+1), 
    \quad 
    |\Gs(\p, \s)|  \leq B_G (|\p|+1), 
    \\
    \label{EST:DDG}
    \norm{\Gsp(\p, \s)} \leq D_G, 
\end{gather}
where $\norm{\,\cdot\,}$ stands for the operator norm.
This directly implies the existence of positive constants $A_N$, $B_N$, and $C_N$
such that for all $\p\in\R^L$ and $\s\in\R^M$:
\begin{gather}
    \label{EST:N}
    |N(\p,\s)| \leq C_N (|{\s}|^2 + |\p||\s| 
        + |\p| + |\s| + 1), \\
    \label{EST:DN}
    |\Np(\p,\s)| \leq A_N (|\s|+1), 
    \quad 
    |\Ns(\p,\s)|  \leq B_N (|\p|+|\s|+1).
\end{gather}
It further follows that $\Nss$ is uniformly positive definite with
\begin{align}
    \label{EST:PD:NSS}
    \bx^\top \Nss(\p,\s) \bx 
    = \chi_{\bsig}\abs{\bx}^2 
\end{align}
for all $\p\in\R^L$ and $\bx,\s\in\R^M$. In combination with the assumptions on the tensor $\mathbb D$ in \ref{ass:tensors}, this ensures that the nutrient equation \eqref{MCHB:5} has a parabolic structure.
Moreover, for all $\p\in\R^L$ and $\s\in\R^M$ the matrix $\Nss(\p,\s)$ is invertible and the operator norm of the inverse matrix is uniformly bounded by
\begin{align}
    \label{EST:INV:NSS}
    \norm{ \big( \Nss(\p,\s) \big)^{-1} } \le \chi_{\bsig}^{-1}.
\end{align}

\item \label{ass:sources} \stepcounter{assump}
The source terms $\S_{\bphi}$ and $\S_{\bsig}$ are  continuously differentiable, vector-valued functions
\begin{align}
    \label{DEF:SPS}
    \Sp:\R^L\times\R^M\times\R^L \to \R^L,
    \quad
    \Ss:\R^L\times\R^M\times\R^L \to \R^M.
\end{align}
We further assume that there exist continuously differentiable  functions 
\begin{alignat}{2}
    \label{DEF:LTP}
    \Lam_{\bphi}&:\R^L\times\R^M \to \R^L, 
    &\quad
    \Th_{\bphi}&:\R^L\times\R^M \to \R^{L\times L},
    \\
    \label{DEF:LTS}
    \Lam_{\bsig}&:\R^L\times\R^M \to \R^M, 
    &\quad
    \Th_{\bsig}&:\R^L\times\R^M \to \R^{M\times L},
\end{alignat}
such that $\Sp$ and $\Ss$ exhibit the following decomposition:
\begin{align}
    \label{DEC:SP}
	\S_{\bphi} (\p,\s,\m)& = \Lam_{\bphi}(\p,\s) 
	    - \Th_{\bphi}(\p,\s) \m ,
	\\ 
	\label{DEC:SS}
	\S_{\bsig} (\p,\s,\m)& = \Lam_{\bsig}(\p,\s) 
	    - \Th_{\bsig}(\p,\s) \m 
\end{align}
for all $\p,\m \in \R^L$ and $\s\in\R^M$.

Moreover, we demand that there exist positive constants $A_{\bphi}, A_{\bsig}, B_{\bphi}$ and $B_{\bsig}$ such that, for all $\p\in\R^L$ and $\s\in\R^M$, we have
\begin{alignat}{2}
    \label{EST:LAPH}
	&|\Lam_{\bphi}(\p,\s)|\leq A_{\bphi} ( |\p| + |\s| +1) , 
	&&\quad 
	\|\Th_{\bphi}(\p,\s)\| \leq B_{\bphi},\\
	\label{EST:LATH}
	&|\Lam_{\bsig}(\p,\s)|\leq A_{\bsig} ( |\p| + |\s| +1) , 
	&&\quad 
	\|\Th_{\bsig}(\p,\s)\| \leq B_{\bsig}.
\end{alignat}
In particular, this entails that there exists a positive constant $B_S$ such that, for all $\p,\m\in\R^L$ and $\s\in\R^M$,
\begin{align}
    \label{EST:SPS}
	|\Sp(\p,\s,\m) |+|\Ss(\p,\s,\m) | 
	\leq B_S ( |\p| + |\s| +|\m| +1).
\end{align}

\item \label{ass:Sv} \stepcounter{assump}
The source term $\Sv$ is a continuously differentiable scalar function
\begin{align}
    \label{DEF:SV}
    \Sv:\R^L\times\R^M \to \R .
\end{align}
We further assume that there exists a
positive constant $A_S$
such that for all $\p \in \R^L$ and $\s \in \R^M$,
\begin{align}
    \label{EST:SV}
	|\Sv(\p,\s) | \leq A_S.
\end{align} 

\item \label{ass:source:boundary} \stepcounter{assump}
The boundary source term $\S_{\Gamma}$ is continuously differentiable, vector-valued function
\begin{align*}
    \SG: {\R^L\times\R^M} \to \R^M.
\end{align*}
Furthermore, there exists a continuously differentiable function
\begin{alignat}{2}
    \label{DEF:LTG}
    \Lam_{\Gamma}:\R^L\times\R^M \to \R^M, 
\end{alignat}
and a nonnegative constant $K$ such that
\begin{align}
    \label{DEC:SG}
	\S_{\Gamma} (\p,\s) & =  K (\Lam_{\Gamma}(\p,\s)
	    - \s),
\end{align}
for all $\p \in \R^L$ and $\s\in\R^M$. 
Moreover, there exists a positive constant $A_\Gamma$ such that for all $\p\in\R^L$ and $\s\in\R^M$,
\begin{align}
    \label{EST:LAGA}
	|\Lam_\Gamma(\p,\s)|\leq A_\Gamma.
\end{align}

\item \label{ass:pot} \stepcounter{assump}
The potential $\Psi$ belongs to $C^2(\R^L; \R)$ and there exist positive constants $A_{\Psi},B_{\Psi}$ such that for every $\p \in \R^L$ it holds that
\begin{align}
    \label{superquad}
    \Psi (\p) \geq A_{\Psi}|\p|^2 - B_{\Psi} .
\end{align}
In addition, the potential can be decomposed as $\Psi= \Psi^{(1)} + \Psi^{(2)}$ with $\Psi^{(1)},\Psi^{(2)}\in C^2(\R^L;\R)$, where $\Psi^{(1)}$ is convex,
and $\Psi^{(2)}: \R^{L} \to {\R^L}$ is Lipschitz-continuous.

For the gradient and the Hessian of $\Psi$, we will write
\begin{align*}
    \Psi_{\bphi} := \nabla \Psi, \quad
    \Psi_{\bphi\bphi} := D^2 \Psi,
\end{align*}
and we will use an analogous notation for $\Psi^{(1)}$ and $\Psi^{(2)}$.

Moreover, we assume:
\begin{enumerate}[label = $\mathbf{A\theassump.\arabic*}$,ref = $\mathbf{A\theassump.\arabic*}$]
   
\item\label{ass:pot:1}
    If the matrix-valued function $\Th_{\bphi}$ in \ref{ass:sources} is uniformly positive definite, that is
    \begin{align}
    \label{COND:POSDEF}
        \exists \theta_0 >0 \quad \forall \p,\bz \in \R^L,\, \s\in\R^M:\quad
        \bz^\top \Th_{\bphi}(\p,\s) \bz \ge \theta_0 \abs{\bz}^2,
        \quad 
    \end{align}
    there exist an exponent $ \rho \in [2,4]$, and positive constants $B_\Psi$, $C_\Psi$ and $D_\Psi$, such that
    \begin{align}
        \label{PSI:CASE2}
        \begin{gathered}
        |\Psi (\p)|  \leq B_{\Psi}(|\p|^\rho+1),
        \quad 
        |\Psi_{\bphi} (\p)| \leq C_{\Psi}(|\p|^{\rho-1}+1),
        \\
        |\Psi_{\bphi\bphi} (\p)|\leq D_{\Psi}(|\p|^{\rho-2} + 1).
        \end{gathered}
    \end{align}
    for all $\p\in\R^L$.
    
    \item\label{ass:pot:2}
    If the matrix-valued function $\Th_{\bphi}$ in \ref{ass:sources} is \emph{not} uniformly positive definite, that is \eqref{COND:POSDEF} does not hold,
    there exist positive constants $B_\Psi$, $C_\Psi$ and $D_\Psi$ such that
    \begin{align}
        \label{PSI:CASE1}
        |\Psi (\p)| \leq B_{\Psi}(|\p|^2+1),
        \quad 
        |\Psi_{\bphi} (\p)| \leq C_{\Psi}(|\p|+1),
        \quad   
        |\Psi_{\bphi\bphi}(\p)| \leq D_{\Psi}
    \end{align}
    for all $\p\in\R^L$.
\end{enumerate}

\item \label{ass:interface} \stepcounter{assump}
The diffuse interface parameter $\eps$ is a fixed positive constant which satisfies
\begin{align}
    \label{COND:EPS}
    \eps < \frac{\gamma \chi_{\bsig} A_\Psi}{8 C_G^2}.
\end{align}
Here, $A_\Psi$ is the constant from \eqref{superquad}, $\chi_{\bsig}$ and $C_G$ are the constants from \ref{ass:nutrient}, and $\gamma>0$ is the surface tension parameter. 

\end{enumerate}

\bigskip
\begin{remark} \label{REM:ASS}
    \textnormal{(a)} Instead of fourth-order tensors, $\CC$ and $\DD$ could also just be matrices (second-order tensors) in $\R^{L\times L}$ and $\R^{M\times M}$, respectively. This because a matrix can still be described by an associated fourth-order tensor in the following way: 

    Suppose that $n,m\in\mathbb N$, and $\mathbf M \in \R^{n\times n}$ is a given matrix. Then the corresponding fourth-order tensor $\mathbb M\in \R^{n\times m\times n \times m}$ is defined as
    \begin{align*}
        [\mathbb M]_{ijkl} := \delta_{jl} [\mathbf M]_{ik} \quad\text{for all $i,k\in\{1,...,n\}$ and $j,l\in\{1,...,m\}$,}       
    \end{align*}
    where $\delta_{jl}$ stands for the Kronecker symbol.
    In particular, for any matrix $\mathbf A\in\R^{n\times m}$, it thus holds
    \begin{align*}
        [\mathbb M \mathbf A]_{ij} 
        = \sum_{k=1}^n \sum_{l=1}^m [\mathbb M]_{ijkl} [\mathbf A]_{kl}
        = \sum_{k=1}^n [\mathbf M]_{ik} [\mathbf A]_{kj}
        = [\mathbf M \mathbf A]_{ij}
        \quad\text{for all $i\in\{1,...,n\}$ and $j\in\{1,...,m\}$,}
    \end{align*}
    which means $\mathbb M \mathbf A = \mathbf M \mathbf A$.
    \\[1ex]
    \textnormal{(b)} We point out that the tensors, the nutrient density and the source terms proposed in Section~\ref{SEC:CONC} for the special case $L=3$ and $M=1$ fit into the framework of the above assumptions.    
    To be precise, 
    the source terms $\Ss$ defined in \eqref{DEF:Ss} and $\Sp$ introduced in \eqref{DEF:SPM} satisfy \ref{ass:sources} (with $\Th_{\bphi}\equiv \mathbf {0}$ and $\Th_{\bsig}\equiv \mathbf {0}$),
    the source term $\Sv$ proposed in \eqref{DEF:SV*} fulfills \ref{ass:Sv},
    the nutrient density defined in \eqref{DEF:N1}--\eqref{DEF:N4} satisfies \ref{ass:nutrient},
    and the tensors $\CC$ and $\DD$ introduced in Section~\ref{SECT:TENS} fulfill \ref{ass:tensors}.
    Moreover, the boundary source term proposed in \eqref{concrete:SG} satisfies \ref{ass:source:boundary} with $\Lam_\Gamma(\p,s) \equiv \sigma_\Gamma$, provided that the prescribed function $\sigma_\Gamma$ is sufficiently regular.
    \\[1ex]
    \textnormal{(c)} As the positive coefficient $\eps$ is related to the thickness of the diffuse interface, it is usually chosen to be very small in the applications. Therefore, assumption \ref{ass:interface} is not a severe restriction.
\end{remark}

%
%

\subsection{Main results} 
Let us now present the main results of this paper.
First, after introducing the notion of weak solutions to the multiphase Cahn--Hilliard--Brinkman system (MCHB), we state the existence of such a weak solution. 
Unfortunately, in this general setting, we are not able to prove the uniqueness of weak solutions. This is mainly due to the fairly low regularity of the nutrient variable $\bsig$, of which we can merely establish $\bsig \in \L\infty {\LL^2} \cap \L2 {\HH^1}$.

A weak solution to the multiphase Cahn--Hilliard--Brinkman system (MCHB) is defined as follows.

\pagebreak[2]

\begin{definition}[Definition of a weak solution to (MCHB)]
\label{DEF:WEAKSOL}
The quintuplet $(\bphi, \bmu, \bsig, \v, p)$ is called a weak solution to the multiphase Cahn--Hilliard--Brinkman system \eqref{MCHB}
if the following conditions are satisfied:
\begin{enumerate}[label=$\mathrm{(\roman*)}$,ref=$\mathrm{(\roman*)}$]
\item The functions $(\bphi, \bmu, \bsig, \v, p)$ possess the regularities
\begin{align}
\label{REG:MCHB}
\left\{\;
\begin{aligned}
	&\bphi \in  \H1 {(\HH^1)'} 
	\cap C([0,T];\LL^2)
	\cap \L\infty {\HH^1} ,
	\\ 
	&\bsig \in \W{1,\frac 43} {(\HH^1)'} 
	    \cap C([0,T];\Vp) 
	    \cap \L\infty{\LL^2}
	    \cap \L2 {\HH^1},
	\\
	&\bphi\vert_\Gamma \in C([0,T];\LL^2_\Gamma),
	\quad \bsig\vert_\Gamma \in \L4{\LL^2_\Gamma},
	\\
	&\bmu \in  \L2 {\HH^1},
	\quad
	\v \in  \L2 {\HH^1},
	\quad
	p \in \L{\frac 43}{L^2},
	\\
	& \Div(\bphi\otimes\v) \in L^2(0,T;\LL^{\frac 32}),
	\quad \Div(\bsig\otimes\v) \in  \L1 {\LL^{\frac 32}} .
\end{aligned}
\right.
\end{align}
\item The weak formulation 
\begin{subequations}
\label{WF}
\begin{align}
    \label{WF:1}
	&\intO  \T (\bphi, \v, p) : \nabla \be + \nu  \v \cdot \be  \dx
	=
    \intO (\nabla \bphi)^\top  \bmu \cdot \be + (\nabla \bsig)^\top \Ns(\bphi,\bsig) \cdot \be \dx,
	\\
	& \label{WF:2}
	\ang{\delt \bphi}{\bz}_{\HH^1}  + \intO \Div(\bphi\otimes\v)\cdot \bz \dx
	= - \intO \CC(\bphi,\bsig) \nabla \bmu : \nabla \bz + \S_{\bphi}(\bphi,\bsig,\bmu) \cdot \bz
	\dx,
	\\ & \label{WF:3}
	\intO \bmu \cdot \bt \dx = \intO \gamma\eps \nabla \bphi : \nabla \bt + \gamma\eps^{-1} \Psi_{\bphi}(\bphi) \cdot \bt
	+ \N_{\bphi}(\bphi,\bsig) \cdot \bt \dx,
	\\
	& \notag
	\ang{\delt \bsig}{\bz}_{\HH^1} 	
	+ \intO \Div(\bsig\otimes\v)\cdot \bx \dx
	= - \intO \DD(\bphi,\bsig) \nabla \Ns(\bphi,\bsig): \Grad \bx \dx
	\\ & \qquad \label{WF:4}
	- \intO \Ss(\bphi,\bsig,\bmu) \cdot \bx \dx
	+ \intG  \SG(\bphi,\bsig)\cdot \bx \dS,
\end{align}
holds almost everywhere on $(0,T)$ for all test functions $\be \in H^1(\Omega;\R^d)$, $\bz,\bt \in H^1(\Omega;\R^L)$ and $\bx \in H^1(\Omega;\R^M)$.
It further holds that
\begin{alignat}{2}
    \label{SF:DIV}
	\Div (\v) & = S_\v(\bphi,\bsig) 
	&&\quad a.e. \, \hbox{in} \,\, Q,
	\\
	\bphi\vert_{t=0}& = \bphi_0 
	&&\quad a.e. \, \hbox{in} \,\, \Omega, 
	\\
    \<\bsig\vert_{t=0}, \Phi>_{\HH^1} &= \<\bsig_0, \Phi>_{\HH^1} 
    &&\quad \text{for all $\Phi \in H^1(\Omega;\R^M)$.}
\end{alignat}
\end{subequations}
\end{enumerate}
\end{definition}

The corresponding existence result reads as follows.
\begin{theorem}[Existence of weak solution to (MCHB)]
\label{THM:EXISTENCE:WEAK}
Suppose that the assumptions \ref{ass:dom}--\ref{ass:interface} hold, and let $\bphi_0 \in H^1(\Omega;\R^L)$ and $\bsig_0 \in L^2(\Omega; \R^M)$ be any initial data.

Then, there exists a weak solution $(\bphi, \bmu, \bsig, \v, p)$ to system \eqref{MCHB} in the sense of Definition~\ref{DEF:WEAKSOL}. 
In addition, this solution satisfies 
\begin{align}
    \label{REG:MCHB:ADD}
    \bphi  \in \L2 {\HH^3} 
        \cap C([0,T];\HH^1), 
    \quad 
    \Psi_{\bphi}(\bphi) \in \L{4}{\LL^2}\cap \L{2}{\LL^6}.
\end{align}
Moreover, there exists a positive constant $C_{B}$ that may depend on the initial data and the constants introduced in Section~\ref{SECT:ASS} except for $\eta_0$ such that
\begin{align}
    \label{EST:CHB}
    &\norm{\bphi}_{\L\infty{\HH^1} \cap \L2{\HH^3}}
    +\norm{\bsig}_{\L\infty{\LL^2} \cap \L2{\HH^1}}
    +\norm{\bsig}_{\L4{\LL^2_\Gamma}}
    \notag\\
    &\quad
    +\norm{\bmu}_{\L2{\HH^1}}
    +\norm{\v}_{\L{2}{\LL^2}}
    +\bignorm{\sqrt{\eta(\bphi)}\, \D \v}_{\L2{\LL^2}}^2
    +\norm{p}_{\L{4/3}{L^2}}
    \notag\\
    &\quad
    + \norm{\Psi_{\bphi}(\bphi)}_{ \L{4} {\LL^{2}}\cap\L{2} {\LL^{6}} }
    + \norm{\nabla \Ns(\bphi,\bsig)}_{ \L{2} {\LL^{2}}} 
    \notag\\
    & \quad 
    +\norm{\Div(\bphi\otimes\v)}_{ \L{4/3}{\LL^{3/2}}}
    +\norm{\Div(\bsig\otimes\v)}_{ \L{1}{\LL^1}}
    \le C_B.
\end{align}

\end{theorem}

The proof of this theorem will be presented in Section~\ref{SECT:EX:MCHB}.

The second main result is the existence of weak solutions to the multiphase Cahn--Hilliard--Darcy system (MCHD).
Roughly speaking, this can be established by investigating the ``Darcy limit'' where the positive viscosities $\eta$ and $\lambda$ in the Cahn--Hilliard--Brinkman system (MCHB) are sent to zero. In this way, we can show that the corresponding weak solutions of the system (MCHB) converge to a weak solution of the system (MCHD).

To begin with, let us start by presenting the notion of weak solution for the Cahn--Hilliard--Darcy model.

\begin{definition}[Definition of a weak solution to (MCHD)]
\label{DEF:WEAKSOL:DARCY}
The quintuplet $(\bphi, \bmu, \bsig, \v, p)$ is called a weak solution to the multiphase Cahn--Hilliard--Darcy system \eqref{MCHB}
if the following conditions are satisfied:
\begin{enumerate}[label=$\mathrm{(\roman*)}$,ref=$\mathrm{(\roman*)}$]
\item The functions $(\bphi, \bmu, \bsig, \v, p)$ possess the regularities 
\begin{align}
\label{REG:MCHD}
 \left\{\;
 \begin{aligned}
	&\bphi \in  \W{1,\frac 85} {(\HH^1)'} \cap C([0,T];\LL^2) \cap \L\infty {\HH^1},
	\\ 
	&\bsig \in \W{1,1} {(\WW^{1,4})'} 
	    \cap C([0,T];(\WW^{1,4})')
	    \cap \L\infty{\LL^2}
	    \cap \L2 {\HH^1},
	\\
	&\bphi\vert_\Gamma \in C([0,T];\LL^2_\Gamma),
	\quad \bsig\vert_\Gamma \in \L4{\LL^2_\Gamma},
	\quad \bmu \in  \L2 {\HH^1},
	\\
	& p \in  \L{\frac 43}{L^2} \cap L^1\big(0,T;{W^{1,\frac 32}_0}\big),
	\quad
	\v \in  \L2 {\LL^2_\Div},
	\\
	&\Div(\bphi\otimes\v) \in L^{\frac 85}(0,T;(\HH^1)') \cap  L^{\frac 43}(0,T;\LL^{\frac 32}), 
	\quad \Div(\bsig\otimes\v) \in L^{1}(0,T;\LL^{1}).
    \end{aligned}
\right.
\end{align}
\item The weak formulation 
\begin{subequations}
\label{WF:D}
\begin{align}
   \label{WF:D:1}
    &\intO \Grad p \cdot \be + \nu \v \cdot \be \dx = 
    \intO (\nabla \bphi)^\top  \bmu \cdot \be + (\nabla \bsig)^\top \Ns(\bphi,\bsig) \cdot \be \dx,
	\\ & 
	\label{WF:D:2}
	\< \delt \bphi , \bz >_{\HH^1}  + \intO \Div(\bphi\otimes\v)\cdot \bz \dx
	= - \intO \CC(\bphi,\bsig) \nabla \bmu : \nabla \bz + \S_{\bphi}(\bphi,\bsig,\bmu) \cdot \bz
	\dx,
	\\ & 
	\label{WF:D:3}
	\intO \bmu \cdot \bt \dx = \intO \gamma\eps \nabla \bphi : \nabla \bt + \gamma\eps^{-1} \Psi_{\bphi}(\bphi) \cdot \bt
	+ \N_{\bphi}(\bphi,\bsig) \cdot \bt \dx,
	\\ & \notag
	\< \delt \bsig , \bx >_{\WW^{1,4}}  	
	+ \intO \Div(\bsig \otimes \v) \cdot \bx \dx
	= - \intO \DD(\bphi,\bsig) \nabla \Ns(\bphi,\bsig): \Grad \bx \dx
	\\ & 
	\label{WF:D:4}
	\qquad 
	- \intO \Ss(\bphi,\bsig,\bmu) \cdot \bx \dx
	+ \intG  \SG(\bphi,\bsig)\cdot \bx \dS,
\end{align}
holds almost everywhere on $(0,T)$ for all test functions $\be \in L^2(\Omega;\R^d)$, $\bz,\bt \in H^1(\Omega;\R^L)$, and $\bx \in W^{1,4}(\Omega;\R^M) $.
Moreover, the following conditions are fulfilled
\begin{alignat}{2}
    \label{SF:D:DIV}
    \Div(\v) &= \Sv(\bphi,\bsig)
    &&\quad \text{a.e.~in $Q$}, \\
    \label{INI:D:1}
	\bphi\vert_{t=0}& = \bphi_0 
	&&\quad \text{a.e.~in $\Omega$}, \\
	\label{INI:D:2}
    \<\bsig\vert_{t=0}, \Phi>_{\WW^{1,4}} 
    &= \<\bsig_0, \Phi>_{\WW^{1,4}} 
    &&\quad\text{for all $\Phi \in W^{1,4}(\Omega;\R^M)$.}
\end{alignat}
\end{subequations}
\end{enumerate}
\end{definition}

The sense in which weak solutions to (MCHB) 
convergence to a weak solution of (MCHD) is specified by the following theorem.

\begin{theorem}[``Darcy limit'' and existence of a weak solution to (MCHD)]
\label{THM:EXISTENCE:WEAK:DARCY}
Suppose that \ref{ass:dom}--\ref{ass:interface} are fulfilled, and let  $\bphi_0 \in H^1(\Omega;\R^L)$ and $\bsig_0 \in L^2(\Omega; \R^M)$ be arbitrary initial data.
Furthermore, let $\{\eta_n\}_{n \in \mathbb{N}}$ and $\{\lambda_n\}_{n \in \mathbb{N}}$ be sequences of viscosity functions such that for each fixed $n\in\mathbb N$, $\eta_n$ and $\lambda_n$ are compatible with \ref{ass:visc}.
We further assume that
\begin{align}
    \label{darcy:ass:viscosities}
    \norm{\eta_n}_{C_b(\R)} \to 0, 
    \qquad 
    \norm{\lambda_n}_{C_b(\R)} \to 0,
    \quad
    \text{as $n\to\infty$.}
\end{align}
For any $n\in\mathbb N$, let $(\bphi_n, \bmu_n, \bsig_n, \v_n, p_n)$ denote the weak solution of the multiphase Cahn--Hilliard--Brinkman system \eqref{MCHB} constructed in Theorem~\ref{THM:EXISTENCE:WEAK} associated with the viscosities $\eta_n$ and $\lambda_n$.

Then, there exists a quintuplet $(\bphi, \bmu, \bsig, \v, p)$ such that for all $s\in [0,1)$,
\begin{subequations}
\label{CD:WS*}
\begin{alignat}{2} 
      \notag
    \bphi_n & \to \bphi 
    &&\quad\text{weakly-$^*$ in $ L^\infty(0,T;\HH^1)$,} \notag\\
    &&&\qquad \text{weakly in $ \W{1,\frac 85} {\Vp} 
        \cap L^2(0,T;\HH^3)$, a.e.~in $Q$,}
    \notag\\       \notag
    &&&\qquad\text{and strongly in $C([0,T];\HH^s) \cap L^2 (0,T;\HH^{2+s})$,}
    \\
      \notag    \bphi_n\vert_\Gamma &\to \bphi\vert_\Gamma
    &&\quad\text{strongly in $C([0,T];\LL^2_\Gamma)$,
        and a.e.~on $\Sigma$,}
    \\ 
      \notag    \bsig_n & \to \bsig
    &&\quad\text{weakly-$^*$ in $\L\infty {\LL^2} $, }\notag\\
    &&&\qquad\text{weakly in $ \W{1,1} {({\WW^{1,4}})'} \cap \L2 {\HH^1}$, a.e.~in $Q$,}  
    \notag\\       \notag
    &&&\qquad\text{and strongly in $C([0,T];({\WW^{1,4}})') \cap L^2 (0,T;\HH^s) $,}
    \\
      \notag    \bsig_n\vert_\Gamma &\to \bsig\vert_\Gamma
    &&\quad\text{weakly in $L^4(0,T;\LL^2_\Gamma)$, 
        strongly in $\L2 {\LL^{2}_\Gamma}$,
        and a.e.~on $\Sigma$,}
    \\ 
      \notag    \bmu_n & \to \bmu && \quad \text{weakly in $\L2 {\HH^1}$,}
    \\
      \notag    \v_n & \to \v && \quad \text{weakly in $ \L2 {\LL^2_{\Div}}$,} 
    \\
      \notag    p_n & \to p && \quad \text{weakly in $\L{\frac 43} {L^2}$,} 
    \\ 
      \notag    \Div(\bphi_n\otimes\v_n) & \to \ttau 
    &&  \quad\text{weakly in $\L{\frac 85} { (\HH^1)'} \cap \L{\frac 43} {\LL^{\frac 32}}$},
    \\ 
      \notag    \Div(\bsig_n\otimes\v_n) & \to  \ttheta 
    && \quad \text{weakly in $\L{1} {\LL^1}$},
\end{alignat}
\end{subequations}
as $n\to\infty$, along a nonrelabeled subsequence.

Moreover, the limit $(\bphi, \bmu, \bsig, \v, p)$ is a weak solution to the multiphase Cahn--Hilliard--Darcy system \eqref{MCHB} in the sense of Definition~\ref{DEF:WEAKSOL:DARCY}.
In addition, this solution satisfies 
\begin{align}
    \label{REG:MCHD:ADD}
    \bphi  \in \L2 {\HH^3} 
        \cap C([0,T];\HH^1), 
    \quad 
    \Psi_{\bphi}(\bphi) \in \L{4}{\LL^2}\cap \L{2}{\LL^6}.
\end{align} 

\end{theorem}

\medskip

\textbf{Comment.} We point out that for any $n\in\mathbb N$, the choice of the corresponding weak solution to (MCHB) is \emph{explicit}, since we are choosing exactly the corresponding weak solution that was constructed in Theorem~\ref{THM:EXISTENCE:WEAK}. This means that even though the weak solutions to (MCHB) might not be unique, we do not require the axiom of choice for our approach.

The proof of Theorem~\ref{THM:EXISTENCE:WEAK:DARCY} is presented in Section~\ref{SECT:EX:MCHD}.

\section{Existence of weak solutions to the (MCHB) system}
\label{SECT:EX:MCHB}
This section is devoted to the construction of a weak solution to the multiphase Cahn--Hilliard--Brinkman system (MCHB)
in the sense of Definition \ref{DEF:WEAKSOL}.

\begin{proof}[Proof of Theorem \ref{THM:EXISTENCE:WEAK}] 
To construct a weak solution we discretize the weak formulation \eqref{WF:1}--\eqref{WF:4} via a Faedo--Galerkin scheme. Then, we derive suitable a priori estimates for the discrete approximate solutions that are independent of the dimension of the finite-dimensional subspace. 
This allows us to show that the sequence of approximate solutions converges to a weak solution of the Brinkman system \eqref{MCHB} in the sense specified by Definition~\ref{DEF:WEAKSOL}. 

In the whole proof, the letter $C$ denotes a generic positive constant that may depend only on the initial data and the constants introduced in Section~\ref{SECT:ASS} (including the final time $T$), except for $\eta_0$ as this constant will play a role in the ``Darcy limit'' (see Theorem~\ref{THM:EXISTENCE:WEAK:DARCY}). The proof is split into several steps.

\noindent
\textbf{Step 1: The Faedo--Galerkin scheme.} 
The idea of our Faedo--Galerkin scheme is to spatially approximate the functions $\bphi$, $\bmu$ and $\bsig$ by functions from suitable finite-dimensional subspaces.

Therefore, we consider the scalar eigenvalue problem
for the Laplace operator with homogeneous Neumann boundary conditions:
\begin{align}
\label{PNEVP}
    -\Lap w = \lambda w \;\;\text{in $\Omega$}, \quad
    \deln w = 0 \;\;\text{on $\Gamma$}.
\end{align}
It is well known, that there exists a sequence $\{(\lambda_i,w_i)\}_{i\in \mathbb{N}}$ of eigenvalues $\lambda_i$ and corresponding eigenfunctions $w_i$.
We further know that all eigenvalues are nonnegative, 
and can be sorted such that they form a nondecreasing sequence $\{\lambda_i\}_{i\in\mathbb N}$ with $\lambda_i \to \infty $ as $i\to \infty$. 
Furthermore, the eigenfunctions can be chosen such that $\norm{w_i}_{L^2} = 1$ for all $i\in\mathbb N$. In particular, for the first eigenfunction, we choose $w_1 = \abs{\Omega}^{-1/2}$. 
As the domain $\Omega$ is smooth, elliptic regularity theory implies that $w_i\in C^\infty(\overline\Omega)$ for all $i\in\mathbb N$. Moreover, the eigenfunctions are orthogonal with respect to the inner product of $L^2(\Omega)$ and thus, they form an orthonormal Schauder basis of $L^2(\Omega)$. 
In addition, the family $\{w_i\}_{i\in\mathbb N}$ is also a Schauder basis of $H^2_\n(\Omega)$.

We now define
\begin{alignat*}{2}
    \w_{(i-1)k+j} &:= w_j \mathbf{e}_i
    &&\quad\text{for all $i\in\{1,...,L\},\,
        j\in\mathbb N$,}\\
    \z_{(i-1)k+j} &:= w_j \mathbf{e}_i
    &&\quad\text{for all $i\in\{1,...,M\},\,
        j\in\mathbb N$.}
\end{alignat*}
where $\mathbf{e}_i$ stands for the $i$-th unit vector in $\R^L$ or $\R^M$, respectively.
It is straightforward to check that the family $\{\w_m\}_{m\in\mathbb N}$ is an orthonormal Schauder basis of $L^2(\Omega;\R^L)$, and also a Schauder basis of $H^2_\n(\Omega;\R^L)$.
Similarly, the family $\{\z_m\}_{m\in\mathbb N}$ is an orthonormal Schauder basis of $L^2(\Omega;\R^M)$, and also a Schauder basis of $H^2_\n(\Omega;\R^M)$.
For any $k\in\mathbb N$, we introduce the finite dimensional subspaces
\begin{align*}
    \CWK
    &:=\mathrm{span}\,\big\{ \w_{(i-1)k+j} \big\}_{i=1,...,L,\, j=1,...,k}
    =\mathrm{span}\,\big\{ \w_1,...,\w_{kL} \big\}
    \subset H^1(\Omega;\R^L), 
    \\
    \CZK
    &:= \mathrm{span}\,\big\{ \z_{(i-1)k+j} \big\}_{i=1,...,M,\, j=1,...,k}
    = \mathrm{span}\,\big\{ \z_1,...,\z_{kM} \big\}
    \subset H^1(\Omega;\R^M), 
\end{align*}
and we write $\mathbb P_{\CWK}$ and $\mathbb P_{\CZK}$ to denote the $\LL^2$-orthogonal projection onto $\CWK$ and $\CZK$, respectively.

We now make the ansatz
\begin{align}
    \label{GAL:ANSATZ}
    \bphi_k(x,t) := \sum_{i=1}^{kL} a_i^k(t) \w_i(x), \quad
    \bmu_k(x,t) := \sum_{i=1}^{kL} b_i^k(t) \w_i(x), \quad
    \bsig_k(x,t) := \sum_{i=1}^{kM} c_i^k(t) \z_i(x),
\end{align}
where the coefficients $a_i^k$, $b_i^k$, $i\in\{1,...,kL\}$, and $c_i^k$, $i\in\{1,...,kM\}$, are assumed to be continuously differentiable functions that are still to be determined.

At every time $t\in[0,T]$ in which the expressions in \eqref{GAL:ANSATZ} are declared, we further introduce the functions $\big(\v_k(t),p_k(t)\big)\in \HH^2\times H^1$ as the unique strong solution of the system
\begin{subequations}
\label{AT}
\begin{alignat}{2}
    \label{AT:1}
    -\Div\big(\T(\bphi_k(t),\v_k(t),p_k(t))\big) + \nu\v_k(t)
    & = \big(\Grad\bphi_k(t)\big)^\top \bmu_k(t) + \big(\Grad\bsig_k(t)\big)^\top \Ns\big(\bphi_k(t),\bsig_k(t) \big)
    &&\quad\text{in $\Omega$,}\\
    \label{AT:2}
    \Div\big(\v_k(t)\big) 
    &= \Sv\big(\bphi_k(t),\bsig_k(t)\big)
    &&\quad\text{in $\Omega$,}\\
    \label{AT:3}
    \T\big(\bphi_k(t),\v_k(t),p_k(t)\big)\n 
    &= \mathbf 0
    &&\quad\text{on $\Gamma$}.
\end{alignat}
\end{subequations}
As the right-hand sides of \eqref{AT:1} and \eqref{AT:2} belong to $L^2(\Omega;\R^d)$ and $H^1(\Omega)$, respectively, the existence and the uniqueness of the solution $\big(\v_k(t),p_k(t)\big)\in \HH^2\times H^1$ follows from a fundamental result on Stokes operators with variable viscosity established in \cite{AbelsTerasawa} that can also be found in \cite[Lem.~1.5]{EbenbeckGarcke}. 

Moreover, we use \eqref{AT:2} and the chain rule to derive the identities
\begin{alignat}{2}
	\label{DEC:divphi}
    \Div\big(\bphi_k(t)\otimes\v_k(t)\big) 
    &= \nabla \bphi_k(t) \; \v_k(t) + \bphi_k(t) \;\Sv\big(\bphi_k(t),\bsig_k(t)\big)
    &&\quad\text{in $\Omega$,}
    \\
    \label{DEC:divsigma}
    \Div\big(\bsig_k(t)\otimes\v_k(t)\big) 
    &= \nabla \bsig_k(t) \; \v_k(t) + \bsig_k(t) \;\Sv\big(\bphi_k(t),\bsig_k(t)\big) 
    &&\quad\text{in $\Omega$,}
\end{alignat}
for all $t\in[0,T]$ in which the expressions in \eqref{GAL:ANSATZ} are declared.

The next goal is to determine the continuously differentiable coefficients  $a_i^k$, $b_i^k$, $i\in\{1,...,kL\}$, and $c_i^k$, $i\in\{1,...,kM\}$, such that the discretized weak formulation
\begin{subequations}
\label{DISC:WF}
\begin{align}
    \label{DISC:WF:1}
	&\intO  \T(\bphi_k,\v_k, p_k) : \nabla \be + \nu  \v_k \cdot \be  \dx
	=
    \intO (\nabla \bphi_k)^\top  \bmu_k \cdot \be + (\nabla \bsig_k)^\top  \Ns(\bphi_k,\bsig_k) \cdot \be \dx,
	\\ &
	\label{DISC:WF:2}
	\ang{\delt \bphi_k}{\bz}_{\HH^1}  + \intO \Div(\bphi_k\otimes\v_k) \cdot \bz \dx
	= - \intO \CC(\bphi_k,\bsig_k) \nabla \bmu_k : \nabla \bz + \Sp(\bphi_k,\bsig_k,\bmu_k) \cdot \bz
	\dx,
	\\ & 
	\label{DISC:WF:3}
	\intO \bmu_k \cdot \bt \dx = \intO \gamma\eps \nabla \bphi_k : \nabla \bt + \gamma\eps^{-1} \Psi_{\bphi}(\bphi_k) \cdot \bt
	+ \Np(\bphi_k,\bsig_k) \cdot \bt \dx,
	\\
	& \notag
	\ang{\delt \bsig_k}{\bx}_{\HH^1}  	
	+ \intO \Div(\bsig_k\otimes\v_k) \cdot \bx \dx
	= - \intO \DD(\bphi_k,\bsig_k) \nabla \Ns(\bphi_k,\bsig_k): \Grad \bx \dx
	\\ & \qquad 
	\label{DISC:WF:4}
	- \intO \Ss(\bphi_k,\bsig_k,\bmu_k) \cdot \bx \dx
	+ \intG  \SG(\bphi_k,\bsig_k)\cdot \bx \dS,
\end{align}
\end{subequations}
is satisfied for all test functions $\be\in H^1(\Omega;\R^d)$, $\bz,\bt\in\CWK$, and $\bx\in\CZK$. Note that we only need to detect  $a_i^k$, $b_i^k$, $i\in\{1,...,kL\}$ and $c_i^k$, $i\in\{1,...,kM\}$, such that the equations \eqref{DISC:WF:2}--\eqref{DISC:WF:4} are fulfilled. Then \eqref{DISC:WF:1} holds automatically due to the construction of $\v_k$ and $p_k$. Of course, also the initial conditions have to be approximated. We thus demand that
\begin{align}
    \label{DISC:INI}
    \bphi_k(0) = \mathbb P_{\CWK}(\bphi_0)\in\CWK,
    \quad
    \bsig_k(0) = \mathbb P_{\CZK}(\bsig_0)\in\CZK
    \quad
    \text{in $\Omega$.}
\end{align}

In the following, we write  $\mathbf{a}^k := (a_1^k,...,a_{kL}^k)$, $\mathbf{b}^k := (b_1^k,...,b_{kL}^k)$, and $\mathbf{c}^k := (c_1^k,...,c_{kM}^k)$ to denote the coefficient vectors.
Plugging the ansatz \eqref{GAL:ANSATZ} into the discrete formulations \eqref{DISC:WF:2}--\eqref{DISC:WF:4}, and testing with  $\w_1,...,\w_{kL}$ and $\z_1,...,\z_{kM}$, respectively, we conclude that the vector $(\mathbf a^k,\mathbf b^k, \mathbf c^k)^\top$ needs to satisfy a system of  $k(L+M)$ nonlinear ordinary differential equations and  $kL$ algebraic equations. By means of the vector-valued algebraic equation resulting from \eqref{DISC:WF:3}, we can replace the variable $\mathbf b^k$ appearing in the right-hand side of the vector-valued ODEs resulting from \eqref{DISC:WF:2} and \eqref{DISC:WF:4} by an expression depending only on $\mathbf a^k$ and $\mathbf c^k$ to eventually obtain a closed system of ODEs for $(\mathbf a^k,\mathbf c^k)^\top$. 
In fact, notice that from \eqref{DISC:INI} we naturally obtain the initial conditions 
\begin{alignat}{2}
    \label{ODE:INI}
    [\mathbf a^k]_i(0) &= a_i^k(0) 
    = \inn{\mathbb \bphi_k(0)}{\w_i}_{\LL^2}
    = \inn{\mathbb \bphi_0}{\w_i}_{\LL^2}, 
    &&\quad i\in\{1,...,kL\},\\
    \quad
    [\mathbf c^k]_i(0) &= c_i^k(0) 
    = \inn{\mathbb \bsig_k(0)}{\z_i}_{\LL^2}
    = \inn{\mathbb \bsig_0}{\z_i}_{\LL^2},
    &&\quad i\in\{1,...,kM\}.
\end{alignat}
In particular, this entails that 
\begin{align*}
    \norm{\bphi_k(0)}_{\HH^1}
    &=
    \norm{\sum_{i=1}^{ kL} [\mathbf a^k]_i(0)\, \w_i}_{\HH^1} 
    \! \! \!  
    \le \norm{\bphi_0}_{\HH^1},
    \quad
    \norm{\bsig_k(0)}_{\LL^2}
    =
    \norm{\sum_{i=1}^{ kM } [\mathbf c^k]_i(0)\, \z_i}_{\LL^2}  
    \! \! \! 
    \le \norm{\bsig_0}_{\LL^2}.
\end{align*} 
Recalling the assumptions \ref{ass:tensors}--\ref{ass:pot}, we notice that the right-hand side of the ODE system depends continuously on the unknown variables $(\mathbf a^k,\mathbf c^k)^\top$. Hence, the Cauchy--Peano theorem implies the existence of at least one local solution 
$(\mathbf a^k,\mathbf c^k)^\top:[0,T_k^*)\cap [0,T] \to  \R^{k(L+M)}$ with $T_k^* >0$. Without loss of generality, we assume that $T_k^*\le T$ and that $(\mathbf a^k,\mathbf c^k)^\top$ is the right-maximal solution of the ODE system mentioned above, that is, $T_k^*$ is chosen as large as possible. 

We can now reconstruct $\mathbf b^k$ by means of the vector-valued algebraic equation as a function $(\mathbf b^k)^\top:[0,T_k^*) \to  \R^{kL}$. Consequently, by \eqref{GAL:ANSATZ} and the construction of $(\v_k,p_k)$, we obtain functions
\begin{align}
\begin{aligned}
    \label{GAL:FUNC}
    &\bphi_k,\bmu_k \in C^1\big([0,T_k^*);C^\infty(\overline\Omega;\R^L)\big),
    \quad
    &&\bsig_k \in C^1\big([0,T_k^*);C^\infty(\overline\Omega;\R^M)\big), 
    \\
    &\v_k \in C^1\big([0,T_k^*);H^2(\Omega;\R^d)\big),
    \quad
    &&p_k \in C^1\big([0,T_k^*);H^1(\Omega)\big),
\end{aligned}
\end{align}
which satisfy the discretized weak formulation \eqref{DISC:WF} on the time interval $[0,T_k^*)$. 

In Step~3, we will see that the solution $(\mathbf a^k,\mathbf c^k)^\top$ of the ODE system mentioned above can actually be extended onto the whole time interval $[0,T]$. 
Then the functions $\bphi_k$, $\bmu_k$, $\bsig_k$, $\v_k$, and $p_k$ given by \eqref{GAL:FUNC} satisfy the discretized weak formulation \eqref{DISC:WF} not only on $[0,T_k^*)$ but on the whole time interval $[0,T]$.

\noindent
\textbf{Step 2: A priori estimates.}
Now, we intend to establish a priori estimates to bound our approximate solution $(\bphi_k,\bmu_k,\bsig_k,\v_k,p_k)$ uniformly with respect to the index $k$ in suitable norms. We claim that
there exist constants $C_\text{AP},C_{AP}'>0$ such that, for all $k\in\mathbb{N}$ and all $T_k<T_k^*$,
\begin{align}
    \label{EST:AP:1}
    &\begin{aligned}
    &\norm{\bphi_k}_{\Lk\infty{\HH^1} \cap \Lk2{\HH^3}}
    \notag\\
   &\qquad+\norm{\bsig_k}_{\Lk\infty{\LL^2} \cap \Lk2{\HH^1}}
    +\norm{\bsig_k}_{\Lk4{\LL^2_\Gamma}}
    \notag\\
    &\qquad+\norm{\bmu_k}_{\Lk2{\HH^1}}
    +\norm{\v_k}_{\Lk2{\LL^2}}
    + \bignorm{\sqrt{\eta(\bphi_k)}\,\D\v_k}_{L^2(0,T_k;\LL^2)}^2
    +\norm{p_k}_{\Lk{4/3}{L^2}}
    \notag\\
    &\qquad
    + \norm{\Psi_{\bphi}(\bphi_k)}_{ \Lk{4} {\LL^{2}}\cap\Lk{2} {\LL^{6}} }
    + \norm{\nabla \Ns(\bphi_k,\bsig_k)}_{ \Lk{2} {\LL^{2}}} 
    \notag\\
    &\qquad 
    +\norm{\Div(\bphi_k\otimes\v_k)}_{\Lk{4/3}{\LL^{3/2}}}
    +\norm{\Div(\bsig_k\otimes\v_k)}_{\Lk{1}{\LL^{1}}}
    \end{aligned}
    \notag\\
    &\quad \le C_\text{AP},
\end{align}
and
\begin{align}
    \label{EST:AP:2}
    & \norm{\bphi_k}_{{\Hk1{(\HH^1)'}}}
    + \norm{\bsig_k}_{\Wk{1,4/3}{(\HH^1)'}}
    + \norm{\v_k}_{\Lk2{\HH^1}}
    \notag\\
    &\qquad 
    +\norm{\Div(\bphi_k\otimes\v_k)}_{\Lk{2}{\LL^{3/2}}}
    +\norm{\Div(\bsig_k\otimes\v_k)}_{\Lk{1}{\LL^{3/2}} }
    \le C_{AP}'\big( 1 + \eta_0^{-1}\big).
\end{align}
We point out that the constants $C_{AP}$ and $C_{AP}'$ depend only on the initial data and the constants introduced in Section~\ref{SECT:ASS} except for $\eta_0$.
In particular, $C_{AP}$ and $C_{AP}'$ are thus independent of $k$ and $T_k$.

In the following proof of these estimates, we omit the subscript $k$ to provide a cleaner presentation. In particular, with some abuse of notation, we will also just write $T$ instead of $T_k$. 

\noindent
{\it Step 2.1: Energy estimate.}
To handle both cases \ref{ass:pot:1} and \ref{ass:pot:2} simultaneously, we introduce constants $\alpha$ and $\beta$ in the following way:
\begin{align}
\label{DEF:AB}
    \left\{
    \begin{aligned}
        & \alpha := 1, 
        && \beta := \theta_0
        && \text{if \ref{ass:pot:1} holds,}\\
        & \alpha := 0, 
        && \beta := 1
        && \text{if \ref{ass:pot:2} holds.}
    \end{aligned}
    \right.
\end{align}

For every $t\in[0,T]$, we choose $\bu(t)$ as a strong solution to the following problem
\begin{align*}
		\begin{cases}
		\Div \big(\bu(t)\big) = \Sv\big(\bphi(t),\bsig(t)\big) \quad  &\hbox{in $\Omega,$}
		\\ 
		\bu(t) = \frac 1 {|\Gamma|} \Big( \intO \Sv\big(\bphi(t),\bsig(t)\big) \dx \Big) {\bf n} = :{\bf r} \quad & \hbox{on $\Gamma,$ }
		\end{cases}
\end{align*}
whose solvability is a direct consequence of Lemma~\ref{LEM:diveq} with $f = \Sv(\bphi(t),\bsig(t))$.
This lemma further implies $\bu \in C^1([0,T];\WW^{1,q})$ for all $q \in (1,\infty)$ as well as the estimate
\begin{align}
    \label{normu}
	\norm{\bu(t)}_{\WW^{1,q}} \leq C\norm{\Sv\big(\bphi(t),\bsig(t)\big)}_{L^q} \le C.
\end{align}
Notice that condition \eqref{cond} in Lemma~\ref{LEM:diveq} is fulfilled as it holds that
\begin{align*}
	\intG {\bf r} \cdot {\bf n} \dS = \frac 1 {|\Gamma|}\Big( \intO \Sv\big(\bphi(t),\bsig(t)\big) \dx \Big)\intG {\bf n}\cdot {\bf n} \dS =  \intO \Sv\big(\bphi(t),\bsig(t)\big) \dx.
\end{align*}

We now recall from Assumption~\ref{ass:nutrient} that
\begin{align*}
    \Ns(\bphi,\bsig) = \chi_{\bsig} \bsig 
    -  \mathbb{B} \bphi - \mathbf{b}
    \quad\text{a.e. in $\Omega\times (0,T)$.}
\end{align*}
Since $\bphi(t)\in\CWK$ for all $t\in [0,T]$, it is straightforward to check that $\mathbb{B} \bphi(t) \in \CZK$ for all $t\in [0,T]$. Consequently, it holds that $\Ns(\bphi(t),\bsig(t)) \in \CZK$ for all $t\in [0,T]$.

Testing \eqref{DISC:WF:1} by $\v-\bu$, \eqref{DISC:WF:2} by $\bmu$, \eqref{DISC:WF:3} by $-\delt \bphi$, \eqref{DISC:WF:4} by $\Ns(\bphi,\bsig)$, adding the resulting equalities, and using the decompositions \eqref{DEC:SP}, \eqref{DEC:divphi} and \eqref{DEC:divsigma}, we infer the discrete energy identity
\begin{align}
	& \notag
	\ddt E(\bphi,\bsig)
	+ \intO 2 \eta(\bphi)|\D \v|^2+\nu|\v|^2 \dx
	\\ &\notag \;\; 
	+ \intO \CC(\bphi,\bsig) \Grad \bmu : \Grad \bmu 
	    +\DD(\bphi,\bsig) \Grad \N_{\bsig}(\bphi,\bsig) : \Grad \N_{\bsig}(\bphi,\bsig) \dx 
	 +  \intG  K \chi_{\bsig} |\bsig|^2 \dS
	\\ & \notag
	= \intO  \Lam_{\bphi}(\bphi,\bsig) \cdot \bmu - \Th_{\bphi}(\bphi,\bsig) \bmu \cdot \bmu
	    - \S_{\bsig}(\bphi,\bsig,\bmu) \cdot \N_{\bsig}(\bphi,\bsig) \dx
	\\ & \notag \;\;
	- \intO [(\nabla \bphi) \bu  + \bphi \Sv(\bphi,\bsig)] \cdot \bmu
	+ [(\nabla \bsig) \bu + \bsig \Sv(\bphi,\bsig)] \cdot \N_{\bsig}(\bphi,\bsig) \dx
	\\ & \;\;
	+ \intG  K \big(\Lam_\Gamma(\bphi,\bsig) \cdot \N_{\bsig}(\bphi,\bsig) + \bsig \cdot \G_{\bsig}(\bphi,\bsig)\big)\dS
	+ \intO 2 \eta(\bphi)\D \v : \D \bu + \nu \v \cdot \bu \dx
	\label{eq:energy}
\end{align}
on $[0,T]$, where the energy $E$ is given by \eqref{DEF:EN}.
We point out that $\Div (\v- \bu)=0$ is essential in the derivation of \eqref{eq:energy}. 

Using the parameter $\alpha$ introduced in \eqref{DEF:AB} and recalling the condition \eqref{COND:POSDEF} in \ref{ass:pot:1}, we derive the estimate
\begin{align}
	& \notag
	\ddt E(\bphi,\bsig)
	+ \intO 2 \eta(\bphi)|\D \v|^2+\nu|\v|^2 \dx
	+ \intO \alpha\, \theta_0 \abs{\bmu}^2 \dx
	\\ &\notag \;\; 
	+ \intO \CC(\bphi,\bsig) \Grad \bmu : \Grad \bmu 
	    +\DD(\bphi,\bsig) \Grad \N_{\bsig}(\bphi,\bsig) : \Grad \N_{\bsig}(\bphi,\bsig) \dx 
	 +  \intG  K \chi_{\bsig} |\bsig|^2 \dS
	\\ & \notag
	\le  \left| \intO  \Lam_{\bphi}(\bphi,\bsig) \cdot \bmu \dx \right| 
	+ (1-\alpha) \left| \intO \Th_{\bphi}(\bphi,\bsig) \bmu \cdot \bmu \dx \right|
	+ \left| \intO \S_{\bsig}(\bphi,\bsig,\bmu) \cdot \N_{\bsig}(\bphi,\bsig) \dx \right|
	\\ & \notag \;\;
	+ \left| \intO [(\nabla \bphi) \bu  + \bphi \Sv(\bphi,\bsig)] \cdot \bmu \dx \right|
	+ \left| \intO [(\nabla \bsig) \bu + \bsig \Sv(\bphi,\bsig)] \cdot \N_{\bsig}(\bphi,\bsig) \dx \right|
	\\ & \;\;
	+ \left|\intG  K \big(\Lam_\Gamma(\bphi,\bsig) \cdot \N_{\bsig}(\bphi,\bsig) + \bsig \cdot \G_{\bsig}(\bphi,\bsig)\big)\dS \right|
	+ \left|\intO 2 \eta(\bphi)\D \v : \D \bu + \nu \v \cdot \bu \dx \right|.
	\label{est:energy}
\end{align}
Using Young's and \Hol's inequalities, and \eqref{normu} with $q=2$, we infer that
\begin{align}
\label{EST:S1:1}
	&\left| \intO 2 \eta(\bphi)\D \v : \D \bu + \nu \v \cdot \bu \dx \right|
	\leq
	\bignorm{\sqrt{\eta(\bphi)}\, \D \v}_{\LL^2}^2 + \frac{\nu}{2}\norm{\v}_{\LL^2}^2  + C
	\quad\text{on $[0,T]$.}
\end{align}
To estimate the boundary integral, we employ the bounds on $\G_{\bsig}$ and $\N_{\bsig}$ demanded in \ref{ass:nutrient}, as well as the trace theorem to infer that
\begin{align}
\label{EST:S1:2}
	\left| \intG  K \big(\Lam_\Gamma(\bphi,\bsig) \cdot \N_{\bsig}(\bphi,\bsig) + \bsig \cdot \G_{\bsig}(\bphi,\bsig) \big)\dS \right| 
    \leq \frac{K\chi_{\bsig}}{2}\norm{\bsig}^2_{\LL^2_\Gamma} + C ( 1 + \norm{\bphi}^2_{\HH^1}
    )
\end{align}
on $[0,T]$.
Having this bound at our disposal, we next estimate the integrals on the \rhs\ of \eqref{est:energy} that depend on $\bmu$.
Recalling the decomposition \eqref{DEC:SP} for $\Sp$ presented in \ref{ass:sources}, we
use \Hol's and Young's inequalities along with \eqref{EST:LATH} to infer that
\begin{align}
    \label{EST:S1:3}
    \notag
    &
    \left|\intO \Lam_{\bphi}(\bphi,\bsig)  \cdot \bmu \dx \right|
    + (1-\alpha) \left| \intO \Th_{\bphi}(\bphi,\bsig) \bmu \cdot \bmu \dx \right|
    \\ 
    & \quad 
    \leq
    \tfrac 16 \beta \norm{\bmu}_{\LL^2}^2
    +  C(\norm{  \bphi}_{\LL^2}^2+\norm{  \bsig}_{\LL^2}^2+1). 
\end{align}
Proceeding similarly and using the decomposition \eqref{DEC:SS} from \ref{ass:sources} as well as the estimates \eqref{EST:DN} and \eqref{EST:SPS}, we further deduce that
\begin{align}
    \label{EST:S1:4}
    & \left|\intO \S_{\bsig}(\bphi,\bsig,\bmu) \cdot \N_{\bsig}(\bphi,\bsig) \dx \right| 
    \leq
    \tfrac 16 \beta  \norm{\bmu}_{\LL^2}^2
    + C (\norm{  \bphi}_{\LL^2}^2+\norm{\bsig}_{\LL^2}^2+1)
    \quad\text{on $[0,T]$.}
\end{align}
Since \eqref{normu} holds for $q=d+1$, and $W^{1,d+1}(\Omega; \R^d)$ is continuously embedded in $C_b(\Omega; \R^d)$, we know that
\begin{align}
\label{EST:S1:U}
    \bu \in C^1\big([0,T]; C_b(\Omega; \R^d)\big)
    \quad\text{with}\quad
    \norm{\bu(t)}_{\mathbf C_b} \le C 
    \quad\text{for all $t\in[0,T]$.}
\end{align}
Hence, using \ref{ass:nutrient} and \ref{ass:Sv}, as well as Young's and \Hol's inequalities, we infer that
\begin{align}
    \label{EST:S1:5}
    \left| \intO [(\nabla \bphi) \bu  + \bphi \Sv(\bphi,\bsig)] \cdot \bmu \dx \right|
    \leq 
    \tfrac 16 \beta \norm{\bmu}_{\LL^2}^2
    + C \norm{\bphi}_{\HH^1}^2.
\end{align}

Recalling \ref{ass:nutrient}, we use the chain rule to derive the identity
\begin{align}
\label{dec:nutrient}
    \nabla \bsig 
    = \big(\Nss(\bphi,\bsig)\big)^{-1} \big( \Grad\Ns(\bphi,\bsig) + \Gsp(\bphi,\bsig) \Grad \bphi \big)
    \quad\text{in $\Omega\times (0,T)$.}
\end{align}
According to \eqref{EST:INV:NSS} in \ref{ass:nutrient}, the operator norm of the matrix $(\Nss(\bphi,\bsig))^{-1}$ is bounded by $\chi_{\bsig}^{-1}$. We now use \eqref{EST:DDG} from \ref{ass:nutrient} to bound $\Gsp(\bphi,\bsig)$, which leads to the estimate
\begin{align}
    \label{EST:GRADSIG}
    \norm{\Grad\bsig}_{\LL^2} 
    \le \chi_{\bsig}^{-1}\big(\norm{\Grad\Ns(\bphi,\bsig)}_{\LL^2}
        + D_G \norm{\Grad\bphi}_{\LL^2} \big)
    \quad\text{on $[0,T]$.}
\end{align}
We can thus use \ref{ass:nutrient}, \ref{ass:Sv} and \eqref{EST:S1:U} along with Young's inequality to conclude that
\begin{align}
    \label{EST:S1:6}
    \notag
    &  	\left| \intO [(\nabla \bsig) \bu + \bsig \Sv(\bphi,\bsig)] \cdot \N_{\bsig}(\bphi,\bsig) \dx \right|
    \leq 
    \frac{D_0\chi_{\bsig}^2}{2} \norm{\Grad\bsig}_{\LL^2}^2 
    + C (\norm{\Ns(\bphi,\bsig)}_{\LL^2}^2 
    + \norm{\bsig}_{\LL^2}^2)
    \\
    &\quad\le
    \frac{D_0}{2} \norm{\Grad\Ns}_{\LL^2}^2 
    + C\big( 1 + \norm{\bphi}_{\HH^1}^2 + \norm{\bsig}_{\LL^2}^2 \big)
\end{align}
on the interval $[0,T]$. 

If \ref{ass:pot:2} holds (i.e., $\alpha=0$ and $\beta=1$), we 
still have to derive an estimate for the $\LL^2$-norm of $\bmu$ since it cannot be absorbed by the left-hand side.
We test \eqref{WF:3} with $\bmu$, and we use \eqref{EST:DN} and Young's inequality to infer that
\begin{align}
    \label{est:mu}
    \norm{ \bmu}_{\LL^2}^2
    & \leq 	
    C_0 \norm{ \nabla \bmu}_{\LL^2}^2
    + C \big( 1 + \norm{\bphi}_{\HH^1}^{2} + \norm{ \bsig}_{\LL^2}^2 \big)
    \quad\text{on $[0,T]$ if \ref{ass:pot:2} holds.}
\end{align}
We now combine the inequalities \eqref{EST:S1:1}--\eqref{EST:S1:4}, \eqref{EST:S1:5}, \eqref{EST:S1:6} and \eqref{est:mu} to estimate the right-hand side of the discrete energy identity \eqref{est:energy}. In the resulting inequality, we observe that several terms on the right-hand side can be absorbed by the left-hand side.
Recalling the definition of the energy $E$ and integrating with respect to time, we eventually obtain for all $t\in[0,T]$,
\begin{align}
    \label{EST:ABSORBED}
	& \notag
	\intO  \gamma\eps^{-1} \Psi\big(\bphi(t)\big) 
        + \frac {\gamma \eps}2 |\nabla \bphi(t)|^2  
        + \frac{\chi_{\bsig}}{2}|\bsig(t)|^2 
        - G\big(\bphi(t),\bsig(t)\big)\dx
    \\ &\notag \qquad 
	+ \int_0^t \bignorm{\sqrt{\eta(\bphi(s))}\,\D\v(s)}_{\LL^2}^2 
	    + \frac{\nu}{2} \norm{\v(s)}_{\LL^2}^2 
	    + \frac{K\chi_{\bsig}}{2} \norm{\bsig(s)}_{\LL^2_\Gamma}^2\ds
	\\ &\notag \qquad 
	+ \int_0^t 
	    \frac{\alpha\theta_0}{2}  \norm{\bmu(s)}_{\LL^2}^2
	    + \frac{C_0}{2} \norm{\Grad\bmu(s)}_{\LL^2}^2 
	    + \frac{D_0}{2} \norm{\Grad\Ns\big(\bphi(s),\bsig(s)\big)}_{\LL^2}^2
	  \ds
	\\ & \quad \le 
	CT
	+ C \int_0^t \norm{\bphi(s)}_{\HH^1}^2 
	    + \norm{\bsig(s)}_{\LL^2}^2 
	    \ds.
\end{align}
We next use the inequality \eqref{EST:G} from \ref{ass:nutrient} along with Young's inequality to derive the estimate
\begin{align*}
    \left|\intO G\big(\bphi(t),\bsig(t)\big) \dx\right|
    \le \intO \frac{C_G}{2\delta} \abs{\bphi(t)}^2 + 2C_G\delta\abs{\bsig(t)}^2 + C_G\left(\delta + \frac{1}{4\delta} +1 \right) \dx
\end{align*}
for all $t\in[0,T]$ and all $\delta>0$. Choosing $\delta = \chi_{\bsig} / (8 C_G )$, using the growth condition \eqref{superquad} from \ref{ass:pot}, and recalling \ref{ass:interface}, we infer that
\begin{align*}
    \notag
    \left|\intO G\big(\bphi(t),\bsig(t)\big) \dx\right|
    &\le \intO \frac{4C_G^2}{\chi_{\bsig}} \abs{\bphi(t)}^2 
    + \frac{\chi_{\bsig}}{4} \abs{\bsig(t)}^2 \dx
    + C
    \\
    &\le \intO \frac{\gamma}{2\eps} \Psi\big(\bphi(t)\big) 
    + \frac{\chi_{\bsig}}{4} \abs{\bsig(t)}^2 \dx
    + C
\end{align*}
for all $t\in[0,T]$. Invoking the growth condition \eqref{superquad} once more, we find that for all $t\in[0,T]$,
\begin{align*}
    \notag
    &\min\left\{ \frac{\gamma A_\Psi}{2\eps} , \frac{\chi_{\bsig}}{4}, \frac{\gamma\eps}{2} \right\} 
    \Big( \norm{\bphi(t)}_{\HH^1}^2 + \norm{\bsig(t)}_{\LL^2}^2 \Big) \\
    &\quad
    \le C + \intO  \gamma\eps^{-1} \Psi\big(\bphi(t)\big) 
        + \frac {\gamma \eps}2 |\nabla \bphi(t)|^2  
        + \frac{\chi_{\bsig}}{2}|\bsig(t)|^2 
        - G\big(\bphi(t),\bsig(t)\big)\dx.
\end{align*}
Using this estimate to bound the left-hand side in \eqref{EST:ABSORBED} from below, we finally conclude that
\begin{align}
    \label{EST:GRONWALL}
    \norm{\bphi(t)}_{\HH^1}^2 + \norm{\bsig(t)}_{\LL^2}^2
	\le 
	C
	+ C \int_0^t \norm{\bphi(s)}_{\HH^1}^2 
	    + \norm{\bsig(s)}_{\LL^2}^2 
	    \ds
\end{align}
for all $t\in[0,T]$.
Invoking Gronwall's lemma, we thus obtain the uniform estimate
\begin{align}
    \label{EST:UNI:1}
    \norm{\bphi}_{L^\infty(0,T;\HH^1)}^2 + \norm{\bsig}_{L^\infty(0,T;\LL^2)}^2 \le C.
\end{align}
Using this inequality to bound the right-hand side of \eqref{EST:ABSORBED}, invoking \eqref{EST:GRADSIG}, and additionally using \eqref{est:mu} if \ref{ass:pot:2} holds, we infer that
\begin{gather}
    \label{EST:UNI:2}
    \norm{\v}_{L^2(0,T;\LL^2)}^2 \le C,
    \quad
    \bignorm{\sqrt{\eta(\bphi)}\,\D\v}_{L^2(0,T;\LL^2)}^2
    \le C,
    \\
    \label{EST:UNI:3}
    \norm{\bsig}_{L^2(0,T;\HH^1)}^2
    + \norm{\bsig}_{L^2(0,T;\LL^2_\Gamma)}^2
    + \norm{\bmu}_{L^2(0,T;\HH^1)}^2
    + \norm{\Grad\Ns\big(\bphi,\bsig\big)}_{L^2(0,T;\LL^2)}^2
    \le C.
\end{gather}
Using the lower bound on $\eta$ from \ref{ass:visc} as well as Korn's inequality \eqref{Korn}, we directly conclude that
\begin{align}
    \label{EST:UNI:4}
    \norm{\v}_{L^2(0,T;\HH^1)}^2
    \le C\big( 1 + \eta_0^{-1} \big).
\end{align}
Moreover, applying a trace estimate presented in \cite[Thm.~II.4.1]{Galdi} (with the parameters therein being chosen as $r=q=2, m=1, n=d, \lambda=0$), we infer that
\begin{align}
    \label{interp:L2}
    \norm{\bsig}_{\LL^2_\Gamma} \leq C \big(\norm{\bsig}_{\LL^2} +\norm{\bsig}_{\LL^2}^{1/2}\norm{\bsig}_{\HH^1}^{1/2} \big)
    \quad\text{on $[0,T]$,}
\end{align}
and in combination with \eqref{EST:UNI:2}, this leads to the uniform bound
\begin{align}
    \label{EST:UNI:5}
    \norm{\bsig}_{\L 4{\LL^2_\Gamma}} \leq C .
\end{align}

{\it Step 2.2: An estimate for the pressure.}
We next want to derive a uniform estimate on the pressure $p$. To this end, we rewrite \eqref{WF:1} as
\begin{align}
\label{EQ:PDIVZ}
\begin{aligned}
    \intO  p \, \Div \be \dx
	& =
	\intO \big(2 \eta (\bphi) \D \v + \lambda (\bphi) \Sv(\bphi,\bsig) \I\big) : \nabla \be \dx
    \\ & \quad 
    +\intO \big(\nu  \v  
    -(\nabla \bphi)^\top  \bmu  -  (\nabla \bsig)^\top\Ns(\bphi,\bsig)
    \big)\cdot \be \dx
\end{aligned}
\end{align}
on $[0,T]$ for every $\be \in H^1(\Omega; \R^d)$.
Then, by invoking Lemma \ref{LEM:diveq}, we infer the existence of a function  ${\bf q}\in C([0,T];\HH^1) $  such that for all $t\in[0,T]$, ${\bf q}(t)$ is a strong solution to the system
\begin{align*}
	\begin{cases}
		\Div \big({\bf q}(t)\big) = p(t) \quad  &\hbox{in $\Omega,$}
		\\ 
		{{\bf q}(t)} = \frac {1}{|\Gamma|} \big(\intO p(t) \dx \big) {\bf n}\quad & \hbox{on $\Gamma.$ }
	\end{cases}
\end{align*}
We point out that the complementary condition \eqref{cond} is fulfilled as
\begin{align*}
    \intG {{\bf q}(t)} \cdot {\bf n}\dS = \frac {1}{|\Gamma|}\left(\intO p(t) \dx \right) \intG {\bf n}\cdot {\bf n} \dS = \intO p(t) \dx
    \quad\text{for all $t\in[0,T]$.}
\end{align*}
In particular, according to Lemma \ref{LEM:diveq}, we have the estimate
\begin{align}
    \norm{{\bf q}}_{\HH^1} \leq C \norm{p}_{L^2}
    \quad\text{on $[0,T]$.}
    \label{estimate:q}
\end{align}
Then, we choose $\be= {\bf q}$ in \eqref{EQ:PDIVZ} and, invoking \ref{ass:visc}, \ref{ass:nutrient} and \ref{ass:Sv}, using the uniform estimates \eqref{EST:UNI:1}, \eqref{EST:UNI:2}, \eqref{estimate:q} as well as \Hol's and Young's inequalities, we derive the estimate
\begin{align}
    \label{EST:P}
    \notag \norm{p}_{L^2}^2 & \leq 
    C \Big(\bignorm{\sqrt{\eta (\bphi)}\D \v}_{\LL^2}^2
    + \norm{\v}_{\LL^2}^2
    + \norm{\Sv(\bphi,\bsig)}_{L^2}^2 
    \\ 
    & \qquad\quad + \norm{\Grad\bphi}_{\LL^2}^2 \norm{\bmu}_{\LL^3}^2
    + \norm{\Grad\bsig}_{\LL^2}^2 \norm{\Ns(\bphi,\bsig)}_{\LL^3}^2
    \Big)
    \le C
\end{align}
on $[0,T]$. Recalling \ref{ass:nutrient} and the uniform estimate \eqref{EST:UNI:1}, we notice that
\begin{align*}
    \norm{\Ns(\bphi,\bsig)}_{\L\infty{\LL^2}} \le C.
\end{align*}
From the Gagliardo--Nirenberg inequality \eqref{GN}, we thus deduce the estimate
\begin{align}
\label{EST:NS}
    \norm{\Ns(\bphi,\bsig)}_{\L4{\LL^3} }^{4} 
    &\leq C \int_0^T  \norm{\Ns(\bphi,\bsig)}_{\LL^2}^{2}
    \norm{\Ns(\bphi,\bsig)}_{\HH^1}^{2} \dt \notag \\
    &\leq C \norm{\Ns(\bphi,\bsig)}_{\L\infty{\LL^2}}^{2}
    \norm{\Ns(\bphi,\bsig)}_{\L2{\HH^1}}^{2} 
    \leq C.
\end{align}
Combining \eqref{EST:P} and \eqref{EST:NS}, we eventually conclude the uniform bound
\begin{align}
    \label{EST:UNI:6}
    & \norm{p}_{L^{4/3}(0,T;L^2)}^{4/3} 
    \notag \\
    & \;\leq 
    C\int_0^T \Big(
    \bignorm{\sqrt{\eta (\bphi)}\D \v}_{\LL^2}^{4/3}
    + \norm{\v}_{\LL^2}^{4/3} +  \norm{\nabla \bphi}_{\LL^2}^{4/3}\norm{\bmu}_{\LL^3}^{4/3} + \norm{\nabla \bsig}_{\LL^2}^{4/3}\norm{\Ns(\bphi,\bsig)}_{\LL^3}^{4/3} +1 
    \Big) \dt
    \notag \\ 
    & \;\leq
     C \Big(
    \bignorm{\sqrt{\eta (\bphi)}\D \v}_{\L2{\LL^2}}^{4/3}
    + \norm{\v}_{\L2{\LL^2}}^{4/3} +  \norm{\nabla \bphi}_{\L4{\LL^2}}^{4/3} \norm{\bmu}_{\L2{\LL^3}}^{4/3}
    \notag \\
    & \;\qquad\quad
    + \norm{\nabla \bsig}_{\L2{\LL^2}}^{4/3}\norm{\Ns(\bphi,\bsig)}_{\L4{\LL^3}}^{4/3}  +1 
    \Big)
    \notag \\
    & \;
    \le C.
\end{align}

\noindent
{\it Step 2.3: Higher regularity for the phase-field.}
To establish higher order uniform a priori estimates on $\bphi$, we test \eqref{DISC:WF:3} with $-\Lap \bphi$ and integrate the resulting equation with respect to time. This is actually allowed since the basis functions $\w_i$, $ i=1,...,kL$ are contructed from eigenfunctions of the eigenvalue problem \eqref{PNEVP} and thus, $-\Lap\bphi(t)\in\CWK$ for all $t\in[0,T]$. Next, we integrate the resulting equation with respect to time and after further integrating by parts, we obtain
\begin{align*}
    \notag
    &\gamma\eps \int_0^T  \norm{\Lap\bphi}_{\LL^2}^2 \dt 
        + \gamma\eps^{-1} \int_0^T\intO \Psi^{(1)}_{\bphi\bphi}(\bphi) \Grad\bphi : \Grad\bphi \dxt \\
    &\quad = \int_0^T\intO - \bmu \cdot \Lap\bphi 
        + \gamma\eps^{-1} 
         \Psi^{(2)}_{\bphi}(\bphi) \cdot \Lap\bphi 
	    + \Np(\bphi,\bsig) \cdot \Lap\bphi \dxt .
\end{align*}
Note that the second integral on the left-hand side is nonnegative as $\Psi^{(1)}_{\bphi\bphi}(\bphi)$ is a positive definite matrix due to the convexity of $\Psi^{(1)}$. Applying Young's inequality on the right-hand side, using \eqref{EST:DN} from \ref{ass:nutrient}, and recalling that $\Psi^{(2)}_{\bphi}$ is Lipschitz continuous (see \ref{ass:pot}), we derive the estimate
\begin{align*}
    \frac{\gamma\eps}{2} \int_0^T \norm{\Delta\bphi}^2_{\LL^2} \dt
    \leq 
    C \int_0^T \big(1 + \norm{\bmu}_{\LL^2}^2 
        + \norm{\bphi}_{\LL^2}^2 
        + \norm{\bsig}_{\LL^2}^2 \big)
    \dt
    \le C.
\end{align*}
Invoking elliptic regularity theory and the uniform estimates \eqref{EST:UNI:1} and \eqref{EST:UNI:3}, we conclude that
\begin{align}
    \label{EST:UNI:12}
    \norm{\bphi}_{L^2(0,T;\HH^2)}
    \le C\big( \norm{\Delta\bphi}_{L^2(0,T;\LL^2)} 
        + \norm{\bphi}_{L^2(0,T;\LL^2)}\big)
    \le C.
\end{align}

We next test \eqref{DISC:WF:3} with $\Lap^2\bphi$ and integrate the resulting equation with respect to time. Arguing similarly as above, $\Lap^2\bphi$ is indeed an admissible test function due to the construction of the basis functions $\w_i$, $ i=1,...,kL$. After integrating by parts, using the chain rule, and recalling that $\Npp = -\Gpp, \Nsp = -\Gsp $ due to \ref{ass:nutrient}, we have

\begin{align}
    \label{EST:L2H3}
    \begin{split}
    \gamma\eps \int_0^T \norm{ \nabla \Lap \bphi }_{\LL^2}^2 \dt
	&= \int_0^T\intO 
	    \Big( \Grad \bmu : \Grad \Lap \bphi 
	    - \gamma\eps^{-1} \Psi_{\bphi\bphi}(\bphi) \Grad\bphi :  \Grad\Lap\bphi 
	\notag\\
	&\qquad\qquad\quad
	    - \Gpp(\bphi,\bsig) \Grad\bphi :  \Grad\Lap\bphi 
	    - \Gsp(\bphi,\bsig) \Grad\bsig :  \Grad\Lap\bphi \Big)
	\dx\dt
	\end{split}
	\notag\\
	&\le C \Big(
	    \norm{\Grad \bmu}_{\L2{\LL^2}}
	    + \norm{\Psi_{\bphi\bphi}(\bphi)}_{\L2{L^\infty}}
	        \norm{\Grad\bphi}_{\L\infty{\LL^2}}
	\notag\\
	&\qquad\quad
	    + \norm{\Grad\bphi}_{\L2{\LL^2}}
	    + \norm{\Grad\bsig}_{\L2{\LL^2}}
	\Big)
	\norm{\Grad\Lap\bphi}_{\L2{\LL^2}}.
\end{align}
If \ref{ass:pot:1} holds, we use Agmon's inequality \eqref{AG} to obtain the bound
\begin{align*}
    \norm{\Psi_{\bphi\bphi}(\bphi)}_{\L2{L^\infty}}^2
    &\le C + C\int_0^T \norm{\bphi}_{\LL^\infty}^4 \dt
    \le C + C\int_0^T \norm{\bphi}_{\HH^1}^2 \norm{\bphi}_{\HH^2}^2 \dt
    \notag\\
    &\le C + C \norm{\bphi}_{\L\infty{\HH^1}}^2 \norm{\bphi}_{\L2{\HH^2}}^2
    \le C.
\end{align*}
On the other hand, if \ref{ass:pot:2} holds, the bound $\norm{\Psi_{\bphi\bphi}(\bphi)}_{\L2{L^\infty}}\le C$ is trivially satisfied.
Applying Young's inequality on the right-hand side of \eqref{EST:L2H3}, we thus infer that
\begin{align*}
    \norm{ \nabla \Lap \bphi }_{\L2{\LL^2}}^2 
	&\le C 
	    \big(\norm{\bmu}_{\L2{\HH^1}}^2
	    + \norm{\bphi}_{\L\infty{\HH^1}}^2
	    + \norm{\bsig}_{\L2{\HH^1}}^2\big)
    \le C.
\end{align*}
Using elliptic regularity theory, as well as the uniform estimates \eqref{EST:UNI:1} and \eqref{EST:UNI:3}, we eventually obtain
\begin{align}
    \label{EST:UNI:13}
    \norm{\bphi}_{L^2(0,T;\HH^3)}
    \le C\big( \norm{\Grad\Delta\bphi}_{L^2(0,T;\LL^2)} 
        + \norm{\Grad\bphi}_{L^2(0,T;\LL^2)}\big)
    \le C.
\end{align}

\noindent
\textit{Step 2.4: Estimates for the gradient of the potential.}
Using interpolation (Lemma~\ref{LEM:INT}) and Sobolev's embedding theorem, we obtain the continuous embedding
\begin{align}
    \label{INT:PHI}
    \L\infty {\HH^1(\Omega)} \cap \L2 {\HH^3(\Omega)} \emb \L \kappa {\LL^{\frac {6 \kappa}{\kappa-8}}(\Omega)}
\end{align}
for all $\kappa \in (8,\infty)$.
This entails that
\begin{align*}
    \bphi \in \L{20} {\LL^{10}(\Omega)} \cap \L{10} {\LL^{30}(\Omega)}.
\end{align*}
Recalling \ref{ass:pot}, for general exponents $q$ and $r$ still to be chosen, we derive the estimate
\begin{align*}
    \norm{\Psi_{\bphi}(\bphi)}^q_{\L q {\LL^r}}
    & \leq C \int_0^T  \Big( \norm{ \,|\bphi|^{\rho -1}\, }^q_{{\LL^r}} +1 \Big) \dt
  \leq C \int_0^T \Big( \norm{ \bphi}^{q (\rho -1)}_{{\LL^{r(\rho -1)}}}  +1 \Big) \dt.
\end{align*}
Choosing $(q,r)=(4,2)$ and $(q,r)=(2,6)$, respectively, and recalling that $\rho\le 4$, we directly conclude that
\begin{align}
    \label{EST:UNI:14}
  \norm{\Psi_{\bphi}(\bphi)}_{  \L{4} {\LL^{2}}\cap\L{2} {\LL^{6}} } \leq C.
\end{align}

\noindent
{\it Step 2.5: Estimates for the convection terms.}
We now use the estimates established above to derive uniform bounds on the convection terms $\Div(\bphi\otimes\v)$ and $\Div(\bsig\otimes\v)$. 

Recalling the decompositions \eqref{DEC:divphi} and \eqref{DEC:divsigma}, and using \ref{ass:Sv}, Lemma~\ref{LEM:INT}, \Hol's inequality and the continuous embedding $\HH^1\emb \LL^6$, we infer the uniform bounds
\begin{align}
    \label{EST:UNI:7*}
     \norm{\Div(\bphi\otimes\v)}_{\L{4/3} {\LL^{3/2}} }^{4/3}
    &  \le C \int_0^T 
        \norm{\nabla \bphi}_{\LL^6}^{4/3} \norm{\v}_{\LL^2}^{4/3}
        + \norm{\bphi}_{\LL^{ 3/2}}^{4/3} \dt
   \notag \\
   & \le C \int_0^T 
        \norm{\bphi}_{\HH^1}^{2/3} 
        \norm{\bphi}_{\HH^3}^{2/3} 
        \norm{\v}_{\LL^2}^{4/3}
        + \norm{\bphi}_{\LL^{ 3/2}}^{4/3} \dt 
    \notag \\
    & \le C (\norm{\bphi}_{\L{\infty}{\HH^1}}^{2/3}
        \norm{\bphi}_{\L{2}{\HH^3}}^{2/3}
        \norm{\v}_{\L{2}{\LL^2}}^{4/3}+1)
    \le C,
    \\[1ex]
    \label{EST:UNI:8*}
     \norm{\Div(\bsig\otimes\v)}_{\L1 {\LL^1}}
    &  \le C \int_0^T \norm{\nabla \bsig}_{\LL^2} \norm{\v}_{\LL^2}
        + \norm{\bsig}_{\LL^{1}} \dt
    \le C,
    \\[1ex]
    \label{EST:UNI:7}
    \norm{\Div(\bphi\otimes\v)}_{\L2 {\LL^{3/2}}}^2
    &\le C \int_0^T \norm{\nabla \bphi}^{2}_{\LL^2} \norm{\v}^{2}_{\LL^6}
        + \norm{\bphi}^{2}_{\LL^{ 3/2}} \dt
    \le C\big( 1 + \eta_0^{-1} \big),
    \\[1ex]
    \label{EST:UNI:8}
    \norm{\Div(\bsig\otimes\v)}_{\L1 {\LL^{3/2}}}
    &\le C \int_0^T \norm{\nabla \bsig}_{\LL^2} \norm{\v}_{\LL^6}
        + \norm{\bsig}_{\LL^{ 3/2}} \dt
    \le C\big( 1 + \eta_0^{-1} \big)^{\frac 12}.
\end{align}

\noindent
{\it Step 2.6: Estimates for the time derivatives.}
Let $\bz\in L^2(0,T;H^1(\Omega;\R^L))$, $\bx\in L^4(0,T;H^1(\Omega;\R^M))$ be arbitrary, and let $\mathbb P\bz$ and $\mathbb P\bx$ denote the projected functions given by
\begin{align}
    \label{DEF:PZPX}
    \mathbb P\bz(x,t) = \big[\mathbb P_{\CWK}\big(\bz(t)\big)\big](x),
    \quad
    \mathbb P\bx(x,t) = \big[\mathbb P_{\CZK}\big(\bx(t)\big)\big](x)
    \quad \text{for almost all $(x,t)\in \Omega\times (0,T)$.}
\end{align}
We thus have the estimates
\begin{align}
    \label{EST:PZPX}
    \norm{\mathbb P\bz(t)}_{\HH^1} \le \norm{\bz(t)}_{\HH^1},
    \quad
    \norm{\mathbb P\bx(t)}_{\HH^1} \le \norm{\bx(t)}_{\HH^1}
\end{align}
for almost all $t\in[0,T]$. Moreover, since the families $\{\w_i\}_{i\in\mathbb N}$ and $\{\z_i\}_{i\in\mathbb N}$ are orthonormal bases of $L^2(\Omega;\R^L)$ and $L^2(\Omega;\R^M)$, respectively, we infer that 
\begin{align*}
    \inn{\delt\bphi}{\bz}_{\LL^2} 
    = \inn{\delt\bphi}{\mathbb P\bz}_{\LL^2},
    \quad
    \inn{\delt\bsig}{\bx}_{\LL^2} 
    = \inn{\delt\bsig}{\mathbb P\bx}_{\LL^2}
    \quad\text{on $[0,T]$.}
\end{align*}
Now we test \eqref{DISC:WF:2} with $\mathbb P\bz$ and we integrate with respect to time from $0$ to $T$. Using \ref{ass:tensors}, \ref{ass:sources}, \ref{ass:Sv}, the uniform estimates \eqref{EST:UNI:1}, \eqref{EST:UNI:3} and \eqref{EST:UNI:7}, and the continuous embedding $\LL^{\frac 32} \emb \Vp $, we derive the estimate
\begin{align*}
    &\big|{\ang{ \delt \bphi}{ \bz}_{L^2(0,T;\HH^1)}\big|}
    = \left| \int_0^T \inn{ \delt \bphi}{ \mathbb P\bz}_{\LL^2} \dt \right|
    \\ & \quad
    \leq C \int_0^T 
    \norm{\Div(\bphi\otimes\v)}_{\Vp}\norm{\mathbb P\bz}_{\HH^1}
    + \norm{\nabla \bmu}_{\LL^2}\norm{\mathbb P\bz}_{\LL^2}
    + \norm{\Sp(\bphi,\bsig, \bmu)}_{\LL^2}\norm{\mathbb P\bz}_{\LL^2}
    \dt
    \\ & \quad 
    \leq C \norm{\bz}_{L^2(0,T;\HH^1)}
    \left(\int_0^T 
        \norm{\Div(\bphi\otimes\v)}_{\LL^{3/2}}^2
        + \norm{\bmu}_{\HH^1}^2
        + \norm{\bphi}_{\LL^2}^2
        + \norm{\bsig}_{\LL^2}^2
        + 1 
    \dt \right)^{\frac 12}
    \\[1ex]&\quad
    \leq C \big( 1 + \eta_0^{-1}\big)^{\frac 12} \norm{\bz}_{L^2(0,T;\HH^1)}.
\end{align*}
Hence, taking the supremum over all $\bz\in L^2(0,T;\HH^1)$ with $\norm{\bz}_{L^2(0,T;\HH^1)}\le 1$, we obtain the uniform bound
\begin{align}
    \label{EST:UNI:10}
    \norm{\delt \bphi}_{\L2 {\Vp} }
    \leq C \big( 1 + \eta_0^{-1}\big)^{\frac 12}.
\end{align}

To estimate the time derivative of $\bsig$ we argue similarly. 
Namely, we test \eqref{DISC:WF:4} with $\mathbb P\bx$ and integrate with respect to time from $0$ to $T$.
Using \ref{ass:tensors}, \ref{ass:sources}, \ref{ass:Sv}, the uniform estimates \eqref{EST:UNI:1}, \eqref{EST:UNI:3} and \eqref{EST:UNI:7}, we obtain
\begin{align*}
    &\big| \ang{\delt\bsig}{\bx}_{L^{4}(0,T;\HH^1)} \big|
    = \left|\int_0^T \inn{\delt\bsig}{\mathbb P\bx}_{\LL^2} \dt  \right| 
    \\
    & \quad \leq 
    C \int_0^T
        \Big( \norm{\Div(\bsig\otimes\v)}_{\Vp}\norm{\mathbb P\bx}_{\HH^1}
        + \norm{\nabla \Ns(\bphi,\bsig)}_{\LL^2}\norm{\mathbb P\bx}_{\LL^2}
    \\ & \hspace{75pt}
        + \norm{\Ss(\bphi,\bsig,\bmu)}_{\LL^2}\norm{\mathbb P\bx}_{\LL^2}
        + \norm{\SG(\bphi,\bsig)}_{\LL^2_\Gamma} \norm{\bx}_{\LL^2_\Gamma} \Big)
    \dt
    \\
    & \quad \leq 
    C \norm{\bx}_{L^4(0,T;\HH^1)} 
    \left( \int_0^T
        \norm{\Div(\bsig\otimes\v)}_{\Vp}^{\frac 43}
        + \norm{\nabla \Ns(\bphi,\bsig)}_{\LL^2}^{\frac 43} \right.
    \\ & \hspace{110pt} 
        \left.\phantom{\int_0^T}
        + \norm{\Ss(\bphi,\bsig,\bmu)}_{\LL^2}^{\frac 43}
        + \norm{\SG(\bphi,\bsig)}_{\LL^2_\Gamma}^{\frac 43} 
    \dt \right)^{\frac 34}
    \\
    & \quad \leq 
    C\big( 1 + \eta_0^{-1} \big)^{\frac 12} \norm{\bx}_{L^4(0,T;\HH^1)}.
\end{align*}
Taking the supremum over all $\bx\in L^4(0,T;\HH^1)$ with $\norm{\bx}_{L^4(0,T;\HH^1)}\le 1$ we eventually get
\begin{align}
    \label{EST:UNI:11}
    \norm{\delt \bsig}_{\L{4/3} {\Vp} }
    \leq C\big( 1 + \eta_0^{-1} \big)^{\frac 12}.
\end{align}

\noindent
\textbf{Step 3: Extension onto the whole time interval $[0,T]$.}
As the constant $C_{AP}$ is independent of the time $T_k$, we will use the a priori estimate \eqref{EST:AP:1} to extend the approximate solution $(\bphi_k,\bmu_k,\bsig_k,\v_k,p_k)$ onto the whole time interval $[0,T]$. 
To see this, we recall from Step~1 that the coefficients $(\mathbf a^k,\mathbf c^k)$ are determined as a solution of a nonlinear ODE system. Using \eqref{GAL:ANSATZ}, we infer that for any $T_k\in[0,T_k^*)$, all $t\in [0,T_k]$, and all $i\in\{1,...,kL\}$, and $j\in\{1,...,kM\}$,
\begin{align*}
    |a_i^k(t)| + |c_j^k(t)| 
    &= \abs{\inn{\bphi_k(t)}{\w_i}_{\LL^2}} 
    + \abs{\inn{\bsig_k(t)}{\z_j}_{\LL^2}} \\
    &\le \norm{\bphi_k(t)}_{L^\infty(0,T_k;\LL^2)} 
    + \norm{\bsig_k(t)}_{L^\infty(0,T_k;\LL^2)} 
    \le C_{AP}.
\end{align*}
This means that the solution $(\mathbf a^k,\mathbf c^k)^\top$ is bounded on the time interval $[0,T_k^*)$ and hence, it can be extended beyond $T_k^*$. However, as $(\mathbf a^k,\mathbf c^k)^\top$ was assumed to be a right-maximal solution, this is a contradiction. We thus conclude that the solution $(\mathbf a^k,\mathbf c^k)^\top$ actually exists on the whole time interval $[0,T]$. As the coefficients $\mathbf b^k$ can be reconstructed from $(\mathbf a^k,\mathbf c^k)^\top$ via the vector-valued algebraic equation mentioned in Step~1, we further infer that $\mathbf b^k$ also exists on $[0,T]$. This directly implies that the functions $\bphi_k$, $\bmu_k$, $\bsig_k$, $\v_k$ and $p_k$ exist on $[0,T]$ and satisfy the discretized weak formulation \eqref{DISC:WF} on $[0,T]$. As the choice of $T_k$ did not play any role in the proof of the a priori estimates presented in Step~2, it is clear that the a priori estimate \eqref{EST:AP:1} holds true with $T_k$ and $T_k^*$ being replaced by the final time $T$. 

\noindent
\textbf{Step 4: Convergence to a weak solution.}
Exploiting the a priori estimates in Step~2,
and using Sobolev's embedding theorem and interpolation (Lemma~\ref{LEM:INT}),
we conclude that there exists a quintuplet $(\bphi,\bmu,\bsig,\v,p)$ 
such that the sequence of approximate solutions $\{(\bphi_k,\bmu_k,\bsig_k,\v_k,p_k)\}_{k\in\mathbb N}$ satisfies 
\begin{subequations}
\label{CONV:WS}
\begin{alignat}{2}
    \label{CONV:PHI}
    \bphi_k &\to \bphi
    &&\quad\text{weakly in $ H^1(0,T;(\HH^1)')\cap L^2(0,T;\HH^3)  $,} \notag\\
    &&&\qquad\text{weakly-$^*$ in $L^\infty(0,T;\HH^1)$, a.e.~in $Q$,} \notag \\
    &&&\qquad\text{and strongly in $C([0,T];\HH^s) \cap L^2(0,T;\HH^{2+s})$ for all $s\in [0,1)$,} \\
    \label{CONV:PHI:G}
    \bphi_k\vert_\Gamma &\to \bphi\vert_\Gamma
    &&\quad\text{strongly in $C([0,T];\LL^2_\Gamma)$,
        and a.e.~on $\Sigma$,}\\
    \label{CONV:SIGMA}
    \bsig_k &\to \bsig
    &&\quad\text{weakly in $W^{1,\frac 43}(0,T;(\HH^1)') \cap  L^2(0,T;\HH^1) $}, \notag\\
    &&&\qquad\text{weakly-$^*$ in $L^\infty(0,T;\LL^2)$, a.e.~in $Q$,} \notag\\
    &&&\qquad\text{and strongly in $C([0,T];(\HH^1)') \cap L^2(0,T;\HH^s)$ for all $s\in [0,1)$,} \\
    \notag
    \bsig_k\vert_\Gamma &\to \bsig\vert_\Gamma
    &&\quad \text{weakly in $L^4(0,T;\LL^2_\Gamma)$, strongly in $L^2(0,T;\LL^2_\Gamma)$,}\\
    &&&\qquad\text{{and a.e.~on} $\Sigma$,}
     \label{CONV:SIGMA:G}
     \\
    \label{CONV:MU}
    \bmu_k &\to \bmu
    &&\quad\text{weakly in $L^2(0,T;\HH^1)$},\\
    \label{CONV:V}
    \v_k &\to \v
    &&\quad\text{weakly in $L^2(0,T;\HH^1)$}, \\
    \label{CONV:P}
    p_k &\to p
    &&\quad\text {weakly in $L^{\frac 43} (0,T;L^2)$},
\end{alignat}
\end{subequations}
as $k\to\infty$, along a nonrelabeled subsequence. We point out that 
the strong convergences in \eqref{CONV:PHI} and \eqref{CONV:SIGMA} are a direct consequence of the Aubin--Lions--Simon lemma (cf.~\cite[Theorem~II.5.16]{boyer}). Then
the strong convergences in \eqref{CONV:PHI:G} and \eqref{CONV:SIGMA:G} follow from \eqref{CONV:PHI} and \eqref{CONV:SIGMA} by means of the trace theorem.
In particular, this entails that the {limit} $(\bphi,\bmu,\bsig,\v,p)$ satisfies the regularity condition \eqref{REG:MCHB}, and we further know that $\bphi\in L^2(0,T;\HH^3)$.
Recalling the assumptions \ref{ass:tensors}--\ref{ass:source:boundary}, we infer from \eqref{CONV:WS} the almost everywhere convergence properties
\begin{subequations}
\label{CONV:AE}
\begin{alignat}{3}
    &\CC(\bphi_k,\bsig_k) \to \CC(\bphi,\bsig),
    &&\quad \DD(\bphi_k,\bsig_k) \to \DD(\bphi,\bsig)
    &&\quad\text{a.e.~in $Q$,}\\
    &\eta(\bphi_k) \to \eta(\bphi),
    &&\quad \lambda(\bphi_k) \to \lambda(\bphi)
    &&\quad\text{a.e.~in $Q$,}\\
    \label{CONV:AE:3}
    &\Np(\bphi_k,\bsig_k) \to \Np(\bphi,\bsig),
    &&\quad \Ns(\bphi_k,\bsig_k) \to \Ns(\bphi,\bsig)
    &&\quad\text{a.e.~in $Q$,}\\
    \label{CONV:AE:4}
    &\Grad\Ns(\bphi_k,\bsig_k) \to \Grad\Ns(\bphi,\bsig),
    &&\quad \Sv(\bphi_k,\bsig_k) \to \Sv(\bphi,\bsig)
    &&\quad\text{a.e.~in $Q$,}\\
    &\Psi_{\bphi}(\bphi_k) \to \Psi_{\bphi}(\bphi)
    &&
    &&\quad\text{a.e.~in $Q$,}\\
    &\SG(\bphi_k,\bsig_k) \to \SG(\bphi,\bsig)
    &&&&\quad\text{a.e.~on $\Sigma$,}\\
    &\Lam_{\bphi}(\bphi_k,\bsig_k) \to \Lam_{\bphi}(\bphi,\bsig),
    &&\quad \Th_{\bphi}(\bphi_k,\bsig_k) \to \Th_{\bphi}(\bphi,\bsig)
    &&\quad\text{a.e.~in $Q$,}\\
    &\Lam_{\bsig}(\bphi_k,\bsig_k) \to \Lam_{\bsig}(\bphi,\bsig),
    &&\quad \Th_{\bsig}(\bphi_k,\bsig_k) \to \Th_{\bsig}(\bphi,\bsig)
    &&\quad\text{a.e.~in $Q$,}
\end{alignat}
\end{subequations}
after another subsequence extraction.
From \eqref{CONV:AE:3}, \eqref{CONV:AE:4} and the a priori estimate \eqref{EST:AP:1}, we further conclude that, 
\begin{alignat}{2}
\label{CONV:PSI}
    &\Psi_{\bphi}(\bphi_k) \to \Psi_{\bphi}(\bphi),
    &&\quad\text{weakly in $L^4(0,T;\LL^2) \cap L^2(0,T;\LL^6)$,}\\
\label{CONV:NS}
    &\Grad\Ns(\bphi_k,\bsig_k) \to \Grad\Ns(\bphi,\bsig)
    &&\quad\text{weakly in $L^2(0,T;\LL^2)$,}
\end{alignat}
as $k\to\infty$, up to subsequence extraction.
Using the decomposition \eqref{DEC:divphi}, and the convergences \eqref{CONV:PHI}, \eqref{CONV:V}, and \eqref{CONV:AE:4}, it is straightforward to check that
\begin{align}
    \label{CONV:DIV:PHIV}
    \Div(\bphi_k\otimes \v_k) \to \Div(\bphi\otimes \v) \quad\text{weakly in $L^2(0,T;\LL^{\frac 32})$}
\end{align}
as $k\to\infty$. Moreover, due to the a priori estimate \eqref{EST:AP:2}, the Banach--Alaoglu theorem implies that there exists a function $\ttau  \in L^1(0,T;\LL^{\frac 32})$ such that
\begin{align*}
    \Div(\bsig_k\otimes \v_k) \to \ttau  \quad\text{weakly in $L^1(0,T;\LL^{\frac 32})$}.
\end{align*}
Let now $\bx\in C^\infty_c(Q)$ be an arbitrary test function. Performing an integration by parts, we obtain
\begin{align*}
    \intQ \Div(\bsig_k\otimes \v_k) \cdot \bx \dxt 
    = - \intQ (\bsig_k \otimes \v_k) : \Grad\bx \dxt .
\end{align*}
Due to \eqref{CONV:SIGMA} and \eqref{CONV:V}, we may pass to the limit on the right-hand side. This yields
\begin{align*}
    \intQ \Div(\bsig_k\otimes \v_k) \cdot \bx \dxt 
    &\to - \intQ (\bsig  \otimes \v) : \Grad\bx \dxt 
\end{align*}
and after another integration by parts, we conclude that
$\Div(\bsig_k\otimes \v_k) \to \Div(\bsig\otimes \v)$
as $k\to\infty$ in the sense of distributions. Because of uniqueness of the limit, we thus have $\ttau = \Div(\bsig\otimes \v)$ almost everywhere in $Q$ and hence,
\begin{align}
    \label{CONV:DIV:SIGV}
    \Div(\bsig_k\otimes \v_k) \to \Div(\bsig\otimes \v)  \quad\text{weakly in $L^1(0,T;\LL^{\frac 32})$}.
\end{align}

Now, let $\delta\in C^\infty([0,T])$ and $\be \in H^1(\Omega;\R^d)$ be arbitrary, and for any fixed $k\in \mathbb{N}$, let  $i\in\{1,...,kL\}$, and $j\in\{1,...,kM\}$ be arbitrary. We test the discretized weak formulation \eqref{DISC:WF} with $\delta\be$, $\delta\w_i$ and $\delta\z_j$ and integrate with respect to time from $0$ to $T$. This yields
\begin{subequations}
\label{DISC:WFT}
\begin{align}
    \label{DISC:WFT:1}
	&\intQ  \T(\bphi_k,\v_k, p_k) : \delta\nabla \be + \nu  \v_k \cdot \delta\be  \dxt
	=
    \intQ  (\nabla \bphi_k)^\top  \bmu_k \cdot \delta\be 
    +  (\nabla \bsig_k)^\top\Ns(\bphi_k,\bsig_k) \cdot \delta\be \dxt,
	\\
	& \notag
	\int_0^T \ang{\delt \bphi_k}{\w_i}_{\HH^1} \delta \dt  
	+ \intQ \Div(\bphi_k\otimes\v_k) \cdot \delta\w_i \dxt
	\\ & \quad \label{DISC:WFT:2}
	= - \intQ \CC(\bphi_k,\bsig_k) \nabla \bmu_k : \delta \nabla \w_i + \Sp(\bphi_k,\bsig_k,\bmu_k) \cdot \delta\w_i
	\dxt,
	\\ & \label{DISC:WFT:3}
	\intQ \bmu_k \cdot \delta\w_i \dxt 
	= \intQ \gamma\eps \nabla \bphi_k : \delta \nabla \w_i 
	+ \gamma\eps^{-1} \Psi_{\bphi}(\bphi_k) \cdot \delta \w_i
	+ \Np(\bphi_k,\bsig_k) \cdot \delta \w_i \dxt,
	\\
	& \notag
	\int_0^T \ang{\delt \bsig_k}{\z_j}_{\HH^1} \delta\dt 	
	+ \intQ \Div(\bsig_k\otimes\v_k) \cdot \delta\z_j \dxt
	\\
	& \notag \quad
	= - \intQ \DD(\bphi_k,\bsig_k) \nabla \Ns(\bphi_k,\bsig_k): \delta \Grad \z_j \dxt
	\\ & \quad\qquad \label{DISC:WFT:4}
	- \intQ \Ss(\bphi_k,\bsig_k,\bmu_k) \cdot \delta\z_j \dxt
	+ \intS  \SG(\bphi_k,\bsig_k)\cdot \delta\z_j \dS\dt.
\end{align}
\end{subequations}
Invoking the convergence properties \eqref{CONV:WS} and \eqref{CONV:AE}, and using Lebesgue's dominated convergence theorem, we deduce that
\begin{subequations}
\label{CONV:LCT}
\begin{alignat}{3}
    &\eta(\bphi_k) \delta\Grad\be 
    \to \eta(\bphi) \delta\Grad\be,
    &&\quad\lambda(\bphi_k)\delta\Grad\be 
    \to \lambda(\bphi)\delta\Grad\be
    &&\quad\text{in $\LL^2(Q)$,}\\
    &\CC(\bphi_k,\bsig_k) \delta\Grad\w_i 
    \to \CC(\bphi,\bsig) \delta\Grad\w_i,
    &&\quad\DD(\bphi_k,\bsig_k) \delta\Grad\z_j 
    \to \DD(\bphi,\bsig) \delta\Grad\z_j 
    &&\quad\text{in $\LL^2(Q)$,}\\
    &\Th_{\bphi}(\bphi_k,\bsig_k) ^\top \delta\w_i 
    \to \Th_{\bphi}(\bphi,\bsig)^\top \delta\w_i,
    &&\quad\Th_{\bsig}(\bphi_k,\bsig_k) ^\top \delta\z_j 
    \to \Th_{\bsig}(\bphi,\bsig)^\top \delta\z_j 
    &&\quad\text{in $L^2(Q)$,}
\end{alignat}
\end{subequations}
as $k\to\infty$.
To establish a similar convergence result for terms like $[\Ns(\bphi_k,\bsig_k)]_j [\delta\be]_i$ for all $i\in\{1,...,d\}$ and $j\in\{1,...,M\}$, we intend to employ a generalized version of Lebesgue's dominated convergence theorem, see \cite[Sec.~3.25]{Alt}. To this end, for any $i\in\{1,...,d\}$ and $j\in\{1,...,M\}$, we first recall that
\begin{alignat*}{2}
    [\Ns(\bphi_k,\bsig_k)]_j [\delta\be]_i
    &\to [\Ns(\bphi,\bsig)]_j [\delta\be]_i
    &&\quad\text{a.e. in $Q$ as $k\to \infty$}, \\
     \big|[\Ns(\bphi_k,\bsig_k)]_j [\delta\be]_i\big|^2 
    &\le B_N^2\big(\abs{\bphi_k}^2 + \abs{\bsig_k}^2 + 1\big)\abs{\delta\be}^2 =: g_k
    &&\quad\text{a.e. in $Q$ for all $k\in\mathbb N$},
\end{alignat*}
due to \eqref{CONV:AE:3} and \ref{ass:nutrient}. Using the convergences in \eqref{CONV:PHI} and \eqref{CONV:SIGMA}, a straightforward computation reveals that
\begin{align*}
    g_k \to g 
    := B_N^2\big(\abs{\bphi}^2 
        + \abs{\bsig}^2 + 1\big)\abs{\delta\be}^2
    \quad\text{in $L^1(Q)$}.
\end{align*}
Hence, we apply Lebesgue's generalized convergence theorem to conclude that 
\begin{subequations}
\label{CONV:LGCT}
\begin{align}
    [\Ns(\bphi_k,\bsig_k)]_j [\delta\be]_i
    \to [\Ns(\bphi,\bsig)]_j [\delta\be]_i
    \quad\text{in $L^2(Q)$}
\end{align}
as $k\to\infty$, for all $i\in\{1,...,d\}$ and $j\in\{1,...,M\}$. 
Proceeding similarly, we further obtain the following convergences:
\begin{alignat}{2}
    \Np(\bphi_k,\bsig_k)\cdot\delta\w_i
    &\to \Np(\bphi,\bsig)\cdot\delta\w_i
    &&\quad\text{in $L^2(Q)$,}\\
    \Lam_{\bphi}(\bphi_k,\bsig_k) \cdot\delta\w_i 
    &\to \Lam_{\bphi}(\bphi,\bsig) \cdot\delta\w_i
    &&\quad\text{in $L^2(Q)$,}\\
    \Lam_{\bsig}(\bphi_k,\bsig_k) \cdot\delta\z_j
    &\to \Lam_{\bsig}(\bphi,\bsig) \cdot\delta\z_j
    &&\quad\text{in $L^2(Q)$,}\\
    \Lam_{\Gamma}(\bphi_k,\bsig_k) \cdot\delta \z_j
    &\to \Lam_{\Gamma}(\bphi,\bsig) \cdot\delta\z_j
    &&\quad\text{in $L^2(\Sigma)$,}
\end{alignat}
\end{subequations}
for all $i\in\{1,...,kL\}$ and $j\in\{1,...,kM\}$.

Eventually, invoking the convergences \eqref{CONV:WS}, \eqref{CONV:PSI}--\eqref{CONV:DIV:SIGV}, \eqref{CONV:LCT} and \eqref{CONV:LGCT}, we may pass to the limit in \eqref{DISC:WFT}. As the test function $\delta$ and the indices $i\in\{1,...,kL\}$, and $j\in\{1,...,kM\}$ can be chosen arbitrarily, we conclude by means of a diagonal argument that the quintuplet 
$(\bphi,\bmu,\bsig,\v,p)$ satisfies the weak formulation \eqref{WF} for all test functions $\be\in H^1(\Omega;\R^d)$, $\bz = \w_i$, $\bt = \w_i$, $\bx = \z_j$ with $i,j\in\mathbb N$. Next, we recall that the families $\{\w_i\}_{i\in\mathbb N}$ and $\{\z_j\}_{j\in\mathbb N}$ are Schauder bases of $H^2_\n(\Omega;\R^L)$ and $H^2_\n(\Omega;\R^M)$, respectively. Since $H^2_\n(\Omega;\R^L)$ is dense in $H^1(\Omega;\R^L)$, and $H^2_\n(\Omega;\R^M)$ is dense in $H^1(\Omega;\R^M)$, we eventually conclude that the weak formulation \eqref{WF} is actually satisfied for all test functions $\be\in H^1(\Omega;\R^d)$, $\bz,\bt\in H^1(\Omega;\R^L)$ and $\bx\in H^1(\Omega;\R^M)$.
Moreover, the identities
\begin{alignat*}{2}
    \Div(\v) &= \Sv(\bphi,\bsig) 
    &&\quad \text{a.e. in $Q$,}\\
    \bphi\vert_{t=0} &= \bphi_0
    &&\quad \text{a.e. in $\Omega$,}\\
    \ang{\bsig\vert_{t=0}}{\Phi}_{\HH^1} &= \ang{\bsig_0}{\Phi}_{\HH^1}
    &&\quad \text{for all $\Phi\in H^1(\Omega; \R^M)$}
\end{alignat*}
follow directly from the convergences stated in \eqref{CONV:WS} and the uniqueness of the limit. This proves that the quintuplet $(\bphi,\bmu,\bsig,\v,p)$ is indeed a weak solution to the multiphase Cahn--Hilliard--Brinkman system \eqref{MCHB} in the sense of Definition~\ref{DEF:WEAKSOL}. 

\noindent
\textbf{Step 5: Further properties.}
We will now establish the remaining properties of the weak solution constructed in Step~4.

Using the convergences \eqref{CONV:WS} and the weak lower semicontinuity of the norms, we infer from the a priori estimate \eqref{EST:AP:1} that
\begin{align}
    \label{EST:APOST}
    &\norm{\bphi}_{\L\infty{\HH^1} \cap \L2{\HH^3}}
    +\norm{\bsig}_{\L\infty{\LL^2} \cap \L2{\HH^1}}
    +\norm{\bsig}_{\L4{\LL^2_\Gamma}}
    \notag\\
    &\quad
    +\norm{\bmu}_{\L2{\HH^1}}
    +\norm{\v}_{\L2{\LL^2}}
    +\bignorm{\sqrt{\eta(\bphi)}\, \D \v}_{\L2{\LL^2}}
    +\norm{p}_{\L{4/3}{L^2}}
    \notag\\
    &\quad
    + \norm{\Psi_{\bphi}(\bphi)}_{ \L{4} {\LL^{2}}\cap\L{2} {\LL^{6}} }
    + \norm{\nabla \Ns(\bphi,\bsig)}_{ \L{2} {\LL^{2}}} 
    \notag\\
    &\quad 
    +\norm{\Div(\bphi\otimes\v)}_{\L{4/3}{\LL^{3/2}}}
    +\norm{\Div(\bsig\otimes\v)}_{\L{1}{\LL^{1}}}
   \le C_{AP}.
\end{align}
In particular, this means that the second regularity in \eqref{REG:MCHB:ADD} is already established.
Furthermore, in Step~4 have already shown that
\begin{align*}
    \bphi \in H^1\big(0,T;(\HH^1)'\big) \cap L^2(0,T;\HH^3)
    \subset H^1\big(0,T;\HH^{-1}\big) \cap L^2(0,T;\HH^3).
\end{align*}
Invoking a result from interpolation theory \cite[Thm.~4.10.2]{amann} as well as Lemma~\ref{LEM:INT}, we conclude that
\begin{align*}
    \bphi \in 
    C\Big([0,T];(\HH^{-1},\HH^3)_{\frac 12, 2}\Big)
    = C\big([0,T];\HH^1\big)
\end{align*}
which proves the first regularity in \eqref{REG:MCHB:ADD}. 
\end{proof}

%
%

\section{``Darcy limit'' and existence of weak solutions to the (MCHD) system} \label{SECT:EX:MCHD}

This section is devoted to the construction of a weak solution to the multiphase Cahn--Hilliard--Darcy system \eqref{MCHB}
in the sense of Definition \ref{DEF:WEAKSOL:DARCY}. This is achieved by an asymptotic technique, where the positive viscosity functions in the system (MCHB) are sent to zero.

\begin{proof}[Proof of Theorem \ref{THM:EXISTENCE:WEAK:DARCY}]
For every $n\in\mathbb N$, let $\eta_n$ and $\lambda_n$ be viscosity functions as described in Theorem~\ref{THM:EXISTENCE:WEAK:DARCY}, and let
$(\bphi_n, \bmu_n, \bsig_n, \v_n, p_n)$ denote a weak solution of the Cahn--Hilliard--Brinkman system \eqref{MCHB} obtained from Theorem~\ref{THM:EXISTENCE:WEAK} with the choices $\eta= \eta_n$ and $\lambda = \lambda_n$. 
We point out that by this explicit choice, we do not require the axiom of choice, even though the uniqueness of the weak solutions is unknown.
We recall that, owing to Theorem \ref{THM:EXISTENCE:WEAK}, the solutions
$(\bphi_n, \bmu_n, \bsig_n, \v_n, p_n) $ satisfy the weak formulation \eqref{WF:1}--\eqref{WF:4} (written for $\eta= \eta_n$ and $\lambda = \lambda_n$) and exhibit the regularities
\begin{align}
\label{REG:MCHB:N}
\left\{\;
\begin{aligned}
	&\bphi_n \in  \H1 {(\HH^1)'} \cap C([0,T];\HH^1)
	    \cap \L2 {\HH^3},
	\\ 
	&\bsig_n \in \W{1,\frac 43} {(\HH^1)'} 
	    \cap  C([0,T];(\HH^1)') 
	    \cap \L\infty{\LL^2}
	    \cap \L2 {\HH^1}	    ,
	\\
	&\bphi_n\vert_\Gamma \in C([0,T];\LL^2_\Gamma),
	\quad \bsig_n\vert_\Gamma \in \L4{\LL^2_\Gamma} ,
	\\
	&\bmu_n \in  \L2 {\HH^1},
	\quad
	\v_n \in  \L{2} {\HH^1},
	\quad
	p_n \in  \L{\frac 43}  {L^2},
	\\
	&\Div(\bphi_n\otimes\v_n) \in L^2(0,T;\LL^{\frac 32}),
	\quad \Div(\bsig_n\otimes\v_n) \in L^1(0,T;\LL^{\frac 32}).
\end{aligned}
\right.
\end{align}
Since for any fixed $n\in\mathbb N$, the viscosities $\eta_n$ and $\lambda_n$ are assumed to be compatible with \ref{ass:visc}, there exist constants $0<\eta_{0,n}<\eta_{1,n}$ and $\lambda_{*,n}>0$ such that \eqref{EST:VISC} is fulfilled.
In view of \eqref{darcy:ass:viscosities}, we assume (without loss of generality) that $\eta_{1,n} = \lambda_{*,n} = 1$ for all $n\in\mathbb N$ and we further fix
\begin{align*}
    \eta_{0,n} := \underset{\p\in\R^L}{\inf} \eta_n(\p).
\end{align*}

To investigate the convergence of the sequence $\{(\bphi_n, \bmu_n, \bsig_n, \v_n, p_n)\}_{n\in\mathbb N}$, we first need to derive suitable bounds that are uniform in $n$.

\noindent
\textbf{Step 1: Uniform estimates.}
In the following, the letter $C$ denotes generic positive constants that do not depend on $n$. They may still depend on the initial data and the other constants from Section~\ref{SECT:ASS}, except for $\eta_{0,n}$.
We already know from Theorem~\ref{THM:EXISTENCE:WEAK} that
\begin{align}
    \label{EST:UNI:D1}
    &\norm{\bphi_n}_{\L\infty{\HH^1} \cap \L2{\HH^3}}
    +\norm{\bsig_n}_{\L\infty{\LL^2} \cap \L2{\HH^1}}
    +\norm{\bsig_n}_{\L4{\LL^2_\Gamma}}
    \notag\\
    &\quad
    +\norm{\bmu_n}_{\L2{\HH^1}}
    +  \norm{\v_n}_{\L{2}{\LL^2}} 
    + \bignorm{\sqrt{\eta_n(\bphi_n)}\, \D \v_n}_{\L2{\LL^2}}
    +\norm{p_n}_{\L{4/3}{L^2}}
    \notag\\
    &\quad
    + \norm{\Psi_{\bphi}(\bphi_n)}_{ \L{4} {\LL^{2}}\cap\L{2} {\LL^{6}} }
    + \norm{\nabla \Ns(\bphi_n,\bsig_n)}_{ \L{2} {\LL^{2}}} 
    \notag\\
    & \quad
    +\norm{\Div(\bphi_n\otimes\v_n)}_{ \L{4/3}{\LL^{3/2}}}
    +\norm{\Div(\bsig_n\otimes\v_n)}_{ \L{1}{\LL^1}}
   \le C.
\end{align}
We still have to derive additional uniform estimates for the time derivatives of $\bphi_n$ and $\bsig_n$. 
To this end, we recall the identities
\begin{alignat}{2}
\label{DEC:DIV:PHI:N}
    \Div(\bphi_n\otimes\v_n) 
    &= ( \nabla \bphi_n ) \v_n+ \bphi_n \Sv (\bphi_n,\bsig_n)
    &&\quad\text{a.e. in $Q$}, \\
\label{DEC:DIV:SIG:N}
    \Div(\bsig_n\otimes\v_n) 
    &= ( \nabla \bsig_n ) \v_n+ \bphi_n \Sv (\bphi_n,\bsig_n)
    &&\quad\text{a.e. in $Q$}.
\end{alignat}
Let now $\bz \in \L{\frac 83} {\HH^1}$ be an arbitrary test function. Using the continuous embedding $\HH^{\frac 32}\emb \LL^3$ as well as Lemma~\ref{LEM:INT}, we obtain the estimate
\begin{align*}
    &\int_0^T\intO \Div(\bphi_n\otimes\v_n) \cdot \bz \dxt
    \le C \int_0^T 
        \big( \norm{ \Grad\bphi_n}_{\LL^3}
        \norm{\v_n}_{\LL^2}  
        + \norm{\bphi_n}_{\LL^2} \big)
        \norm{\bz}_{ {\HH^1}} \dt 
    \\&\quad
    \le C \int_0^T 
        \big( \norm{ \bphi_n}_{\HH^1 }^{3/4}
        \norm{ \bphi_n}_{\HH^3 }^{1/4} 
        \norm{\v_n}_{\LL^2}  
        + \norm{\bphi_n}_{\LL^2} \big)
        \norm{\bz}_{ {\HH^1}} \dt 
     \\&\quad
     \le C\Big( \norm{\bphi_n}_{\L\infty{\HH^1}}^{3/4}
        \norm{ \bphi_n}_{\L2{\HH^3}}^{1/4} 
        \norm{\v_n}_{\L{2}{\LL^2}}  
        + \norm{\bphi_n}_{\L{\infty}{\LL^2}} \Big)
        \norm{\bz}_{ \L{8/3}{\HH^1}}
     \\&\quad
     \le C \norm{\bz}_{ \L{8/3}{\HH^1}}.
\end{align*}
Taking the supremum over all $\bz \in \L{8/3} {\HH^1}$ with $\norm{\bz}_{\L{8/3} {\HH^1}} \le 1$, we thus conclude the uniform estimate
\begin{align}
    \label{EST:UNI:D2}
    \norm{\Div(\bphi_n\otimes\v_n)}_{\L{8/5}{(\HH^1)'}} \le C.
\end{align}
Now, by a comparison argument, we infer from \eqref{WF:2} (written for the functions with index $n$)  that
\begin{align}
    \label{EST:UNI:D3}
    \norm{\delt \bphi_n}_{\L{8/5} {\Vp}} 
    \leq C.
\end{align}

Since $\LL^1$ is continuously embedded in $(\WW^{1,4})'$, we infer from \eqref{EST:UNI:D1} that 
\begin{align}
    \label{EST:UNI:D3*} 
    \norm{\Div(\bsig_n\otimes\v_n)}_{\L{1}{(\WW^{1,4})'}} \le C.
\end{align}
By means of a comparison argument, we eventually conclude from \eqref{WF:4} (written for the functions with index $n$) that
\begin{align}
    \label{DAR:EST:7} 
    \norm{\delt \bsig_n}_{\L{1} {({\WW^{1,4}})'}} \leq C.
\end{align}
Combining \eqref{EST:UNI:D1} with \eqref{EST:UNI:D2}--\eqref{DAR:EST:7}, we eventually obtain the uniform estimate 
\begin{align}
    \label{DAR:FINALEST}
    &\norm{\bphi_n}_{\W{1,8/5}{(\HH^1)'}\cap \L\infty{\HH^1} \cap \L2{\HH^3}}
    +\norm{\bsig_n}_{\W{1,1}{(\WW^{1,4})'} \cap \L\infty {\LL^2} \cap \L2{\HH^1}}
    \notag\\
    &\quad
    +\norm{\bsig_n}_{\L4{\LL^2_\Gamma}}
    +\norm{\bmu_n}_{\L2{\HH^1}}
    +\norm{\v_n}_{\L{2}{\LL^2}}
    +\bignorm{\sqrt{\eta_n(\bphi_n)}\, \D \v_n}_{\L2{\LL^2}}
    \notag\\
    &\quad
    + \norm{p_n}_{\L{4/3}{L^2}}
    + \norm{\Psi_{\bphi}(\bphi_n)}_{ \L{4} {\LL^{2}}\cap\L{2} {\LL^{6}} }
    + \norm{\nabla \Ns(\bphi_n,\bsig_n)}_{ \L{2} {\LL^{2}}} 
    \notag\\
    & \quad
    +\norm{\Div(\bphi_n\otimes\v_n)}_{\L{8/5}{(\HH^1)'}\cap  \L{4/3}{\LL^{3/2}}} 
    +\norm{\Div(\bsig_n\otimes\v_n)}_{ \L{1}{\LL^1} }
    \le C.
\end{align}

\noindent
\textbf{Step 2: Passing to the limit.}
The next step is to pass to the limit as $n\to \infty$.
From the uniform estimate \eqref{DAR:FINALEST}, we infer the existence of a quintuplet $(\bphi,\bsig,\bmu,\v,p)$ as well as limits $\ttau$ and $\ttheta$ such that for all $s\in[0,1)$,%
\begin{subequations}
\label{CD:WS}
\begin{alignat}{2} 
    \label{CD:PHI}
    \bphi_n & \to \bphi 
    &&\quad\text{weakly-$^*$ in $ L^\infty(0,T;\HH^1)$,} \notag\\
    &&&\qquad \text{weakly in $ \W{1,\frac 85} {\Vp} 
        \cap L^2(0,T;\HH^3)$, a.e.~in $Q$,}
    \notag\\ 
    &&&\qquad\text{and strongly in $C([0,T];\HH^s) \cap L^2 (0,T;\HH^{2+s})$,}
    \\
    \label{CD:PHI:G}
    \bphi_n\vert_\Gamma &\to \bphi\vert_\Gamma
    &&\quad\text{strongly in $C([0,T];\LL^2_\Gamma)$,
        and a.e.~on $\Sigma$,}
    \\ 
    \label{CD:SIGMA}
    \bsig_n & \to \bsig
    &&\quad\text{weakly-$^*$ in $\L\infty {\LL^2} $, }\notag\\
    &&&\qquad\text{weakly in $ \W{1,1} {({\WW^{1,4}})'} \cap \L2 {\HH^1}$, a.e.~in $Q$,}  
    \notag\\ 
    &&&\qquad\text{and strongly in $C([0,T];({\WW^{1,4}})') \cap L^2 (0,T;\HH^s) $,}
    \\
    \label{CD:SIGMA:G}
    \bsig_n\vert_\Gamma &\to \bsig\vert_\Gamma
    &&\quad\text{weakly in $L^4(0,T;\LL^2_\Gamma)$, 
        strongly in $\L2 {\LL^{2}_\Gamma}$,
        and a.e.~on $\Sigma$,}
    \\ 
    \label{CD:MU}
    \bmu_n & \to \bmu && \quad \text{weakly in $\L2 {\HH^1}$,}
    \\
    \label{CD:V}
    \v_n & \to \v && \quad \text{weakly in $ \L2 {\LL^2_{\Div}}$,} 
    \\
    \label{CD:P}
    p_n & \to p && \quad \text{weakly in $\L{\frac 43} {L^2}$,} 
    \\ 
    \label{CD:DIV:PHI}
    \Div(\bphi_n\otimes\v_n) & \to \ttau 
    && \quad\text{weakly in $\L{\frac 85} { (\HH^1)'} \cap \L{\frac 43} {\LL^{\frac 32}}$},
    \\ 
    \label{CD:DIV:SIG}
    \Div(\bsig_n\otimes\v_n) & \to  \ttheta 
    && \quad \text{weakly in $\L{1} {\LL^1}$},
    \\ 
    \eta_n(\bphi_n) & \to  0 
    && \quad \text{strongly in $L^\infty(Q)$, and a.e.~in $Q$},   
    \\ 
    \lambda_n(\bphi_n) & \to  0 
    &&\quad \text{strongly in $L^\infty(Q)$, and a.e.~in $Q$}, 
\end{alignat}
\end{subequations}
as $n\to\infty$, along a non-relabeled subsequence.
The strong convergences in \eqref{CD:PHI} and \eqref{CD:SIGMA} are a direct consequence of the Aubin--Lions--Simon lemma (see \cite[Theorem~II.5.16]{boyer}), and
the strong convergences in \eqref{CD:PHI:G} and \eqref{CD:SIGMA:G} are obtained through the trace theorem.
We further point out that the convergence $\v_n\to \v$ in $\L2{\LL^2_{\Div}}$ (see \eqref{CD:V}) already entails that
\begin{align}
    \label{CD:DIV:V}
    \Div(\v_n) \to \Div(\v) \quad\text{weakly in $\L2{L^2}$}
\end{align}
as $n\to\infty$. 
Recalling the assumptions \ref{ass:tensors}--\ref{ass:pot}, we use the above convergences to infer that
\begin{subequations}
\label{CD:AE}
\begin{alignat}{3}
    &\CC(\bphi_n,\bsig_n) \to \CC(\bphi,\bsig),
    &&\quad \DD(\bphi_n,\bsig_n) \to \DD(\bphi,\bsig)
    &&\quad\text{a.e.~in $Q$,}\\
    \label{CD:AE:3}
    &\Np(\bphi_n,\bsig_n) \to \Np(\bphi,\bsig),
    &&\quad \Ns(\bphi_n,\bsig_n) \to \Ns(\bphi,\bsig)
    &&\quad\text{a.e.~in $Q$,}\\
    \label{CD:AE:4}
    &\Grad\Ns(\bphi_n,\bsig_n) \to \Grad\Ns(\bphi,\bsig),
    &&\quad \Sv(\bphi_n,\bsig_n) \to \Sv(\bphi,\bsig)
    &&\quad\text{a.e.~in $Q$,}\\
    \label{CD:AE:5}
    &\Psi_{\bphi}(\bphi_n) \to \Psi_{\bphi}(\bphi)
    &&
    &&\quad\text{a.e.~in $Q$,}\\
    &\SG(\bphi_n,\bsig_n) \to \SG(\bphi,\bsig)
    &&&&\quad\text{a.e.~on $\Sigma$,}\\
    &\Lam_{\bphi}(\bphi_n,\bsig_n) \to \Lam_{\bphi}(\bphi,\bsig),
    &&\quad \Th_{\bphi}(\bphi_n,\bsig_n) \to \Th_{\bphi}(\bphi,\bsig)
    &&\quad\text{a.e.~in $Q$,}\\
    &\Lam_{\bsig}(\bphi_n,\bsig_n) \to \Lam_{\bsig}(\bphi,\bsig),
    &&\quad \Th_{\bsig}(\bphi_n,\bsig_n) \to \Th_{\bsig}(\bphi,\bsig)
    &&\quad\text{a.e.~in $Q$,}
\end{alignat}
\end{subequations}
as $n\to\infty$, after extracting a subsequence. Now, using \eqref{CD:AE:4}, \eqref{CD:AE:5}, the uniform estimate \eqref{DAR:FINALEST}, and the uniqueness of the limit, we further conclude that
\begin{alignat}{2}
\label{CD:PSI}
    &\Psi_{\bphi}(\bphi_n) \to \Psi_{\bphi}(\bphi)
    &&\quad\text{weakly in $L^4(0,T;\LL^2) \cap L^2(0,T;\LL^6)$,}\\
\label{CD:NS}
    &\Grad\Ns(\bphi_n,\bsig_n) \to \Grad\Ns(\bphi,\bsig)
    &&\quad\text{weakly in $L^2(0,T;\LL^2)$}
\end{alignat}
as $n\to\infty$, after another subsequence extraction.

Let now $\delta \in C^\infty_c(0,T)$, $\be \in H^1(\Omega;\R^d)$, $\bz,\bt \in H^1(\Omega;\R^L)$, $\bx \in W^{1,4}(\Omega;\R^M) \emb L^\infty(\Omega;\R^M)$, and $q \in H^1(\Omega)$ be arbitrary test functions.
We now test the weak formulation \eqref{WF} (written for $(\bphi_n, \bmu_n, \bsig_n, \v_n, p_n) $ and the viscosities $\eta_n$ and $\lambda_n$) with $\delta\be$, $\delta\bz$, $\delta\bt$ and $\delta\bx$, and we integrate the resulting equations with respect to time from $0$ to $T$. We further multiply the identity \eqref{SF:DIV} (written for $\v_n$, $\bphi_n$ and $\bsig_n$) by $\delta q$ and integrate the resulting equation over $Q$.
In summary, we obtain
\begin{subequations}
\label{WF:LIM}
\begin{align}
    \notag
	0= & \intQ  \big(2 \eta_n(\bphi_n) \D\v_n + \lambda_n(\bphi_n) \Div(\v_n) \I - p_n\I \big) : \delta \nabla \be + \nu  \v_n \cdot \delta\be \dxt
	\\ & \quad     \label{WF:LIM:1}
	-  \intQ  (\nabla \bphi_n)^\top  \bmu_n \cdot \delta\be + (\nabla \bsig_n)^\top \Ns(\bphi_n,\bsig_n) \cdot \delta\be  \dxt,
	\\
	\notag
	0= & \int_0^T \ang{\delt \bphi_n}{ \bz }_{\HH^1} \, \delta  \dt
	+ \intQ  \Div(\bphi_n\otimes\v_n) \cdot \delta\bz \dxt 
	\\ & \qquad \label{WF:LIM:2}
	 + \intQ  \CC(\bphi_n,\bsig_n) \nabla \bmu_n : \delta \nabla \bz - \S_{\bphi}(\bphi_n,\bsig_n,\bmu_n) \cdot \delta\bz 
	\dxt,
	\\ \notag
	0= & \intQ   -\bmu_n \cdot  \delta\bt \dx  
	+ \gamma\eps \nabla \bphi_n : \delta \nabla \bt + \gamma\eps^{-1} \Psi_{\bphi}(\bphi_n) \cdot  \delta\bt   \dxt
	\\ & \quad  \label{WF:LIM:3}
	+ \intQ   \N_{\bphi}(\bphi_n,\bsig_n) \cdot  \delta\bt \dxt,
	\\ \notag
	0= & \int_0^T \ang{ \delt \bsig_n}{\bx}_{\WW^{1,4}} \; \delta \dt
	+ \intQ \Div(\bsig_n\otimes\v_n) \cdot \delta\bx \dxt
	\\ & \quad \notag
    + \intQ \DD(\bphi_n,\bsig_n) \nabla \Ns(\bphi_n,\bsig_n): \delta \nabla \bx \dxt
	\\ & \quad \label{WF:LIM:4}	
	+ \intQ  \Ss(\bphi_n,\bsig_n,\bmu_n) \cdot \delta\bx \dxt
	- \intS  \SG(\bphi_n,\bsig_n)\cdot \delta\bx \dS \dt,
	\\ \label{WF:LIM:5}	
	0= & \intQ   \Div (\v_n)\, \delta q -  S_\v(\bphi_n,\bsig_n)\, \delta q  \dxt.
\end{align}
\end{subequations}
Our next goal is then to pass to the limit $n\to\infty$ in this variational formulation.
Invoking the convergence properties \eqref{CD:WS} and \eqref{CD:AE}, and using Lebesgue's dominated convergence theorem, 
\ref{ass:sources}--\ref{ass:source:boundary}, we deduce that%
\begin{subequations}
\label{CD:LCT}
\begin{alignat}{3}
    &\CC(\bphi_n,\bsig_n) \delta\Grad\bz 
    \to \CC(\bphi,\bsig) \delta\Grad\bz,
    &&\quad\DD(\bphi_n,\bsig_n) \delta\Grad\bx 
    \to \DD(\bphi,\bsig) \delta\Grad\bx 
    &&\quad\text{in $\LL^2(Q)$,}\\
    &\Th_{\bphi}(\bphi_n,\bsig_n)^\top \delta\bz 
    \to \Th_{\bphi}(\bphi,\bsig)^\top \delta\bz,
    &&\quad\Th_{\bsig}(\bphi_n,\bsig_n)^\top \delta\bx 
    \to \Th_{\bsig}(\bphi,\bsig)^\top\delta\bx 
    &&\quad\text{in $L^2(Q)$,}
\end{alignat}
\end{subequations}
as $n\to\infty$. Furthermore, proceeding as in Step~4 of the proof of Theorem~\ref{THM:EXISTENCE:WEAK}, we use Lebesgue's generalized convergence theorem \cite[Sec.~3.25]{Alt} to conclude that
\begin{subequations}
\label{CD:LGCT}
\begin{alignat}{2}
    [\Ns(\bphi_n,\bsig_n)]_j[\delta\be]_i
    &\to [\Ns(\bphi,\bsig)]_j[\delta\be]_i
    &&\quad\text{in $L^2(Q)$},\\
    \Np(\bphi_n,\bsig_n)\cdot\delta\bt
    &\to \Np(\bphi,\bsig)\cdot\delta\bt
    &&\quad\text{in $L^2(Q)$,}\\
    \Lam_{\bphi}(\bphi_n,\bsig_n) \cdot\delta\bz 
    &\to \Lam_{\bphi}(\bphi,\bsig) \cdot\delta\bz
    &&\quad\text{in $L^2(Q)$,}\\
    \Lam_{\bsig}(\bphi_n,\bsig_n) \cdot\delta\bx 
    &\to \Lam_{\bsig}(\bphi,\bsig) \cdot\delta\bx 
    &&\quad\text{in $L^2(Q)$,}\\
    \Lam_{\Gamma}(\bphi_n,\bsig_n) \cdot\delta\bx 
    &\to \Lam_{\Gamma}(\bphi,\bsig) \cdot\delta\bx
    &&\quad\text{in $L^2(\Sigma)$,}
\end{alignat}
\end{subequations}
as $n\to\infty$, for all $i\in\{1,...,d\}$ and $j\in\{1,...,M\}$.
For most of the terms in \eqref{WF:LIM}, we can simply use the convergences \eqref{CD:WS}, \eqref{CD:PSI}, \eqref{CD:NS}, \eqref{CD:LCT} and \eqref{CD:LGCT} to pass to the limit $n\to\infty$. However, some of the terms require a closer investigation.

In the terms depending on the viscosity functions $\eta_n$ and $\lambda_n$, we can pass to the limit $n\to\infty$ as follows:
\begin{align}
    \label{CD:ETA}
     & \intQ 2 \eta_n(\bphi_n) \D\v_n : \delta \nabla \be \dxt
     \notag \\ & \quad 
     \leq C \, \bignorm{ \sqrt{\eta_n (\bphi_n)}\D \v_n}_{\L2 {\LL^2}}
     \norm{ \eta_n (\bphi_n)}_{L^\infty(Q)}^{\frac 12}
     \norm{\delta}_{L^\infty([0,T])} \norm{\be}_{\HH^1} 
     \notag \\ & \quad 
     \leq C \, 
     \norm{ \eta_n (\bphi_n)}_{L^\infty(Q)}^{\frac 12}
     \norm{\delta}_{L^\infty([0,T])} \norm{\be}_{\HH^1} 
     \to 0, \\[1ex]
     \label{CD:LAMBDA}
     & \intQ   \delta \lambda_n(\bphi_n) \Div(\v_n)\I: \nabla \be  \dxt
     \notag \\ & \quad 
     \leq C  
     \norm{ {\lambda_n (\bphi_n)}}_{L^\infty(Q)}
     \norm{ \Sv(\bphi_n,\bsig_n)}_{\L2 {L^2}}
     \norm{\delta}_{L^\infty([0,T])} \norm{\be}_{\HH^1} 
     \to 0.
\end{align}

Furthermore, we still need to recover the identities $\ttau = \Div(\bphi\otimes\v)$ and $\ttheta = \Div(\bsig\otimes\v)$ almost everywhere in $Q$.  
To prove the latter identity, we first deduce from \eqref{CD:DIV:SIG} that
\begin{align}
\label{CD:DIV:SIGV}
    \intQ \Div(\bsig_n\otimes\v_n) \cdot \delta \bx_* \dxt 
    \to
    \intQ \ttheta \cdot \delta \bx_* \dxt
    \quad\text{for all $\bx_* \in \WW^{1,4}$.}
\end{align}
Let us now consider an arbitrary test function $\bx_0\in C^\infty_c(\Omega; \R^M)$.
Performing an integration by parts, we obtain
\begin{align*}
   \intQ \Div(\bsig_n\otimes\v_n) \cdot \delta \bx_0  \dxt 
   = -  \intQ (\bsig_n\otimes\v_n) : \delta \nabla \bx_0 \dxt.
\end{align*}
Due to \eqref{CD:SIGMA} and \eqref{CD:V}, we may pass to the limit on the right-hand side by the weak-strong convergence principle. After another integration by parts, we get
\begin{align*}
   \intQ \Div(\bsig_n\otimes\v_n) \cdot \delta \bx_0  \dxt 
   &\to  -  \intQ (\bsig \otimes\v ) : \delta \nabla \bx_0 \dxt \\
   & = \intQ \Div(\bsig\otimes\v) \cdot \delta \bx_0  \dxt
\end{align*}
as $n\to\infty$, for all $\bx_0\in C^\infty_c(\Omega; \R^M)$. Since \eqref{CD:DIV:SIGV} holds true for all $\bx_* = \bx_0\in C^\infty_c(\Omega; \R^M)$, we eventually have
\begin{align*}
    \intQ \ttheta \cdot \delta\bx_0  \dxt
    = \intQ \Div(\bsig\otimes\v) \cdot \delta \bx_0  \dxt
\end{align*}
for all $\bx_0\in C^\infty_c(\Omega; \R^M),\, \delta\in C^\infty([0,T])$, which is enough to conclude that $\ttheta = \Div(\bsig\otimes\v)$ almost everywhere in $Q$. In particular, this proves that
\begin{align}
\label{CD:DIV:SIGV*}
    \intQ \Div(\bsig_n\otimes\v_n) \cdot \delta \bx \dxt 
    \to
    \intQ \Div(\bsig\otimes\v) \cdot \delta \bx \dxt
\end{align}
as $n\to\infty$.
Proceeding similarly with the convection term associated with the phase-field variable, we conclude that $\ttau = \Div(\bphi\otimes\v)$ almost everywhere in $Q$, and
\begin{align}
\label{CD:DIV:PHIV*}
    \intQ \Div(\bphi_n\otimes\v_n) \cdot \delta \bz \dxt 
    \to
    \intQ \Div(\bphi\otimes\v) \cdot \delta \bz \dxt
\end{align}
as $n\to \infty$.
We can now use the convergences \eqref{CD:WS}--\eqref{CD:NS}, \eqref{CD:LCT}--\eqref{CD:LAMBDA}, \eqref{CD:DIV:SIGV*}, and \eqref{CD:DIV:PHIV*},
along with the identities $\ttau = \Div(\bphi\otimes\v)$ and $\ttheta = \Div(\bsig\otimes\v)$ a.e.~in $Q$, to pass to the limit $n\to\infty$ in the variational formulation \eqref{WF:LIM}. Since $\delta\in C^\infty([0,T])$ was arbitrary, this proves that the quintuplet $(\bphi,\bmu,\bsig,\v,p)$ satisfies the equations 
\begin{subequations}
\begin{align}
    \label{DAR:WF:LIM:1}
	0= & \intO  - p\, \Div (\be) + \big(\nu  \v - (\nabla \bphi)^\top \bmu 
	- (\nabla \bsig)^\top \Ns(\bphi,\bsig) \big) \cdot \be\dx,
	\\
    \notag
	0= & \< \delt \bphi , \bz >_{\HH^1} 
	+ \intO \big(   (\nabla \bphi) \v \cdot \bz + \bphi \Sv(\bphi,\bsig)\cdot  \bz \big)\dx 
	\\ & \qquad \label{DAR:WF:LIM:2}
	 +\intO\CC(\bphi,\bsig) \nabla \bmu : \nabla \bz - \Sp(\bphi,\bsig,\bmu) \cdot \bz \dx,
	\\ \label{DAR:WF:LIM:3}
	0= & \intO  -\bmu \cdot \bt \dx  + \gamma\eps \nabla \bphi : \nabla \bt + \gamma\eps^{-1} \Psi_{\bphi}(\bphi) \cdot \bt + \Np(\bphi,\bsig) \cdot \bt\dx,
	\\     \notag
	0= & \< \delt \bsig , \bx >_{\WW^{1,4}}
	+ \intO   (\nabla \bsig) \v \cdot \bx - \bsig \Sv (\bphi,\bsig)\cdot \bx 
	+ \DD(\bphi,\bsig) \nabla \Ns(\bphi,\bsig): \Grad \bx \dx
	\\ & \quad \label{DAR:WF:LIM:4}
	+ \intO \Ss(\bphi,\bsig,\bmu) \cdot \bx \dx -\intG \SG(\bphi,\bsig)\cdot \bx \dS
\end{align}
almost everywhere in $(0,T)$, for all test functions $\be \in H^1(\Omega;\R^d)$, $\bz,\bt \in H^1(\Omega;\R^L)$, $\bx \in {W^{1,4}}(\Omega;\R^M)$, 
as well as the identity
\begin{align}
    \label{DAR:WF:LIM:5}
    \Div (\v) =  \Sv(\bphi,\bsig)  \quad \text{a.e. in $Q$.}
\end{align}
\end{subequations}
Testing \eqref{DAR:WF:LIM:1} with any function $\be_0\in C^\infty_c(\Omega;\R^d)$, we deduce that
\begin{align*}
    \intO p\, \Div (\be_0) \dx
    = \intO \big(\nu  \v - (\nabla \bphi)^\top \bmu 
	- (\nabla \bsig)^\top \Ns(\bphi,\bsig) \big) 
        \cdot \be_0 \dx.
\end{align*}
Since
\begin{align*}
    &\bignorm{\nu  \v - (\nabla \bphi)^\top \bmu 
	- (\nabla \bsig)^\top \Ns(\bphi,\bsig)}_{\L1{\LL^{3/2}}}
	\notag\\
	&\quad \le C\norm{\v}_{\L2{\LL^{2}}}
	+ \norm{\Grad\bphi}_{\L2{\LL^2}} \norm{\bmu}_{\L2{\LL^6}}
	\notag\\
	&\qquad + C \norm{\Grad\bsig}_{\L2{\LL^2}} 
	\big(\norm{\bphi}_{\L2{\LL^6}} + \norm{\bsig}_{\L2{\LL^6}} +1 \big)
	\notag\\
	&\quad\le C,
\end{align*}
we conclude that $\Grad p$ exists in the weak sense with
\begin{align*}
    \Grad p = \nu  \v - (\nabla \bphi)^\top \bmu 
	- (\nabla \bsig)^\top \Ns(\bphi,\bsig) 
	\in \L {1} {\LL^{\frac 32}}
    \quad\text{a.e.~in $Q$.}
\end{align*}
Plugging this identity into \eqref{DAR:WF:LIM:1} and integrating the resulting expression by parts, we infer that
\begin{align}
    \label{ID:IBP}
    0 
    =  - \intO p(t)\, \Div (\be) + \Grad p(t) \cdot \be\dx 
    =  - \intG p(t) \be\cdot\n \dS
\end{align}
for almost all $t\in(0,T)$ and all $\be\in \HH^1$. 
For any $q\in C^1_b(\Gamma)$, we have $-q\n \in \HH^1$ and we may thus choose $\be = -q\n$. We thus obtain
\begin{align}
    0 
    =  \intG p(t)\, q \, \n\cdot\n \dS
    =  \intG p(t)\, q \dS
\end{align}
for all $q\in C^1_b(\Gamma)$ and almost all $t\in(0,T)$,
which directly proves that
 $p\vert_\Sigma=0$ a.e.~on $\Sigma$. In summary, we have
\begin{align}
    p \in \L{\frac 43}{L^2} \cap L^1\big(0,T;{W^{1,\frac 32}_0}\big),
\end{align}
and thus, all regularities in \eqref{REG:MCHD} are established. 
In particular, via integration by parts, \eqref{DAR:WF:LIM:1} can be replaced by the equivalent formulation
\begin{align}
    \label{DAR:WF:LIM:1*}
	0= & \intO  \Grad p \cdot \be + \big(\nu  \v - (\nabla \bphi)^\top  \bmu - (\nabla \bsig)^\top \Ns(\bphi,\bsig) \big) \cdot \be\dx.
\end{align}
We thus conclude that the quintuplet $(\bphi,\bmu,\bsig,\v,p)$ satisfies the weak formulation \eqref{WF:D}. 

As a further consequence of the convergences $\bphi_n\to\bphi$ in $C([0,T];\LL^2)$ from \eqref{CD:PHI} and $\bsig_n\to\bsig$ in $C([0,T];(\WW^{1,4})')$ from \eqref{CD:SIGMA} we have 
\begin{alignat*}{2}
    \bphi_0 
    &= \bphi_n\vert_{t=0} \to \bphi\vert_{t=0} 
    &&\quad \text{in $L^2(\Omega; \R^L),$}
    \\ 
	\< \bsig_0, \Phi>_{\HH^1} 
	&= \< \bsig_n\vert_{t=0}, \Phi>_{\WW^{1,4}} \to \< \bsig\vert_{t=0}, \Phi>_{\WW^{1,4}} 
	&&\quad \text{for all $\Phi \in{W^{1,4}}(\Omega;\R^M)$},
\end{alignat*}
as $n\to\infty$, meaning that $\bphi$ and $\bsig$ satisfy the initial conditions \eqref{INI:D:1} and \eqref{INI:D:2}.

This proves that the limit $(\bphi,\bmu,\bsig,\v,p)$ is a weak solution of the multiphase Cahn--Hilliard--Darcy system in the sense of Definition~\ref{DEF:WEAKSOL:DARCY}.
We further point out that the additional regularity property \eqref{REG:MCHD:ADD} can be verified by arguing exactly as in the proof of Theorem~\ref{THM:EXISTENCE:WEAK}.
Thus, the proof of Theorem~\ref{THM:EXISTENCE:WEAK:DARCY} is complete.
\end{proof}

\section*{Acknowledgment}
The authors want to thank Harald Garcke for helpful discussions. 
Patrik Knopf was partially supported by the RTG 2339 ``Interfaces, Complex Structures, and Singular Limits'' of the German Science Foundation (DFG).
Their support is gratefully acknowledged.
In addition, Andrea Signori wants to acknowledge the affiliation
to the GNAMPA (Gruppo Nazionale per l'Analisi Matematica, 
la Probabilit\`a e le loro Applicazioni) of INdAM (Isti\-tuto 
Nazionale di Alta Matematica).




\footnotesize

\bibliographystyle{plain}
\bibliography{KS}

\begin{thebibliography}{10}

\bibitem{AbelsTerasawa}
H.~Abels and Y.~Terasawa.
\newblock On {S}tokes operators with variable viscosity in bounded and
  unbounded domains.
\newblock {\em Math. Ann.}, 344(2):381--429, 2009.

\bibitem{AgostiEtAl}
A.~Agosti, C.~Cattaneo, C.~Giverso, D.~Ambrosi, and P.~Ciarletta.
\newblock A computational framework for the personalized clinical treatment of
  glioblastoma multiforme.
\newblock {\em ZAMM - Journal of Applied Mathematics and Mechanics /
  Zeitschrift für Angewandte Mathematik und Mechanik}, 98(12):2307--2327,
  2018.

\bibitem{ACGH}
A.~Agosti, P.~Ciarletta, H.~Garcke, and M.~Hinze.
\newblock Learning patient-specific parameters for a diffuse interface
  glioblastoma model from neuroimaging data.
\newblock {\em Math. Methods Appl. Sci.}, 43(15):8945--8979, 2020.

\bibitem{Alt}
H.W. Alt.
\newblock {\em {Linear Functional Analysis - An Application-Oriented
  Introduction}}.
\newblock Springer, London, 2016.

\bibitem{amann}
H.~Amann.
\newblock {\em {Linear and Quasilinear Parabolic Problems, Volume I: Abstract
  Linear Theory}}.
\newblock Birkhäuser Basel, 1995.

\bibitem{araujo}
R.P. Araujo and D.L.S. McElwain.
\newblock A history of the study of solid tumour growth: the contribution of
  mathematical modelling.
\newblock {\em Bulletin of mathematical biology}, 66(5):1039--1091, 2004.

\bibitem{AstaninPreziosi}
S.~Astanin and L.~Preziosi.
\newblock Multiphase models of tumour growth.
\newblock In {\em Selected topics in cancer modeling}, Model. Simul. Sci. Eng.
  Technol., pages 223--253. Birkh\"{a}user Boston, Boston, MA, 2008.

\bibitem{BearerEtAl}
E.L. Bearer, J.S. Lowengrub, H.B. Frieboes, Y.L. Chuang, F.~Jin, S.M. Wise,
  M.~Ferrari, and V.~Agus, D.B.and~Cristini.
\newblock Multiparameter computational modeling of tumor invasion.
\newblock {\em Cancer Research}, 69(10):4493--4501, 2009.

\bibitem{bosiaconti}
S.~Bosia, M.~Conti, and M.~Grasselli.
\newblock On the {C}ahn--{H}illiard--{B}rinkman system.
\newblock {\em Comm. Math. Sci.}, 13(6):1541--1567, 2015.

\bibitem{boyer}
F.~Boyer and P.~Fabrie.
\newblock {\em Mathematical tools for the study of the incompressible
  {N}avier-{S}tokes equations and related models}, volume 183 of {\em Applied
  Mathematical Sciences}.
\newblock Springer, New York, 2013.

\bibitem{Byrne}
H.~{Byrne} and L.~{Preziosi}.
\newblock Modelling solid tumour growth using the theory of mixtures.
\newblock {\em Math. Med. Biol.}, 20(4):341--366, 2003.

\bibitem{ByrneChaplain}
H.M. Byrne and M.A.J. Chaplain.
\newblock Free boundary value problems associated with the growth and
  development of multicellular spheroids.
\newblock {\em European J. Appl. Math.}, 8(6):639--658, 1997.

\bibitem{ByrneKing}
H.M. Byrne, J.R. King, D.L.S. McElwain, and L.~Preziosi.
\newblock A two-phase model of solid tumour growth.
\newblock {\em Appl. Math. Lett.}, 16(4):567--573, 2003.

\bibitem{Cavaterra}
C.~Cavaterra, E.~Rocca, and H.~Wu.
\newblock {Long-Time Dynamics and Optimal Control of a Diffuse Interface Model
  for Tumor Growth}.
\newblock {\em Appl. Math. Optim.}, 83:739--787, 2021.

\bibitem{Chaplain}
M.A.J. Chaplain.
\newblock Avascular growth, angiogenesis and vascular growth in solid tumours:
  The mathematical modelling of the stages of tumour development.
\newblock {\em Mathematical and Computer Modelling}, 23(6):47--87, 1996.

\bibitem{ColliGilardiHilhorst}
P.~Colli, G.~Gilardi, and D.~Hilhorst.
\newblock On a {C}ahn--{H}illiard type phase field system related to tumor
  growth.
\newblock {\em Discrete Contin. Dyn. Syst.}, 35(6):2423--2442, 2015.

\bibitem{col-gil-roc-spr}
P.~Colli, G.~Gilardi, E.~Rocca, and J.~Sprekels.
\newblock Vanishing viscosities and error estimate for a {C}ahn--{H}illiard
  type phase field system related to tumor growth.
\newblock {\em Nonlinear Anal. Real World Appl.}, 26:93--108, 2015.

\bibitem{col-gil-roc-spr2}
P.~Colli, G.~Gilardi, E.~Rocca, and J.~Sprekels.
\newblock Asymptotic analyses and error estimates for a {C}ahn--{H}illiard type
  phase field system modelling tumor growth.
\newblock {\em Discrete Contin. Dyn. Syst. Ser. S}, 10(1):37--54, 2017.

\bibitem{ColliGilardiRoccaSprekels}
P.~Colli, G.~Gilardi, E.~Rocca, and J.~Sprekels.
\newblock Optimal distributed control of a diffuse interface model of tumor
  growth.
\newblock {\em Nonlinearity}, 30(6):2518--2546, 2017.

\bibitem{CSS1}
P.~Colli, A.~Signori, and J.~Sprekels.
\newblock Optimal control of a phase field system modelling tumor growth with
  chemotaxis and singular potentials.
\newblock {\em Appl. Math. Optim.}, 82:517--549, 2020.

\bibitem{CSS2}
P.~Colli, A.~Signori, and J.~Sprekels.
\newblock Second-order analysis of an optimal control problem in a phase field
  tumor growth model with singular potentials and chemotaxis.
\newblock {\em ESAIM Control Optim. Calc. Var.}, 27, 2021.
\newblock \url{https://doi.org/10.1051/cocv/2021072}.

\bibitem{ConstantinFoias}
P.~Constantin and C.~Foias.
\newblock {\em Navier-{S}tokes equations}.
\newblock Chicago Lectures in Mathematics. University of Chicago Press,
  Chicago, IL, 1988.

\bibitem{contigiorg}
M.~Conti and A.~Giorgini.
\newblock Well-posedness for the {B}rinkman--{C}ahn--{H}illiard system with
  unmatched viscosities.
\newblock {\em Journal of Differential Equations}, 268(10):6350--6384, 2020.

\bibitem{DFRSS}
M.~Dai, E.~Feireisl, E.~Rocca, G.~Schimperna, and M.E Schonbek.
\newblock Analysis of a diffuse interface model of multispecies tumor growth.
\newblock {\em Nonlinearity}, 30(4):1639, 2017.

\bibitem{Dede}
L.~Ded\`e, H.~Garcke, and K.F. Lam.
\newblock A {H}ele-{S}haw-{C}ahn--{H}illiard model for incompressible two-phase
  flows with different densities.
\newblock {\em J. Math. Fluid Mech.}, 20(2):531--567, 2018.

\bibitem{Dharmatti}
S.~Dharmatti and L.N.M. Perisetti.
\newblock Nonlocal {C}ahn--{H}illiard--{B}rinkman {S}ystem with {R}egular
  {P}otential: {R}egularity and {O}ptimal {C}ontrol.
\newblock {\em J. Dyn. Control Syst.}, 27(2):221--246, 2021.

\bibitem{Donatelli}
D.~Donatelli and K.~Trivisa.
\newblock On a nonlinear model for tumour growth with drug application.
\newblock {\em Nonlinearity}, 28(5):1463--1481, 2015.

\bibitem{EbenbeckGarcke}
M.~Ebenbeck and H.~Garcke.
\newblock {Analysis of a {C}ahn--{H}illiard--{B}rinkman model for tumour growth
  with chemotaxis}.
\newblock {\em J.~Differential~Equations}, 266(9):5998--6036, 2019.

\bibitem{EbenbeckGarcke2}
M.~Ebenbeck and H.~Garcke.
\newblock On a {C}ahn--{H}illiard--{B}rinkman model for tumor growth and its
  singular limits.
\newblock {\em SIAM J. Math. Anal.}, 51(3):1868--1912, 2019.

\bibitem{EbenbeckGarcke3}
M.~Ebenbeck, H.~Garcke, and R.~Nürnberg.
\newblock {C}ahn--{H}illiard--{B}rinkman systems for tumour growth.
\newblock {\em Discrete Contin. Dyn. Syst. Ser. S}, 2021.
\newblock \url{https://doi.org/10.3934/dcdss.2021034}.

\bibitem{eben-knopf1}
M.~Ebenbeck and P.~Knopf.
\newblock Optimal medication for tumors modeled by a
  {C}ahn--{H}illiard--{B}rinkman equation.
\newblock {\em Calculus of Variations and Partial Differential Equations},
  58(4):31~pp., 2019.

\bibitem{eben-knopf2}
M.~Ebenbeck and P.~Knopf.
\newblock Optimal control theory and advanced optimality conditions for a
  diffuse interface model of tumor growth.
\newblock {\em ESAIM Control Optim. Calc. Var.}, 26:Paper No. 71, 38, 2020.

\bibitem{EbenbeckLam}
M.~Ebenbeck and K.F. Lam.
\newblock Weak and stationary solutions to a {C}ahn--{H}illiard--{B}rinkman
  model with singular potentials and source terms.
\newblock {\em Adv. Nonlinear Anal.}, 10(1):24--65, 2021.

\bibitem{Matioc}
J.~Escher, A.-V. Matioc, and B.-V. Matioc.
\newblock Analysis of a mathematical model describing necrotic tumor growth.
\newblock In {\em Modelling, simulation and software concepts for
  scientific-technological problems}, volume~57 of {\em Lect. Notes Appl.
  Comput. Mech.}, pages 237--250. Springer, Berlin, 2011.

\bibitem{Feng}
X.~Feng and S.~Wise.
\newblock Analysis of a {D}arcy--{C}ahn--{H}illiard diffuse interface model for
  the {H}ele--{S}haw flow and its fully discrete finite element approximation.
\newblock {\em SIAM J. Numer. Anal.}, 50(3):1320--1343, 2012.

\bibitem{FranksKing2}
S.J. Franks and J.R. King.
\newblock {Interactions between a uniformly proliferating tumour and its
  surroundings: uniform material properties}.
\newblock {\em Math. Med. Biol.}, 20(1):47--89, 2003.

\bibitem{FranksKing}
S.J. Franks and J.R. King.
\newblock Interactions between a uniformly proliferating tumour and its
  surroundings: stability analysis for variable material properties.
\newblock {\em Internat. J. Engrg. Sci.}, 47(11-12):1182--1192, 2009.

\bibitem{Frieb}
H.B. Frieboes, F.~Jin, Y.-L. Chuang, S.M. Wise, J.S. Lowengrub, and Vi.
  Cristini.
\newblock Three-dimensional multispecies nonlinear tumor growth-{II}: tumor
  invasion and angiogenesis.
\newblock {\em Journal of theoretical biology}, 264(4):1254--1278, 2010.

\bibitem{FrieboesEtAl}
H.B. Frieboes, J.~Lowengrub, S.~Wise, X.~Zheng, P.~Macklin, E.~Bearer, and
  V.~Cristini.
\newblock Computer simulation of glioma growth and morphology.
\newblock {\em NeuroImage}, 37 Suppl 1:59--70, 02 2007.

\bibitem{Friedman}
A.~Friedman.
\newblock Free boundary problems associated with multiscale tumor models.
\newblock {\em Math. Model. Nat. Phenom.}, 4(3):134--155, 2009.

\bibitem{Friedman2}
A.~Friedman.
\newblock Free boundary problems for systems of {S}tokes equations.
\newblock {\em Discrete Contin. Dyn. Syst. Ser. B}, 21(5):1455--1468, 2016.

\bibitem{FrigeriGrasselliRocca}
S.~Frigeri, M.~Grasselli, and E.~Rocca.
\newblock On a diffuse interface model of tumour growth.
\newblock {\em European J. Appl. Math.}, 26(2):215--243, 2015.

\bibitem{frig-lam-roc}
S.~Frigeri, K.F. Lam, and E.~Rocca.
\newblock On a diffuse interface model for tumour growth with non-local
  interactions and degenerate mobilities.
\newblock In {\em Solvability, regularity, and optimal control of boundary
  value problems for {PDE}s}, volume~22 of {\em Springer INdAM Ser.}, pages
  217--254. Springer, Cham, 2017.

\bibitem{FLRS}
S.~Frigeri, K.F. Lam, E.~Rocca, and G.~Schimperna.
\newblock On a multi-species {C}ahn--{H}illiard--{D}arcy tumor growth model
  with singular potentials.
\newblock {\em Comm. in Math. Sci.}, 16(3):821--856, 2018.

\bibitem{FLS}
S.~Frigeri, K.F. Lam, and A.~Signori.
\newblock Strong well-posedness and inverse identification problem of a
  non-local phase field tumor model with degenerate mobilities.
\newblock {\em European J. Appl. Math.}, pages 1--42, 2021.

\bibitem{fritz}
M.~Fritz, P.K. Jha, T.~Köppl, J.T. Oden, and B.~Wohlmuth.
\newblock Analysis of a new multispecies tumor growth model coupling 3d
  phase-fields with a 1d vascular network.
\newblock {\em Nonlinear Analysis: Real World Applications}, 61:103331, 2021.

\bibitem{Galdi}
G.P. Galdi.
\newblock {\em An introduction to the mathematical theory of the
  {N}avier-{S}tokes equations}.
\newblock Springer Monographs in Mathematics. Springer, New York, second
  edition, 2011.
\newblock Steady-state problems.

\bibitem{GarckeLam1}
H.~Garcke and K.F. Lam.
\newblock Global weak solutions and asymptotic limits of a
  {C}ahn--{H}illiard--{D}arcy system modelling tumour growth.
\newblock {\em AIMS Mathematics}, 1(Math-01-00318):318--360, 2016.

\bibitem{GarckeLam2}
H.~Garcke and K.F. Lam.
\newblock Analysis of a {C}ahn--{H}illiard system with non-zero {D}irichlet
  conditions modeling tumor growth with chemotaxis.
\newblock {\em Discrete Contin. Dyn. Syst.}, 37(8):42--77, 2017.

\bibitem{GarckeLam3}
H.~Garcke and K.F. Lam.
\newblock {Well-posedness of a Cahn--Hilliard system modelling tumour growth
  with chemotaxis and active transport}.
\newblock {\em European J. Appl. Math.}, 28(2):284--316, 2017.

\bibitem{GarckeLam4}
H.~Garcke and K.F. Lam.
\newblock On a {C}ahn--{H}illiard--{D}arcy system for tumour growth with
  solution dependent source terms.
\newblock In {\em Trends in applications of mathematics to mechanics},
  volume~27 of {\em Springer INdAM Ser.}, pages 243--264. Springer, Cham, 2018.

\bibitem{GarckeLamNuernbergSitka}
H.~Garcke, K.F. Lam, R.~N\"urnberg, and E.~Sitka.
\newblock {A multiphase {C}ahn--{H}illiard--{D}arcy model for tumour growth
  with necrosis}.
\newblock {\em Math. Models Methods Appl. Sci.}, 28(3):525--577, 2018.

\bibitem{GarckeLamRocca}
H.~Garcke, K.F. Lam, and E.~Rocca.
\newblock Optimal control of treatment time in a diffuse interface model of
  tumor growth.
\newblock {\em Appl. Math. Optim.}, 78(3):495--544, 2018.

\bibitem{garcke-lam-signori}
H.~Garcke, K.F. Lam, and A.~Signori.
\newblock On a phase field model of {C}ahn--{H}illiard type for tumour growth
  with mechanical effects.
\newblock {\em Nonlinear Analysis: Real World Applications}, 57:103192, 2021.

\bibitem{GLS_OPT}
H.~Garcke, K.F. Lam, and A.~Signori.
\newblock Sparse optimal control of a phase field tumor model with mechanical
  effects.
\newblock {\em SIAM J. Control Optim.}, 59(2):1555--1580, 2021.

\bibitem{GarckeLamSitkaStyles}
H.~Garcke, K.F. Lam, E.~Sitka, and V.~Styles.
\newblock A {C}ahn--{H}illiard--{D}arcy model for tumour growth with chemotaxis
  and active transport.
\newblock {\em Math. Models Methods Appl. Sci.}, 26(6):1095--1148, 2016.

\bibitem{Giorgini}
A.~Giorgini, M.~Grasselli, and H.~Wu.
\newblock The {C}ahn--{H}illiard--{H}ele--{S}haw system with singular
  potential.
\newblock {\em Ann. Inst. H. Poincar\'{e} Anal. Non Lin\'{e}aire},
  35(4):1079--1118, 2018.

\bibitem{Greenspan}
H.P. Greenspan.
\newblock On the growth and stability of cell cultures and solid tumors.
\newblock {\em J. Theoret. Biol.}, 56(1):229--242, 1976.

\bibitem{HilhorstKampmannNguyenZee}
D.~Hilhorst, J.~Kampmann, T.N. Nguyen, and K.G. Van Der~Zee.
\newblock Formal asymptotic limit of a diffuse-interface tumor-growth model.
\newblock {\em Math. Models Methods Appl. Sci.}, 25(6):1011--1043, 2015.

\bibitem{KahleLam}
C.~Kahle and K.F. Lam.
\newblock {Parameter Identification via Optimal Control for a
  Cahn--Hilliard-Chemotaxis System with a Variable Mobility}.
\newblock {\em Appl. Math. Optim.}, 82(63--104), 2018.

\bibitem{OdenTinsleyHawkins}
J.~T. Oden, A.~Hawkins, and S.~Prudhomme.
\newblock General diffuse-interface theories and an approach to predictive
  tumor growth modeling.
\newblock {\em Math. Models Methods Appl. Sci.}, 20(3):477--517, 2010.

\bibitem{Perthame}
B.~Perthame and N.~Vauchelet.
\newblock Incompressible limit of a mechanical model of tumour growth with
  viscosity.
\newblock {\em Philos. Trans. Roy. Soc. A}, 373(2050):20140283, 16, 2015.

\bibitem{RSS}
E.~Rocca, L.~Scarpa, and A.~Signori.
\newblock Parameter identification for nonlocal phase field models for tumor
  growth via optimal control and asymptotic analysis.
\newblock {\em To appear in: Math. Models Methods Appl. Sci.}, 2020.
\newblock \href{https://arxiv.org/abs/2009.11159}{arXiv:2009.11159 [math.AP]}.

\bibitem{Roose}
T.~Roose, S.J. Chapman, and P.K. Maini.
\newblock Mathematical models of avascular tumor growth.
\newblock {\em SIAM Rev.}, 49(2):179--208, 2007.

\bibitem{SS}
L.~Scarpa and A.~Signori.
\newblock On a class of non-local phase-field models for tumor growth with
  possibly singular potentials, chemotaxis, and active transport.
\newblock {\em Nonlinearity}, 34:3199--3250, 2021.

\bibitem{sciume}
G.~Scium{\`e}, S.~Shelton, W.G. Gray, C.T. Miller, F.~Hussain, M.~Ferrari,
  P.~Decuzzi, and B.A. Schrefler.
\newblock A multiphase model for three-dimensional tumor growth.
\newblock {\em New journal of physics}, 15(1):015005, 2013.

\bibitem{Sheratt}
J.A. Sherratt and M.A.J. Chaplain.
\newblock A new mathematical model for avascular tumour growth.
\newblock {\em J. Math. Biol.}, 43(4):291--312, 2001.

\bibitem{signori1}
A.~Signori.
\newblock Optimal distributed control of an extended model of tumor growth with
  logarithmic potential.
\newblock {\em Appl. Math. Optim.}, 82:517--549, 2020.

\bibitem{signori3}
A.~Signori.
\newblock Optimal treatment for a phase field system of {C}ahn--{H}illiard type
  modeling tumor growth by asymptotic scheme.
\newblock {\em Math. Control Relat. Fields}, 10:305--331, 2020.

\bibitem{signori2}
A.~Signori.
\newblock Optimality conditions for an extended tumor growth model with double
  obstacle potential via deep quench approach.
\newblock {\em Evol. Equ. Control Theory}, 9:193--217, 2020.

\bibitem{signori5}
A.~Signori.
\newblock Penalisation of long treatment time and optimal control of a tumour
  growth model of {C}ahn--{H}illiard type with singular potential.
\newblock {\em Discrete Contin. Dyn. Syst. Ser. A}, 41(6):2519--2542, 2020.

\bibitem{signori4}
A.~Signori.
\newblock Vanishing parameter for an optimal control problem modeling tumor
  growth.
\newblock {\em Asymptot. Anal.}, 117:43--66, 2020.

\bibitem{ST}
J.~Sprekels and F.~Tr{\"o}ltzsch.
\newblock Sparse optimal control of a phase field system with singular
  potentials arising in the modeling of tumor growth.
\newblock {\em ESAIM: Control, Optimisation and Calculus of Variations},
  27:S26, 2021.

\bibitem{Srinivasan}
S.~Srinivasan and K.R. Rajagopal.
\newblock A thermodynamic basis for the derivation of the {D}arcy,
  {F}orchheimer and {B}rinkman models for flows through porous media and their
  generalizations.
\newblock {\em Internat. J. Non-Linear Mech.}, 58:162--166, 2014.

\bibitem{triebel}
H.~Triebel.
\newblock {\em {Interpolation Theory, Function Spaces, Differential
  Operators}}.
\newblock North-Holland Publishing Company, Amsterdam--New York, 1978.

\bibitem{Wallace}
D.I. Wallace and X.~Guo.
\newblock {Properties of tumor spheroid growth exhibited by simple mathematical
  models}.
\newblock {\em Front Oncol.}, 3(51), 2013.

\bibitem{wise2008three}
S.M. Wise, J.S. Lowengrub, H.B. Frieboes, and V.~Cristini.
\newblock Three-dimensional multispecies nonlinear tumor growth-{I}: model and
  numerical method.
\newblock {\em Journal of theoretical biology}, 253(3):524--543, 2008.

\bibitem{Zhang-IB}
J.Z. Zhang, N.S. Bryce, R.~Siegele, E.A. Carter, D.~Paterson, M.D. de~Jonge,
  D.L. Howard, C.G. Ryan, and T.W. Hambley.
\newblock {The use of spectroscopic imaging and mapping techniques in the
  characterisation and study of DLD-1 cell spheroid tumour models}.
\newblock {\em Integr. Biol.}, 4(9):1072--1080, 2012.

\bibitem{Zheng}
X.~Zheng, S.M. Wise, and V.~Cristini.
\newblock {Nonlinear simulation of tumor necrosis, neovascularization and
  tissue invasion via an adaptive finite-element/level-set method}.
\newblock {\em Bull. Math. Biol.}, 67(2):211--259, 2005.

\end{thebibliography}
%

\end{document}